\documentclass[mnsc,nonblindrev]{informs3_hide}
\OneAndAHalfSpacedXI 

\usepackage[english]{babel}
\usepackage[autostyle, english = american]{csquotes}
\MakeOuterQuote{"}
\usepackage[colorlinks,linkcolor=blue, citecolor=blue]{hyperref}
\usepackage[normalem]{ulem}

\usepackage{cleveref}
\crefname{subsection}{section}{subsections}
\usepackage{eqnarray}
\usepackage{algorithm,algpseudocode}
\usepackage{amsmath, bbm, enumitem, xparse, bm, lipsum}
\usepackage{amssymb}

\newcommand{\eps}{\epsilon}
\newcommand{\veps}{\varepsilon}

\newcommand{\prim}{{\mathrm{Prime}}}
\newcommand{\dual}{{\mathrm{Dual}}}
\newcommand{\calA}{{\mathcal{A}}}

\newcommand{\calX}{{\mathcal{X}}}

\newcommand{\calF}{{\mathcal{F}}}
\newcommand{\calI}{{\mathcal{I}}}
\newcommand{\calJ}{{\mathcal{J}}}

\newcommand{\calH}{{\mathcal{H}}}

\newcommand{\calY}{{\mathcal{Y}}}
\newcommand{\order}{\mathcal{O}}
\newcommand{\otil}{\widetilde{\mathcal{O}}}
\newcommand{\E}{\mathbb{E}}
\newcommand{\R}{\mathbb{R}}
\newcommand{\wt}{\widetilde}

\newcommand{\Rad}{\mathrm{Rad}}

\newcommand{\VI}{V^{\mathrm{Prime}}}
\newcommand{\hatVII}{\wh{V}^{\mathrm{Prime}}_{\calI}}
\newcommand{\VD}{V^{\mathrm{Dual}}}

\newcommand{\bv}{\bm{v}}
\newcommand{\by}{\bm{y}}

\newcommand{\ba}{\bm{a}}
\newcommand{\bh}{\bm{h}}

\newcommand{\bmu}{\bm{\mu}}

\newcommand{\bx}{\bm{x}}

\newcommand{\wh}[1]{\widehat{#1}}

\usepackage[dvipsnames]{xcolor}
\newcommand{\Authornote}[3]{{\leavevmode\color{#2}\sf\small$<${#1: #3}$>$}}

\newcommand{\mznote}[1]{\Authornote{MZ}{cyan}{#1}}

\usepackage{xparse}

\NewDocumentEnvironment{myproof}{o}
{\IfNoValueTF{#1}{\paragraph{{Proof.} }} {\paragraph{{#1.} }} }
{\hfill$\Halmos$}

\usepackage{natbib,bbm,multirow,multicol}
 \bibpunct[, ]{(}{)}{,}{a}{}{,}
 \def\bibfont{\small}

\TheoremsNumberedThrough
\ECRepeatTheorems

\EquationsNumberedThrough    
\allowdisplaybreaks

\begin{document}

\RUNAUTHOR{Jiang, Zhang}

\RUNTITLE{Instance-dependent Guarantees for Bilinear Saddle Point Problem}

\TITLE{
\Large An LP-Based Approach for Bilinear Saddle Point Problem with Instance-dependent Guarantee and Noisy Feedback
}

\ARTICLEAUTHORS{%
\AUTHOR{$\text{Jiashuo Jiang}^{\dag}$, ~~$\text{Mengxiao Zhang}^{\ddag}$}

\AFF{$\dag$~ Department of Industrial Engineering and Decision Analytics, Hong Kong University of Science and Technology  \\
$\ddag~$Department of Business Analytics, Tippie College of Business, University of Iowa
}
}

\ABSTRACT{
In this work, we study the sample complexity of obtaining a Nash equilibrium (NE) estimate in two-player zero-sum matrix games with noisy feedback. Specifically, we propose a novel algorithm that repeatedly solves linear programs (LPs) to obtain an NE estimate with bias at most $\veps$ with a sample complexity of $\order\left(\frac{m_1 m_2}{\veps\min\{\delta^2,\sigma_0^2,\sigma^3\}} \log\frac{m_1 m_2}{\veps}\right)$
for general $m_1 \times m_2$ game matrices, where $\sigma$, $\sigma_0$, $\delta$ are some problem-dependent constants. To our knowledge, this is the first instance-dependent sample complexity bound for finding an NE estimate with $\veps$ bias in general-dimension matrix games with noisy feedback and potentially non-unique equilibria. Our algorithm builds on recent advances in online resource allocation and operates in two stages: (1) identifying the support set of an NE, and (2) computing the unique NE restricted to this support. Both stages rely on a careful analysis of LP solutions derived from noisy samples. 
}


\maketitle

\section{Introduction}\label{sec: intro}

Minimax optimization~\citep{v1928theorie} is a fundamental problem in machine learning and game theory, with applications including adversarial training~\citep{goodfellow2014explaining} and robust optimization~\citep{rahimian2019distributionally}. A canonical example of a minimax problem is the two-player zero-sum matrix game, also known as the bilinear saddle-point problem, defined as $\min_{\bx}\max_{\by}\bx^\top A\by$, where $A$ is the payoff matrix, and $(\bx,\by)$ are the strategies of the two players lying within the simplex. The goal is to compute a \emph{Nash Equilibrium (NE)}, a strategy pair in which neither player can improve their expected payoff by deviating from their strategy.  

When the payoff matrix $A$ is fully known, computing an NE reduces to solving a pair of primal and dual linear programs or an equivalent saddle-point formulation. This full-information setting is polynomial-time solvable and well understood: one can compute the game value and an equilibrium strategy profile by solving an LP. When the game is of large scale, iterative approaches become essential to ensure an efficient computation of an NE. A typical approach is via {regret minimization}, where both players employ no-regret learning algorithms to sequentially update their strategies. Specifically, \citet{freund1999adaptive} shows that if two players adopts algorithms such as Hedge or Online Gradient Descent (OGD), the average iterate converges to an $\order(1/\sqrt{T})$-approximate NE, where $T$ is the total number of iterates. Later, there is a surge of works~\citep{daskalakis2011near,rakhlin2013online,syrgkanis2015fast} showing that with an optimistic variant of these algorithms, the convergence rate can be improved to $\mathcal{O}(1/T)$. 

Compared to the full information setting, a much harder learning environment is when each player picks an action and only receives a \emph{noisy} observation of $A$ on the \emph{chosen action pair}. This indeed includes many real-world applications which we discuss as follows.



\paragraph{Application 1: Dueling bandit and preference learning}

Dueling bandits~\citep{yue2012k} are a class of online learning problems that naturally form a two-player zero-sum game, in which the expected utility matrix can be formed as $A\in[0,1]^{m_1\times m_2}$ with $m_1=m_2$. Each time, the learner selects two actions $i,j\in[m_1]$ to duel and observes a binary outcome indicating which action is preferred in that comparison. Specifically, the probability of action $i$ winning over action $j$ is represented by an unknown payoff matrix entry $A_{ij}$, and binary feedback is stochastic under this underlying preference.

Dueling bandits in fact appear in many diverse real-world applications. For instance, in clinical trials, two treatments can be compared by presenting subsets of patients with either treatment and observing binary preference feedback, efficiently determining the superior option with minimal patient exposure~\citep{sui2018advancements}. In recommendation systems, users presented with pairs of products or movies provide binary preference feedback, enabling quick convergence to preferred items with fewer comparative tests~\citep{zoghi2014relative}. Similarly, search engine rankings optimized via dueling comparisons can efficiently identify preferable result rankings from noisy user click data. Another important application is in online A/B testing~\citep{o2021matrix,quin2024b} with limited feedback, commonly used in e-commerce platforms. Instead of measuring absolute performance of individual variants, the platform compares two website layouts, pricing schemes, or advertisement creatives and receive only a binary signal. This dueling setup allows for more sample-efficient learning and avoids assumptions on cardinal rewards, making it suitable for high-traffic, real-time systems with noisy preference data.
A more recent and significant example of dueling bandits is Reinforcement Learning with Human Feedback (RLHF), widely used in fine-tuning large language models, where human evaluators indicate preferences between two model-generated outputs. These pairwise comparisons form the noisy binary feedback and guide the model to learn a policy that aligns with human judgment~\citep{ouyang2022training}.


\paragraph{Application 2: Market making in financial markets}
Financial market making provides another significant application for two-player zero-sum games under noisy and bandit feedback~\citep{vayanos2001strategic,varsos2023noisy}. Market making involves two opposing players: the market maker, who provides liquidity by quoting bid and ask prices, and traders who seek to exploit profitable trades by selecting optimal trade execution strategies. Each player's actions lead to uncertain outcomes influenced by stochastic price movements, liquidity variations, and unpredictable external events. The market maker aims to minimize losses and maintain profitability, while traders strive to maximize their gains through informed or strategic trading. Learning optimal strategies through limited, noisy observations is essential for achieving competitive advantages in these adversarial financial interactions.

\paragraph{Application 3: Pricing and bidding competition in revenue management} {Zero-sum matrix games are useful tools in revenue management for modeling competitive or adversarial decision-making situations. For example, suppose there are two hotels, located in the same area, competing for the same group of customers. Each hotel wants to set a room price from a set of pricing strategies in order to compete with the other hotel's pricing strategy. This is a classic case of strategic decision-making and can be modeled as a matrix game (see for example \citep{baum1995empirical, dev1989strategic}), where each entry of the game matrix represents the possible outcome based on the pricing choices of both hotels. The outcome can be stochastic due to uncertain customer demand, which introduces noise when we make observations over matrix entries. Such a matrix game formulation helps hotels optimize strategically over pricing in a competitive market.}


\paragraph{Application 4: Attacker-defender games in blockchain security}
Beyond dueling bandits, blockchain security provides another compelling domain for learning Nash equilibria under noisy, bandit feedback. Blockchain networks can often be modeled as two-player zero-sum games between attackers, aiming to disrupt or compromise the system, and defenders, striving to maintain the integrity and security of the network. Actions taken by attackers, such as launching double-spending attacks, mining pool attacks, or manipulating consensus mechanisms, result in probabilistic outcomes influenced by factors like network latency, node honesty, and random block validation processes. Defenders, conversely, select protective strategies such as dynamic difficulty adjustments, node validation schemes, or network partitioning techniques, which similarly yield uncertain results due to inherent system stochasticity and evolving adversarial capabilities. Learning effective strategies under such noisy bandit feedback, where only outcomes of specific attacks or defenses are observable, is crucial for robust blockchain operation and resilience against sophisticated adversaries.

Despite its wide applications, less is known about learning NE under this more challenging bandit and noisy feedback environment.
\citet{maiti2023instance} considers this problem with unique NE when $m_2=2$ and studies two different notions of approximate NE. Specifically, they show that when $m_1=2$, with $\wt{\Theta}(\min\{\frac{1}{\veps^2},\max\{\frac{1}{\Delta_{\min}^2},\frac{1}{\veps D}\}\})$ samples, an $\veps$-good NE can be computed with probability at least $1-\delta$ where $\Delta_{\min}$ and $D$ are some problem-dependent constants and $\veps$-good NE means that the game value corresponding to the obtained strategy is at most $\veps$ away from the minimax value. For a stronger notion of $\veps$-sub-optimality-gap NE, they show that a worse minimax sample complexity bound of $\wt{\Theta}(\min\{\frac{1}{\veps^2},\max\{\frac{1}{\Delta_{\min}^2},\frac{c}{\veps^2 D^2}\}\})$ where $c$ is another problem-dependent constant. \footnote{We use $\otil(\cdot)$ and $\widetilde{\Theta}(\cdot)$ to suppress all logarithmic dependencies. \citet{maiti2023instance} uses the notion of $\epsilon$-NE in their paper. We use the notion of $\veps$-sub-optimality-gap NE in order to match the definition used in this paper. See \Cref{sec:pre} for definitions.} When $m_1>2$, \citet{maiti2023instance} further proposes an algorithm finding the NE supported on $2$ actions for the row player, reducing the problem to the $2\times 2$ case. Recently, \citet{ito2025instance} shows that when the problem has a unique pure Nash Equilibrium (PSNE), if two players both apply Tsallis-INF~\citep{JMLR:v22:19-753}, the exact PSNE can be computed using $\order(c'\log(1/\eps))$ samples with probability at least $1-\eps$ for some problem-dependent constant $c'$.\footnote{Throughout the paper, $\veps$ denotes the bias (error) from the NE, while $\eps$ denotes the probability of a rare event.}

While \citet{maiti2023instance} has shown that $\wt{\Omega}(\min\{\frac{1}{\veps^2},\frac{c}{\veps^2D^2}\})$ instance-dependent sample complexity is unavoidable in order to achieve $\veps$-sub-optimality-gap NE in general, it remains unclear whether a better instance-dependent sample complexity can be achieved if we consider finding an estimate of NE with $\veps$ bias, which is a weaker notion of NE compared to what has been considered previously. Therefore, this paper aims to investigate the following question

\begin{center}
    \emph{How to achieve better instance-dependent sample complexity for computing an estimate of Nash Equilibrium with $\veps$ bias in two-player zero-sum matrix games with noisy and bandit feedback?}
\end{center}

\subsection{Main Results and Contributions}
Inspired by the recent advance of online resource allocation~(e.g. \cite{vera2021bayesian, li2022online, jiang2022degeneracy, ma2024optimal}), we design a practically efficient algorithm based on resolving linear programs (LP) achieving an estimate NE with $\veps$ bias with sample complexity $\order\left(\frac{m_1 m_2}{\veps\min\{\delta^2,\sigma_0^2,\sigma^3\}} \log\frac{m_1 m_2}{\veps}\right)$, where $\delta$, $\sigma$, $\sigma_0$ are problem-dependent constants.\footnote{All problem-dependent constants are formally defined in later sections.} 

Specifically, in \Cref{sec:support-identification}, we introduce our first component, which is to identify the support of an NE. Specifically, we first sample each entry of matrix $N/m$ times and construct an empirical LP formulation of the game. Our algorithm then iteratively computes the estimated support of the NE by restricting the strategy to progressively smaller supports. Concretely, at each step, we select an entry from the current support and compare the objective value of the empirical LP in two cases: one where the strategy remains unrestricted and one where the selected entry is removed. If the two objective values are sufficiently close, this indicates that the selected entry can be safely removed while still ensuring the existence of an NE. Otherwise, we terminate the process and output the current support. With a high probability, this guarantees to find the support of an NE when $N$ exceeds certain constant related to some instance-dependent constant $\delta>0$ defined in \Cref{sec:instance-dependent-id}. {We also develop the corresponding suboptimality gap of the identified support given a finite sample size $N$, which will be useful to derive a $\delta$-independent guarantee for our algorithm.} 

After obtaining the support of an NE, we show in \Cref{sec: sample_complexity} how to find the NE on this support. To achieve an instance-dependent sample complexity, our algorithm iteratively solves the LP based on the historical samples with an adaptive choice of the constraint slackness, which is inspired by the recent advances in solving online resource allocation that achieves logarithmic regret via LP resolving~(e.g. \cite{vera2021bayesian, li2022online, jiang2022degeneracy, ma2024optimal}). We first show in \Cref{sec: instance-dependent} our algorithm achieves a instance-dependent sample complexity of $\order(\frac{m}{\delta^2\veps}\log\frac{m}{\delta\veps})$, then complementing with an $\order(\frac{m}{\veps^2}\log\frac{m}{\veps})$ $\delta$-independent sample complexity guarantee in \Cref{sec: instance-independent}.\footnote{We omit some other problem-dependent constants for conciseness. These dependencies are explicitly written in the theorem statement.} Combining both results, we show our final sample complexity guarantee in \Cref{sec:Combine}.



\subsection{Other Related Works}
The literature on learning in game is broad. There is a line of works considering the sample complexity of finding NE for matrix games. When the observations are deterministic, \citet{maiti2023query} proposed an inefficient algorithm that identifies all Nash Equilibria by querying \(\mathcal{O}(nk^5\log^2n)\) entries of the matrix, where \( k \) denotes the size of the support across all Nash Equilibria and $A\in \R^{n\times n}$. For Pure Strategy Nash Equilibria (PSNE), a special class of NE where each player's strategy is restricted to a single action, \citet{dallant2024finding} developed a deterministic and efficient algorithm requiring only \(\mathcal{O}(n)\) samples. This was subsequently improved by \citet{dallant2024optimal}, who introduced a randomized algorithm that further reduced the runtime complexity while maintaining the same sample complexity.

When the observations are noisy, \citet{zhou2017identify} proposed a lower-upper-confidence-bound (LUCB)-based algorithm~\citep{kalyanakrishnan2012pac} for identifying PSNE with probability at least $1-\epsilon$, achieving a sample complexity of $\mathcal{O}(\log(1/\epsilon))$. \citet{maiti2024near} later improved this result by refining problem-dependent constants to achieve better sample efficiency. When the game does not necessarily have a PSNE, when $m_1=m_2=2$, \citet{maiti2023instance} shows an $\otil(\min\{\frac{1}{\veps^2},\max\{\frac{1}{\Delta_{\min}^2},\frac{c}{\veps^2 D^2}\}\}\cdot\log(1/\epsilon))$ sample complexity to find an $\veps$-sub-optimality-gap NE with probability at least $1-\epsilon$, together with a matching lower bound up to logarithmic factors. However, we remark that this does not violate our $\otil(\frac{1}{\veps})$ sample complexity bound, since we aim to find an estimate of NE with $\veps$ bias instead of finding an $\epsilon$-sub-optimality-gap NE with high probability. Transforming our result to an $\epsilon$-sub-optimality-gap NE with high probability requires multiple runs of our algorithm.


There is another related line of work which aims to find the NE of a matrix game via regret minimization. Specifically, \citep{freund1999adaptive} shows that if both players adopt an algorithm with individual regret bound $R$ over $T$ rounds, then the average-iterate for each player forms an $\frac{R}{T}$-approximate NE. Since $R=\Theta(\sqrt{T})$ in the worst case, this leads to an $\mathcal{O}(1/\sqrt{T})$-approximate NE. Later, a line of works \citep{daskalakis2011near,rakhlin2013optimization,syrgkanis2015fast} show that with full gradient feedback, if both players apply certain optimistic online learning algorithms, the average-iterate converges to NE with a faster rate. A recent work \citep{ito2025instance} shows that under noisy bandit feedback, if both players apply Tsallis-INF~\citep{JMLR:v22:19-753}, then the expected average-iterate converges to an $\mathcal{\widetilde{O}}(1/\sqrt{T})$-Nash Equilibrium, with a faster $\order(\log T/T)$ instance-dependent rate if the game has a unique PSNE. Besides average-iterate convergence, there are also works studying the last-iterate convergence \citep{daskalakis2018last,wei2020linear,hsieh2021adaptive,anagnostides2022last,cai2022tight,cai2024fast,ito2025instance}, which is more preferable in many applications.  There are also works considering broader games beyond normal form games including convex-concave games~\citep{abernethy2018faster}, strongly monotone games~\citep{ba2025doubly,jordan2025adaptive}, and extensive-form games~\citep{farina2019online,lee2021last}.

Our algorithm design is inspired by the recent advances in near-optimal algorithms for the online resource allocation problem, which has been extensively studied, with diverse applications modeled through various linear program (LP) formulations. Representative examples include the secretary problem \citep{ferguson1989solved}, online knapsack \citep{arlotto2020logarithmic}, network revenue management \citep{gallego1997multiproduct}, network routing \citep{buchbinder2009online}, and online matching \citep{mehta2007adwords}. 
Two primary input models are commonly considered in online LP research: i) the stochastic model, where each constraint column and objective coefficient is drawn independently from an unknown distribution, and ii) the random permutation model, where inputs arrive in a uniformly random order \citep{molinaro2014geometry, agrawal2014dynamic, kesselheim2014primal, gupta2014experts}.
Under a standard non-degeneracy assumption, logarithmic regret bounds have been established for problems such as quantity-based network revenue management \citep{jasin2012re, jasin2014reoptimization}, general online LPs \citep{li2022online}, and convex resource allocation \citep{ma2024optimal}. More recent work has relaxed this assumption (e.g. \cite{bumpensanti2020re, vera2021bayesian, banerjee2024good, jiang2022degeneracy, wei2023constant, jiang2024achieving, ao2025learning, xie2025benefits}), leading to improved theoretical guarantees under broader and more realistic conditions. Our algorithm is new compared to the existing ones and we also relax the non-degeneracy assumption.
\section{Preliminary}\label{sec:pre}
We consider the bilinear saddle point problem defined as follows
\begin{equation}\label{eqn:Formulation}
    \min_{\bm{x}\in\mathcal{X}=\Delta_{m_1}} \max_{\bm{y}\in\mathcal{Y}=\Delta_{m_2}} \bx^\top A\by,
\end{equation}
where $\Delta_n=\{\bx\in\R^n: \sum_{i=1}^n\bx_i=1,\bx_i\geq 0,\forall i\in[n]\}$ denotes the $(n-1)$-dimensional simplex and $A\in[-1,1]^{m_1\times m_2}$ is the payoff matrix. According to the celebrated minimax theorem~\citep{v1928theorie}, we know that $\min_{\bx\in\calX}\max_{\by\in\calY}\bx^\top A\by = \max_{\by\in\calY}\min_{\bx\in\calX}\bx^\top A\by$. Define $\calX^*\times\calY^*$ as the set of Nash Equilibria where $\calX^*=\argmin_{\bx\in\calX}\max_{\by\in\calY}\bx^\top A\by$ and $\calY^*=\argmax_{\by\in\calY}\min_{\bx\in\calX}\bx^\top A\by$.
In the following, we introduce two definitions that measure of the closeness between a strategy $(x,y)$ and the Nash Equilibria set $\calX^*\times\calY^*$.

\begin{definition}\label{def:eps-NE}
    We call a pair of strategy $(\bx,\by)\in\Delta_{m_1}\times\Delta_{m_2}$ an \emph{$\veps$-close NE} if $\argmin_{\bx^*\in\calX^*}\|\bx-\bx^*\|_2\leq \veps$ and $\argmin_{\by^*\in\calY}\|\by-\by^*\|_2\leq \veps$.
\end{definition}

\begin{definition}\label{def:eps-duality} Define the sub-optimality gap of $\bx\in\calX$ and $\by\in\calY$ as $\max_{\by'\in\Delta_{m_2}}\bx^\top A\by' - \max_{\by'\in\Delta_{m_2}}{x^{*\top}} A\by'$ and $\min_{\bx'\in\Delta_{m_1}}\bx{'^\top} A\by^* - \min_{\bx'\in\Delta_{m_1}}{x^{'\top}} A\by$ (independent fo the choice of $x^*\in\calX^*$ and $y^*\in\calY$).
    We call $\bx\in\calX$ ($\by\in\calY$) an $\veps$-sub-optimality-gap NE if $\bx$'s ($\by$'s) sub-optimality gap is no more than $\veps$.
\end{definition}
Since we assume that $A\in[-1,1]^{m_1\times m_2}$, we know that an $\veps$-close NE implies an $\veps$-sub-optimality-gap NE, while the reverse is not true. 

In this paper, we consider the setting where only noisy bandit feedback is available. Specifically, the learner can only query an entry $(i,j)$ of $A$ and observe a sample $A_{i,j}+\eta\in[-1,1]$ where $\eta$ is a zero-mean noise.

\paragraph{Other notations} 
For a vector $\bv \in \mathbb{R}^d$, let $v_i$ denote its $i$-th entry, and for a subset $S \subseteq [d]$, let $\bv_S$ denote the sub-vector of $\bv$ containing the entries indexed by $S$. Denote $\mathbf{e}^{n}$ as the all-one vector in $n$-dimensional space and $\mathbf{0}$ as the all-zero vector in an appropriate dimension. 
We also denote 
$\bm{h}_{n}$ as the one-hot vector with the $n$-th position being $1$ and other positions being $0$ in an appropriate dimension. Given two vectors $\mathbf{u},\bv\in \R^d$, we say $\mathbf{u} \succeq \bv$ ($\mathbf{u}\preceq \bv$) if $u_i \geq v_i$ ($u_i\leq v_i$) for all $i \in [d]$.
For a matrix $M \in \mathbb{R}^{n_1 \times n_2}$ and two sets $S_1 \subseteq [n_1]$, $S_2 \subseteq [n_2]$, let $M_{S_1,:} \in \mathbb{R}^{|S_1| \times n_2}$ denote the sub-matrix containing the rows of $M$ indexed by $S_1$. Similarly, let $M_{:,S_2} \in \mathbb{R}^{n_1 \times |S_2|}$ denote the sub-matrix containing the columns of $M$ indexed by $S_2$, and let $M_{S_1, S_2} \in \mathbb{R}^{|S_1| \times |S_2|}$ denote the one containing the entries with the row indices in $S_1$ and the column indices in $S_2$. We also let $m=m_1m_2$ for notational convenience.

\subsection{Reformulation as Linear Programming}

In this section, we present a reformulation of the bilinear saddle-point problem defined in \Cref{eqn:Formulation}. We show that in order to solve \Cref{eqn:Formulation}, we can focus on two linear programs, which can be regarded as the primal-dual form to each other. We then exploit the specific structures of the linear programming reformulations to derive our algorithms. 
Specifically, for a fixed $\bx$, the inner maximization problem over $\by$ can be viewed as a linear program, the dual of which can be written as 
\begin{equation}
\min_{\mu\in\mathbb{R}}  ~\mu  ~~~\mbox{s.t.} ~\mu\cdot\bm{e}^{m_2}\succeq A^\top\bx.    
\end{equation}
Further minimizing over $\bx\in\mathcal{X}$ reaches the following ``primal'' linear program.\footnote{With a slight abuse of notation, we use the same symbol to denote both the LP and its optimal objective, depending on the context.}
\begin{equation}\label{eqn:primal}
    \VI=\min_{\bx\succeq0, \mu\in\mathbb{R}}~ \mu ~~~\mbox{s.t.}~ \mu\cdot\bm{e}^{m_2}\succeq A^\top\bx, ~(\bm{e}^{m_1})^\top\bx=1.
\end{equation}
According to the celebrated minimax theorem, solving the problem in \Cref{eqn:Formulation} is equivalent to solving $\max_{\by\in\calY}\min_{\bx\in\calX}\bx^\top A\by$. Then, following a similar process, we obtain the ``dual'' linear programming that computes $\by$.
\begin{equation}\label{eqn:dual}
    \VD=\max_{\by\succeq0, \nu\in \R} ~\nu  ~~~\mbox{s.t.}~\nu\cdot\bm{e}^{m_1} \preceq A\by, ~(\bm{e}^{m_2})^\top\by=1.
\end{equation}
Direct calculation shows that LP in \Cref{eqn:primal} and LP in \Cref{eqn:dual} are primal-dual to each other. Moreover, in order to solve \Cref{eqn:Formulation}, classic LP analysis (e.g. Theorem 16.5 in \citep{lecture16}) shows that it suffices to find an optimal solution $\bx^*$ to LP \Cref{eqn:primal} and an optimal solution $\by^*$ to LP \Cref{eqn:dual}, such that $\bx^*$ and $\by^*$ are the corresponding primal-dual solutions. 
\begin{lemma}\label{lem:LPformulation}
For any optimal solution to the primal LP \Cref{eqn:primal}, denoted by $(\bx^*, \bmu^*)$, and the corresponding optimal dual solution to LP \Cref{eqn:dual}, denoted by $(\by^*, \bm{\nu}^*)$,  $(\bx^*, \by^*)$ is an optimal solution to \Cref{eqn:Formulation}.
\end{lemma}

In the following, we first present in \Cref{sec:support-identification} a method for identifying the support of an NE, using $\otil(\frac{m}{\delta^2})$ samples, where $\delta>0$ is a problem-dependent constant defined in \Cref{sec:instance-dependent-id}. Then, in \Cref{sec: sample_complexity}, we demonstrate how to compute the NE supported on the identified support.

\section{Saddle Point Support Identification}\label{sec:support-identification}

In this section, we develop methods to identify the \emph{support} for an optimal solution $\bx^*$ to the primal LP \Cref{eqn:primal}. 
Following standard LP theory, we know that when solving an LP, we can restrict to the corner points of the feasible region, which is also referred to as \textit{basic solution}. 
The following lemma shows that there always exists two index sets, denoted as $\mathcal{I}^{*'}$ and $\mathcal{J}^{*'}$ such that one optimal solution to the LP $\VI$ can be fully characterized by the sets $\mathcal{I}^{*'}$ and $\mathcal{J}^{*'}$ as the solution to a set of linear equations.
\begin{theorem}\label{lem:Basis}
There exists an index set $\mathcal{I}^{*'}\subset [m_1]$ and an index set $\mathcal{J}^{*'}\subset [m_2]$, such that $|\mathcal{I}^{*'}|=|\mathcal{J}^{*'}|=d$, for some integer $d\leq \min\{ m_1, m_2 \}$. Also, for the given sets $\mathcal{I}^{*'}$ and $\mathcal{J}^{*'}$, there exists an optimal solution $(\bm{x}^*, \mu^*)$ to the LP in \Cref{eqn:primal} {such that $(\bm{x}^*, \mu^*)$ is the unique solution to the following linear system}
\begin{equation}\label{eqn:Lsystem1}
    A_{\mathcal{I}^{*'}, \mathcal{J}^{*'}}^\top \bm{x}^*_{\mathcal{I}^{*'}} = \mu^*\cdot\bm{e}^{|\mathcal{J}^{*'}|}\text{~~and~~}
    (\bm{e}^{m_1})^\top\bx^* = 1,
\end{equation}
and
$\bx^*_{(\mathcal{I}^{*'})^c} = \bm{0}$,
where $(\mathcal{I}'^{*'})^{c}=[m_1]\setminus\mathcal{I}^{*'}$, and $(\mathcal{J}^{*'})^{c}=[m_2]\setminus\mathcal{J}^{*'}$ denote the complementary sets.
\end{theorem}

In order to identify the optimal basis $\mathcal{I}'$ and $\mathcal{J}'$ from noisy observations, we first sample each entry $A_{i,j}$ a number of times and construct an empirical game matrix $\wh{A}$ with each entry the empirical average of the $N/m$ observations, and construct the primal LP of \Cref{eqn:primal} using $\wh{A}$. Next, we successively shrink the size of $x$-player's support based on the objective value $\wh{V}_{\calI}^{\prim}$ defined in \Cref{eqn:Izero}. Specifically, $\wh{V}_{\calI}^{\prim}$ is defined as the empirical primal LP with non-zero values only on $\calI$ and we initialize the support set to be $\calI=[m_1]$. At each time, we sequentially pick each $i\in\calI$ and compare the objective value of $\wh{V}_{\calI}^{\prim}$ with $\wh{V}_{\calI\backslash\{i\}}^{\prim}$. If these two quantities are the same, this means that there exists an NE that is not supported on $i$ and we set $\calI\leftarrow\calI\backslash\{i\}$; otherwise, we keep $i$ in $\calI$. In this way, we obtain a support set $\wh{\calI}^*$ after iterating over all $i\in[m_1]$.

\begin{equation}\label{eqn:Izero}
\begin{aligned}
\VI_{\mathcal{I}}=& \min&&\mu && \wh{V}^{\mathrm{Prime}}_{\mathcal{I}}= && \min&&\mu \\
&~~ \mbox{s.t.} &&\mu\cdot\bm{e}^{m_2}\succeq A^\top\bm{x} &&  &&~~\mbox{s.t.} &&\mu\cdot\bm{e}^{m_2}\succeq\wh{A}^\top\bm{x}\\
& &&(\bm{e}^{m_1})^\top\bm{x}=1 && && &&(\bm{e}^{m_1})^\top\bm{x}=1\\
& &&\bm{x}_{\mathcal{I}^c}=\mathbf{0} && && &&\bm{x}_{\mathcal{I}^c}=\mathbf{0}\\
& &&\bm{x}\succeq\mathbf{0}, \mu\in\mathbb{R}, && && &&\bm{x}\succeq\mathbf{0}, \mu\in\mathbb{R}.
\end{aligned}
\end{equation}

After identifying the set $\wh{\calI}^*$, we also need to find the set $\wh{\calJ}^*$, which can then be used to solve the linear equation in \Cref{eqn:Lsystem1} to compute the optimal solution $\bx^*$. Specifically, based on $\wh{\calI}^*$ computed in the previous phase, we define the dual LP $\wh{V}_{\wh{\calI}^*,\calJ}^{\dual}$ in \Cref{eqn:Jzero} (as well as the dual LP based on true matrix given by $\VD_{\wh{\mathcal{I}}^*, \mathcal{J}}$ in \Cref{eqn:Jzero}) in which we only include the rows of $\wh{A}$ whose indices are in $\wh{\calI}^*$ and set the dual variables outside $\calJ$ to be $0$. 
We apply a similar approach to find the active constraints for the support $\wh{\calI}^*$. Specifically,  we initialize $\calJ=[m_2]$ and successively eliminate an element $j\in\calJ$ if $\wh{V}_{\wh{\calI}^*,\calJ\backslash\{j\}}$ equals to $\wh{V}_{\wh{\calI}^*, \calJ}^{\dual}$ and $(\wh{\calI}^*, \calJ\backslash\{j\})$ contains a basis, which is equivalent to checking whether the matrix $\begin{bmatrix}
\wh{A}_{\wh{\calI}^*, \calJ\backslash\{j\}} & -\bm{e}^{|\calI|} \\
(\bm{e}^{|\calJ|-1})^\top & 0
\end{bmatrix}$ is full rank or not. 
Finally, we obtain $\wh{\calJ}^*$ after iterating over all $j\in[m_2]$, where we stop as long as $|\wh{\calJ}^*| = |\wh{\calI}^*|$. The pseudo code of the algorithm is shown in \Cref{alg:Idenbasis}. Note that it is possible that there are multiple optimal basis and we show later that our approach essentially
identifies one particular optimal basis, since our goal is to find one NE of the game. In the following, we provide both instance-dependent and worst-case sample complexity guarantees.

\begin{equation}\label{eqn:Jzero}
\begin{aligned}
\VD_{\wh{\mathcal{I}}^*, \mathcal{J}}=& \max&&\nu && \wh{V}^{\mathrm{Dual}}_{\wh{\mathcal{I}}^*, \mathcal{J}}= && \max&&\nu \\
&~~ \mbox{s.t.} &&\nu\cdot\bm{e}^{|\wh{\mathcal{I}}^*|}\preceq A_{\wh{\mathcal{I}}^*, :}\bm{y} &&  &&~~\mbox{s.t.} &&\nu\cdot\bm{e}^{|\wh{\mathcal{I}}^*|}\preceq\wh{A}_{\wh{\mathcal{I}}^*, :}\bm{y}\\
& &&(\bm{e}^{m_2})^\top\by=1 && && &&(\bm{e}^{m_2})^\top\by=1 \\
& &&\bm{y}_{\mathcal{J}^c}=\mathbf{0} && && &&\bm{y}_{\mathcal{J}^c}=\mathbf{0}\\
& &&\bm{y}\succeq0, \nu\in\mathbb{R}, && && &&\bm{y}\succeq0, \nu\in\mathbb{R}.
\end{aligned}
\end{equation}

\begin{algorithm}[ht!]
\caption{Support Identification}
\label{alg:Idenbasis}
\begin{algorithmic}[1]
\State \textbf{Input:} 
the failure probability $\eps$ and the matrix $\wh{A}$ with each entry being the sample average of $N'$ samples.

\State Denote by $\mathcal{I}_0$ the row index set of $\wh{A}$ and $\mathcal{J}_0$ the column index set of $\wh{A}$.

\State Initialize $\mathcal{I}$ as $\mathcal{I}_0$ and $\calJ$ as $\mathcal{J}_0$.

\State Compute the objective value of $\wh{V}^{\mathrm{Prime}}_{\calI}$ as in \Cref{eqn:Izero}. Set this value to be $V$.
\For{$i\in\mathcal{I}_0$}
\State Let $\mathcal{I}'=\mathcal{I}\backslash\{i\}$ and compute the value of $\wh{V}^{\mathrm{Prime}}_{\mathcal{I}'}$.
\If{$V=\wh{V}^{\mathrm{Prime}}_{\mathcal{I}'}$} 
\State $\mathcal{I}=\mathcal{I}'.$
\EndIf
\EndFor
\For{$j\in\mathcal{J}_0$}
\State Let $\mathcal{J}'=\mathcal{J} \backslash\{j\}$ and compute the value of $\wh{V}^{\mathrm{Dual}}_{\mathcal{I}, \mathcal{J}'}$
\State Compute the smallest singular value of the matrix $\begin{bmatrix}
\wh{A}_{\calI, \calJ'} & -\bm{e}^{|\calI|} \\
(\bm{e}^{|\calJ'|})^\top & 0
\end{bmatrix}$, denoted as $\wh{\sigma}_{\calI, \calJ'}$.

\If{$V=\wh{V}^{\mathrm{Dual}}_{\mathcal{I}, \mathcal{J}'}$ and $\wh{\sigma}_{\calI, \calJ'}> |\calI||\calJ'|\cdot\Rad(N'/m, \eps/m)$ with $\Rad(N'/m, \eps/m)=\sqrt{\frac{m\cdot\log(2m/\eps)}{2N'}}$.} 
\State $\mathcal{J}=\mathcal{J}'.$
\EndIf
\State \textbf{Break} if $|\calJ| = |\calI|$.
\EndFor
\State \textbf{Output}: the sets of indices $\wh{\mathcal{I}}^*=\mathcal{I}$ and $\wh{\mathcal{J}}^*=\mathcal{J}$.
\end{algorithmic}
\end{algorithm}

\subsection{Instance-Dependent Constant Sample Complexity Guarantee}\label{sec:instance-dependent-id}

In this section, we show that \Cref{alg:Idenbasis} can successfully detect one optimal support index sets with a high probability with the number of samples a constant with respect to certain instance-dependent quantities. Specifically, we define $\delta>0$ as follows.
\begin{definition}\label{def:delta}
    Define $$\delta_1\triangleq\min_{\mathcal{I}\subseteq[m_1]}\{\VI_{\mathcal{I}}-\VI: \VI_{\mathcal{I}}-\VI>0\}$$ to be the minimum non-zero primal gap where $\VI$ is defined in \Cref{eqn:primal} and $\VI_{\mathcal{I}}$ is defined in \Cref{eqn:Izero}. Define $$\delta_2\triangleq\min_{\mathcal{I}, \mathcal{J}}\{\VD_{\mathcal{I},[m_2]} - \VD_{\mathcal{I}, \mathcal{J}}: \VD_{\mathcal{I},[m_2]} - \VD_{\mathcal{I}, \mathcal{J}} >0, \VD_{\mathcal{I},[m_2]}=\VI\}$$ to be the minimum non-zero dual gap to the optimal objective where $\VD_{\mathcal{I}, \mathcal{J}}$ is defined in~\Cref{eqn:Jzero}. Define $\delta \triangleq \min\{\delta_1,\delta_2\}$.
\end{definition}
Specifically, $\delta$ measures the minimum non-zero gap to the optimal objective values (for both the primal and the dual) when the value on a support of the NE is set to be $0$.

Since we need to check the full-rankness of a matrix in line 14 of \Cref{alg:Idenbasis}, we also need a parameter reflecting the smallest possible singular value. We denote by $\mathcal{F}$ the collection of all possible support $(\calI, \calJ)$ that could potentially form a basis of $\VI$. To be specific, $\mathcal{F}$ contains all $(\calI, \calJ)\subset[m_1]\times[m_2]$ such that the matrix $
\begin{bmatrix}
A^\top_{\calI, \calJ} & -\bm{e}^{|\calJ|} \\
(\bm{e}^{|\calI|})^\top & 0
\end{bmatrix}$ is full-rank. Formally, we define \begin{align}\label{eqn:F}
    \calF = \left\{(\calI,\calJ)\in[m_1]\times[m_2]: \begin{bmatrix}
A^\top_{\calI, \calJ} & -\bm{e}^{|\calJ|} \\
(\bm{e}^{|\calI|})^\top & 0
\end{bmatrix} \text{~is full-rank}\right\}.
\end{align}
Then, for any $(\calI, \calJ)\in\mathcal{F}$, we denote by $\sigma_{\calI, \calJ}$ the smallest singular value of the matrix $
\begin{bmatrix}
A^\top_{\calI, \calJ} & -\bm{e}^{|\calJ|} \\
(\bm{e}^{|\calI|})^\top & 0
\end{bmatrix}$. From the definition of $\mathcal{F}$, we know that $\sigma_{\calI, \calJ}>0$ for any $(\calI, \calJ)\in\mathcal{F}$. We
provide the following parameter definition that reflects how hard it is to identify whether the matrix is full-rank or not.
\begin{definition}\label{def:sigma0}
Define
\begin{equation}\label{eqn:Sigma0}
\sigma_0=\min_{(\calI, \calJ)\in\mathcal{F}}\left\{\frac{\sigma_{\calI, \calJ}}{2|\calI||\calJ|}: \sigma_{\calI, \calJ}>0\right\}
\end{equation}
to be the smallest singular value of the matrix $
\begin{bmatrix}
A^\top_{\calI, \calJ} & -\bm{e}^{|\calJ|} \\
(\bm{e}^{|\calI|})^\top & 0
\end{bmatrix}$ for all $(\calI, \calJ)\in\mathcal{F}$.
\end{definition}

The next theorem shows that when we input \Cref{alg:Idenbasis} with an empirical matrix $\wh{A}$ such that the number of samples for each entry is at least $\Omega(\frac{1}{\sigma_0^2}+\frac{1}{\delta^2})$, then with a high probability, $\wh{\calI}^*$ and $\wh{\calJ}^*$ output by \Cref{alg:Idenbasis} satisfies \Cref{lem:Basis}. 

\begin{theorem}\label{thm:Infibasis2}
For any $\eps>0$, and a fixed $N \geq \left(\frac{1}{\sigma_0^2}+\frac{1}{\delta^2} \right)\cdot 2m\log(2m/\epsilon)$, if we input \Cref{alg:Idenbasis} with the failure probability $\eps$ and an empirical matrix $\wh{A}$ with each entry the average of $\frac{N}{m}$ noisy observations, then, with probability at least $1-\eps$, the outputs $\wh{\mathcal{I}}^*$ and $\wh{\mathcal{J}}^*$ of \Cref{alg:Idenbasis} satisfy the conditions described in \Cref{lem:Basis}.
\end{theorem}
The proof is deferred to \Cref{app:infibasis2}. To provide a proof sketch, following standard concentration inequalities, with $N\geq \left(\frac{1}{\sigma_0^2}+\frac{1}{\delta^2} \right)\cdot 2m\log(2m/\epsilon)$, each entry of the game matrix $A$ is approximated with an error of $\min\{\frac{\delta}{2},\frac{\sigma_0}{2}\}$ with probability at least $1-\eps$. Under this high probability event, we can successfully check the matrix singularity and show that the gap between the optimal objective value of $\wh{V}_{\calI}^{\mathrm{Prime}}$ and $\VI_{\calI}$, as well as the dual $\wh{V}_{\calI,\calJ}^{\mathrm{Dual}}$ and $\VD_{\calI,\calJ}$ for all $\calI\subseteq[m_1]$ and $\calJ\subseteq[m_2]$, is smaller than $\frac{\delta}{2}$. 
Then, we are able to show that the optimal basis of the empirical LP $\wh{V}^{\mathrm{Prime}}$ estimated from samples is also an optimal basis of the original LP $\VI$.

Three remarks are as follows. First, the proof sketch above shows that we can also directly compute an optimal basis of $\wh{V}^{\mathrm{Prime}}$ and, when $N\geq \left(\frac{1}{\sigma_0^2}+\frac{1}{\delta^2} \right)\cdot 2m\log(2m/\epsilon)$, this basis is also optimal to the original LP $\VI$ with probability $1-\eps$. While there are standard methods to compute the optimal basis of a given LP (such as the simplex method), these can require exponential time in the worst case. In contrast, the approach in \Cref{alg:Idenbasis}, which computes the LP value to select the basis, achieves a polynomial time complexity since linear programming itself can be solved in polynomial time. Second, while we require the knowledge of $\delta$ to decide $N$ in \Cref{thm:Infibasis2}, we show in \Cref{app:parameter_estimation} that we can also estimate $\delta$ within a factor of $2$ using an additional $\order(\frac{m\log(m/\epsilon)}{\delta^2})$ samples \emph{without knowing the parameter $\delta$} by solving mixed integer programs with modern LP solvers as shown in \Cref{sec:oracle}, meaning that we can find the support set without the knowledge of $\delta$. Finally, though
our sample complexity bound depends on $\sigma_0$, in practice, we do not require the knowledge of $\sigma_0$ to carry out our algorithm. As shown in the next section, we can adopt a doubling trick to exponentially increase the sample size to guarantee that \Cref{alg:Idenbasis} can be successfully executed without knowing the value of $\sigma_0$.

We further note that \Cref{alg:Idenbasis} in fact identifies the same optimal basis, with a high probability, for any large enough sample size $N$. Specifically, we define $\calI^*$ and $\calJ^*$ as follows:
\begin{definition}\label{def:istar_jstar}
    Define $\calI^*$ and $\calJ^*$ as the output of \Cref{alg:Idenbasis} with the input being the true matrix $A$. Define $\sigma>0$ as the minimum singular value of $A^*$ defined as 
\begin{equation}\label{eqn:012401}
A^*=\begin{bmatrix}
A^\top_{\mathcal{I}^*, \mathcal{J}^*} & -\bm{e}^{|{\mathcal{J}^*}|} \\
(\bm{e}^{|\mathcal{I}^*|})^\top & 0
\end{bmatrix}.
\end{equation}
\end{definition}
{Following \Cref{thm:Infibasis2} by replacing the empirical matrix $\wh{A}$ with the true matrix, we can show that $\calI^*$ and $\calJ^*$ satisfy the conditions in \Cref{lem:Basis}. This implies that} $A^*$ is of full-rank and we must have $\sigma > 0$. We then show that given $N \geq \left(\frac{1}{\sigma_0^2}+\frac{1}{\delta^2} \right)\cdot 2m\log(2m/\epsilon)$, we will have $\wh{\calI}^*=\calI^*$ and $\wh{\calJ}^*=\calJ^*$ with a probability at least $1-\eps$ after running \Cref{alg:Idenbasis}. We formalize the above argument in the following proposition
\begin{proposition}\label{prop:Basis}
For any $\eps>0$ and fixed $N \geq \left(\frac{1}{\sigma_0^2}+\frac{1}{\delta^2} \right)\cdot 2m\log(2m/\epsilon)$, we will have $\wh{\calI}^*=\calI^*$ and $\wh{\calJ}^*=\calJ^*$ with a probability at least $1-\eps$, where $\calI^*$ and $\calJ^*$ are defined as \Cref{def:istar_jstar}, and $\wh{\calI}^*$ and $\wh{\calJ}^*$ are the output of \Cref{alg:Idenbasis} with input $\eps$ and $\wh{A}$ an empirical matrix with each entry of the average of $\frac{N}{m}$ noisy observations.
\end{proposition}

\subsection{$\delta$-Independent $\otil(m/\veps^2)$ Sample Complexity Guarantee}\label{sec: instance-independent-id}
In \Cref{sec:instance-dependent-id}, we show that with $N= \Theta(\frac{m\log(m/\eps)}{\min\{\delta^2,\sigma_0^2\}})$ samples, \Cref{alg:Idenbasis} identifies the desired optimal support with a probability at least $1-\eps$. 
However, when $\delta$ is small, to obtain the optimal support that satisfies \Cref{lem:Basis}, the number of samples we need can be very large. 
Therefore, in this section, we aim to provide a $\delta$-independent sample complexity guarantee on the obtained support output by \Cref{alg:Idenbasis}. 
In this case, instead of proving that \Cref{alg:Idenbasis} finds the optimal support with a high probability, we show that with $N=\otil(m/\veps^2)$ samples, \Cref{alg:Idenbasis} finds $\mathcal{I}$ and ${\mathcal{J}}$ such that the optimal solution supported on $\mathcal{I}$ has a sub-optimality gap bounded at most $\veps$. 
To this end, for possible basis sets $\mathcal{I}\subset[m_1]$ and $\mathcal{J}\subset[m_2]$ such that $|\mathcal{I}|=|\mathcal{J}|$, we define $\bx^*(\mathcal{I}, \mathcal{J})$ as the solution such that $\bx_{\mathcal{I}^c}^*(\mathcal{I}, \mathcal{J})=\bm{0}$ and the non-zero elements $\bx_{\mathcal{I}}^*(\mathcal{I}, \mathcal{J})$ are the solution to the following equation:
\begin{equation}\label{eqn:IJOptQ}
\begin{bmatrix}
A^\top_{\mathcal{I}, \mathcal{J}} & -\bm{e}^{|\mathcal{J}|} \\
(\bm{e}^{|\mathcal{I}|})^\top & 0
\end{bmatrix}
\begin{bmatrix}
\bx_{\mathcal{I}}^*(\mathcal{I}, \mathcal{J})\\
\mu^*(\mathcal{I}, \mathcal{J})
\end{bmatrix}=
\begin{bmatrix}
\bm{0}\\
1
\end{bmatrix}.
\end{equation}

The following theorem shows the sub-optimality gap of $\bx^*(\mathcal{I}^N, \mathcal{J}^N)$, where $\mathcal{I}^N$ and $\mathcal{J}^N$ denote the output of \Cref{alg:Idenbasis} with input $\wh{A}$ the empirical matrix with each entry the average of $\frac{N}{m}$ noisy observations, for a fixed sample size $N$, as long as $|\calI^N| = |\calJ^N|$. 
\begin{theorem}\label{thm:FiniteGap}
Let $\Rad(n,\epsilon)=\sqrt{\frac{\log(2/\epsilon)}{2n}}$ and $\wh{A}$ be the empirical matrix with each entry the average of $\frac{N}{m}$ noisy observations, for a fixed sample size $N$. Denote by $\calI^N$ and $\calJ^N$ the output of \Cref{alg:Idenbasis} with the input being $\eps$ and $\wh{A}$. Then, as long as $|\calI^N|=|\calJ^N|$, with a probability at least $1-\eps$, we have that
\begin{equation}\label{eqn:012310}
|\VI-\VI_{\calI^N}|\leq\order\left(\Rad(N/m, \eps/m)\right)\text{~~and~~}|\VI-\VD_{\calI^N, \calJ^N}|\leq\order\left(\Rad(N/m,\eps/m)\right).
\end{equation}
Moreover, with probability at least $1-\eps$,
the solution $\bx^*(\calI^N, \calJ^N)$ satisfies that
\[
A^\top\bx^*(\calI^N, \calJ^N) \preceq \left(\VI + \order\left((d^N)^2\cdot\wh{\kappa}\cdot\Rad(N/m, \eps/m)\right) \right)\cdot \bm{e}^{m_2}
\]
where $d^N=|\calI^N|=|\calJ^N|$ and $\wh{\kappa}>0$ is the condition number of the matrix 
$
\begin{bmatrix}
\wh{A}^\top_{\calI^N, \calJ^N} & -\bm{e}^{|\calJ^N|} \\
(\bm{e}^{|\calI^N|})^\top & 0
\end{bmatrix}$. 
\end{theorem}
Note that in \Cref{thm:FiniteGap}, the value $\wh{\kappa}$ converges to the conditional number of the matrix $\begin{bmatrix}
A^\top_{\calI^N, \calJ^N} & -\bm{e}^{|\calJ^N|} \\
(\bm{e}^{|\calI^N|})^\top & 0
\end{bmatrix}$ as $N$ gets large.
The proof is deferred to \Cref{app:instance-independent-id} and we provide a proof sketch here. Specifically, similar to the proof sketch for \Cref{thm:Infibasis2}, based on the high probability event that each entry of $A$ is approximated within an error of $\Rad(N/m,\eps/m)$, we know that the objective value gap between $V^{\prim}$ and $\wh{V}^{\prim}$, as well as $V_{\calI^N}^{\prim}$ and $\wh{V}_{\calI^N}^{\prim}$, are bounded by $\Rad(N/m,\eps/m)$. Since \Cref{alg:Idenbasis} guarantees that $\wh{V}_{\calI^N}^{\prim}=\wh{V}^{\prim}$, we know that the gap between $V_{\calI^N}^{\prim}$ and $V^{\prim}$ is bounded by $\Rad(N/m,\epsilon/m)$, proving the first half of \Cref{eqn:012310}. The second half can be obtained via a similar argument. To prove the sub-optimality gap, we control the distance between $x^*(\calI^N,\calJ^N)$ and $\wh{x}^*(\calI^N,\calJ^N)$, which is the solution of $$\begin{bmatrix}
\wh{A}^\top_{\calI^N, \calJ^N} & -\bm{e}^{|\calJ^N|} \\
(\bm{e}^{|\calI^N|})^\top & 0
\end{bmatrix}\begin{bmatrix}
 \bx \\
 \mu
\end{bmatrix} = \begin{bmatrix}
 \bm{0}\\
 1
\end{bmatrix}.$$
Following a standard perturbation analysis of the linear system, we can show that
$$\|\wh{x}^*(\calI^N,\calJ^N)-x^*(\calI^N,\calJ^N)\|_2\leq \order\left(\wh{\kappa}\cdot \Rad(N/m,\epsilon/m)\right).$$ The remaining analysis follows a direct derivation.

\section{LP-Resolving Based Algorithm}\label{sec: sample_complexity}

In the previous section, we described how to identify one optimal basis $(\wh{\calI}^*, \wh{\calJ}^*)$ by applying \Cref{alg:Idenbasis}. In this section, we describe how to approximate the optimal solutions $\bm{x}^*$ and $\bm{y}^*$ that corresponds to this identified optimal basis. Note that the optimal solution $\bm{x}^*$ enjoys the following structure: $\bm{x}^*_{(\wh{\mathcal{I}}^{*})^{c}}=\bm{0}$ and the other elements $\bm{x}^*_{\wh{\mathcal{I}}^*}$ can be given as the solution to the following linear equation\footnote{$\bx^*$ is in fact $\bx^*(\wh{\calI}^*, \wh{\calJ}^*)$ but we abbreviate as $\bx^*$ in this section for notation simplicity.}
\begin{equation}\label{eqn:OptQ}
\begin{bmatrix}
A^\top_{\wh{\mathcal{I}}^*, \wh{\mathcal{J}}^*} & -\bm{e}^{|\wh{\mathcal{J}}^*|} \\
(\bm{e}^{|\wh{\mathcal{I}}^*|})^\top & 0
\end{bmatrix}
\begin{bmatrix}
\bx_{\wh{\mathcal{I}}^*}^*\\
\mu^*
\end{bmatrix}=
\begin{bmatrix}
\bm{0}\\
1
\end{bmatrix}.
\end{equation}
While \Cref{alg:Idenbasis} outputs an optimal basis $\wh{\calI}^*$ and $\wh{\calJ}^*$ with sufficient number of samples, since the payoff matrix $A$ is unknown and we can only access noisy observations, we must construct estimates of $A$ using sequentially collected samples to approximately solve \Cref{eqn:OptQ}.

Our two-phase procedure in \Cref{alg:Twophase} first employs a doubling schedule for $N'$ (lines 3-4) in order to invoke \Cref{alg:Idenbasis} without requiring prior knowledge of $\sigma_0$. This is necessary because the sample complexity guarantees of \Cref{alg:Idenbasis} (see \Cref{thm:Infibasis2} and \Cref{prop:Basis}) rely on $\sigma_0$. The dependence arises from line 13 of \Cref{alg:Idenbasis}, which aims to test whether the augmented matrix
\begin{align}\label{eqn:true}
    M(\calI,\calJ')=\begin{bmatrix}
A_{\calI,\calJ'} & -\bm{e}^{|\calI|} \\
(\bm{e}^{|\calJ'|})^\top & 0
\end{bmatrix}
\end{align}
is non-singular. This test is carried out by comparing the minimum singular value of its empirical estimate
\begin{align}\label{eqn:empirical}
    \wh{M}(\calI,\calJ')=\begin{bmatrix}
\wh{A}_{\calI,\calJ'} & -\bm{e}^{|\calI|} \\
(\bm{e}^{|\calJ'|})^\top & 0
\end{bmatrix}
\end{align}
with the statistical error level $|\calI|\cdot\Rad(N'/m,\eps/m)$. It succeeds with high probability only if it is no less than $|\calI|\cdot\Rad(N'/m,\eps/m)$, which explains why the required $N'$ depends on $\sigma_0$.

Specifically, when the minimum singular value of \Cref{eqn:empirical} is smaller than the error level, \Cref{alg:Idenbasis} cannot confidently eliminate an element $j\in\wh{\calJ}^*$ {since \Cref{eqn:true} may not be in full-rank with high probability} and therefore may return a basis with $|\wh{\calJ}^*|>|\wh{\calI}^*|$. In this case, we know that the number of samples is not enough to identify a correct basis, and we therefore double $N'$ in the next iteration (line 4 of \Cref{alg:Twophase}). Once $N'$ is large enough, the algorithm returns a basis with $|\wh{\calJ}^*|=|\wh{\calI}^*|$, in which case both $M(\wh{\calI}^*,\wh{\calJ}^*)$ and $\wh {M}(\wh{\calI}^*,\wh{\calJ}^*)$ are square matrices and non-singular. From \Cref{thm:Infibasis2} and \Cref{thm:FiniteGap}, we then know that $\bx^*(\wh{\calI}^*,\wh{\calJ}^*)$ is a feasible solution to $\VI$ with suboptimality $\order(\Rad(N',\eps))$, and becomes optimal when $N' \ge \Omega(\log(m/\eps)/\min\{\delta^2,\sigma_0^2\})$, with probability at least $1-\eps$. The following proposition formalizes this argument and shows that the doubling trick successfully removes the need to know $\sigma_0$ beforehand.

\begin{proposition}\label{prop:DoublingTrick}
Let $N_1$ be the initial value for $N'$, which is the input to \Cref{alg:Twophase}, and let $\wh{\calI}^*$ and $\wh{\calJ}^*$ be the final basis obtained by \Cref{alg:Idenbasis} under the final value of $N'$. Then, we have the following results:
\begin{itemize}
\item[1.] If $N_1$ satisfies that $N_1\geq \frac{2m\cdot\log(16m/\veps)}{\delta^2}$, then with probability at least $1-\veps/4$, we have $\wh{\calI}^*$ and $\wh{\calJ}^*$ to be an optimal basis to $\VI$ satisfying the conditions described in \Cref{lem:Basis}. 
\item[2.] For an arbitrary $N_1$, denote by $K$ the number of times that \Cref{alg:Idenbasis} is called in the iteration in line 3 and line 4 of \Cref{alg:Twophase}. Then, with probability at least $1-\veps/4$, the solution $\bx^*(\wh{\calI}^*, \wh{\calJ}^*)$ forms a feasible solution to $\VI$ 
with a sub-optimality gap bounded by $\order\left(\wh{\kappa}\cdot\Rad(N'/m, \eps/m)\right)$,
where $\eps=\veps/(8K^2)$ and $\wh{\kappa}>0$ is the condition number of the matrix 
$
\begin{bmatrix}
\wh{A}^\top_{\wh{\calI}^*, \wh{\calJ}^*} & -\bm{e}^{|\wh{\calJ}^*|} \\
(\bm{e}^{|\wh{\calI}^*|})^\top & 0
\end{bmatrix}$.
\item[3.] With probability at least $1-\veps/4$, the following bound holds for $N_2$, which is the total number of samples used in line 3 and line 4 of \Cref{alg:Twophase}:
\begin{equation}\label{eqn:FinalNprime}
N_2 \leq N_1 + \order\left(\frac{m}{\sigma_0^2}\cdot\log(m/\veps)\right),
\end{equation}
where $\sigma_0>0$ is defined in \Cref{def:sigma0}.
\end{itemize}
\end{proposition}

Specifically, \Cref{prop:DoublingTrick} shows that, depending on the initialization of $N_1$, the algorithm either recovers the optimal basis or guarantees a bounded suboptimality gap with high probability. In particular, if $N_1$ is chosen such that $N_1 \geq \frac{2m\cdot\log(16m/\veps)}{\delta^2}$, then the procedure returns the optimal basis with high probability. On the other hand, for an arbitrary (and $\delta$-independent) choice of $N_1$, the procedure also yields a feasible basis with a bounded suboptimality gap. In practice, the value of $\delta$ can be estimated, for instance, via \Cref{alg:estimateDelta} introduced in \Cref{app:parameter_estimation}.

It remains to approximate the LP solution $\bx^*(\wh{\calI}^*, \wh{\calJ}^*)$ associated with the identified basis $(\wh{\calI}^*, \wh{\calJ}^*)$. To do this, we solve the LP in \Cref{eqn:OptQ} using observed samples, and the design of this stage (line 5 to line 14) is inspired by LP-resolving algorithms for online resource allocation~\citep{vera2021bayesian, li2022online, jiang2022degeneracy, ma2024optimal} using noisy observations. 

The key idea is to interpret \Cref{eqn:OptQ} as a resource allocation problem: the right-hand side specifies the available resources, the decision variable $\mathbf{x}$ represents actions, and the left-hand side matrix encodes the resource consumption. Under this view, the algorithm selects a sequence of actions $\{\mathbf{x}_n\}_{n=N_2+1}^N$ so that resources are consumed appropriately across $N-N_2$ iterations.  

A crucial feature of \Cref{alg:Twophase} is its adaptive resource allocation step (line 12). By updating $\mathbf{a}^n$ according to \Cref{eqn:UpdateAlpha2}, the algorithm implements a self-correcting LP system as in \Cref{eqn:OptQ2}. This step ensures that $\mathbf{a}^n$ converges to $\mathbf{0}$ with a gap of order $\order(1/(N-n+1))$, which ensures the instance-dependent guarantee. Intuitively, if a binding constraint $j \in \wh{\calJ}^*$ is lagging behind its target under the current average action $\sum_{n'=N_2+1}^n \mathbf{x}^{n'}/(n-N_2+1)$, the algorithm compensates by inflating $\mathbf{a}^n_j/(N-n+1)$ above $\mathbf{a}_j$, which increases the weight of that constraint in computing $\tilde{\mathbf{x}}^{n+1}$. Through this continuous feedback loop, constraint violations are corrected on the fly, leading to the convergence rate $\order(1/(N-n+1))$.

{Let $\bar{\bx}$ denote the final output of \Cref{alg:Twophase}, and let $\bx^*(\wh{\calI}^*,\wh{\calJ}^*)$ be the basic solution corresponding to the basis $(\wh{\calI}^*,\wh{\calJ}^*)$. The main result we show is that after $N-N_2$ iterations, our second phase (line 5 to line 14 of \Cref{alg:Twophase}) guarantees that
\begin{align}\label{eqn:main_gap}
    \|\mathbb{E}[\bar{\bx}] - \bx^*(\wh{\calI}^*, \wh{\calJ}^*)\|_2
    \;\leq\;
    \order\!\left(\frac{d^{13/2}}{\sigma_{\wh{\calI}^*, \wh{\calJ}^*}^3}\cdot\frac{\log(m(N-N_2))}{N-N_2}\right),
\end{align}
where the expectation is taken over the randomness of the samples collected in the resolving procedure of \Cref{alg:Twophase}.
To achieve the target accuracy level that scales with $\order(\veps)$, it suffices to ensure that $\|\mathbb{E}[\bar{\bx}] - \bx^*(\wh{\calI}^*, \wh{\calJ}^*)\|_2\leq \order(\frac{\veps}{d})$. Therefore, the total number of samples $N$ can be chosen as
\begin{align}\label{eqn:value_N}
    N = N_2 + \Theta\left(\frac{d^{15/2}}{\sigma_{\wh{\calI}^*, \wh{\calJ}^*}^3}\cdot\frac{\log(m/\veps)}{\veps}\right),
\end{align}
where $\sigma_{\wh{\calI}^*, \wh{\calJ}^*}$ is the smallest singular value of
\[
\begin{bmatrix}
{A}^\top_{\wh{\calI}^*, \wh{\calJ}^*} & -\bm{e}^{|\wh{\calJ}^*|} \\
(\bm{e}^{|\wh{\calI}^*|})^\top & 0
\end{bmatrix}.
\]
{We arrive at our final sampling complexity guarantee by combining this with the sample complexity of obtaining $(\wh{\calI}^*, \wh{\calJ}^*)$ proven in \Cref{prop:DoublingTrick}.}}

One remaining step is that we can not decide the value of $N$ according to \Cref{eqn:value_N} when we run \Cref{alg:Twophase} since $\sigma_{\wh{\calI}^*, \wh{\calJ}^*}$ is unknown. To address this issue, we estimate $\sigma_{\wh{\calI}^*, \wh{\calJ}^*}$ using \Cref{alg:estimateSigma} (see \Cref{app:parameter_estimation}), which outputs $\sigma'$ satisfying $\tfrac{\sigma'}{2} \leq \sigma_{\wh{\calI}^*, \wh{\calJ}^*} \leq 2\sigma'$ with probability at least $1-\eps$, where we set $\eps=\veps/4$. We therefore replace $\sigma_{\wh{\calI}^*, \wh{\calJ}^*}$ with $\sigma'$ in \Cref{eqn:value_N} when we decide the value of $N$ in \Cref{alg:Twophase} (line 6 of \Cref{alg:Twophase}).

In the following, we first present the formal result for the resolving procedure of \Cref{alg:Twophase} in \Cref{lem:resolving} with the formal proof in \Cref{sec:resolving}. Then, we introduce our $\delta$-dependent sample complexity for \Cref{alg:Twophase} in \Cref{sec: instance-dependent}, followed by a $\delta$-independent sample complexity in \Cref{sec: instance-independent}.

\begin{algorithm}[t]
\caption{The Resolving Algorithm}
\label{alg:Twophase}
\begin{algorithmic}[1]

\State \textbf{Input:} the accuracy level $\veps$, $L=4$ and the initial number of samples $N_1$.
\State Initialize $N'=N_1$ and $k=1$.

\State Obtain the support $\wh{\calI}^*$ and $\wh{\calJ}^*$ output by \Cref{alg:Idenbasis} where the input is $\eps=\veps/(8k^2)$ an empirical matrix with each entry the average of $\frac{N'}{m}$ noisy observations.
\State \textbf{if} $|\wh{\calI}^*| < |\wh{\calJ}^*|$ \textbf{then} $N'\leftarrow 2N'$ and $k\leftarrow k + 1$ and GOTO line 3.

\State Denote $N_2=2N'-N_1$ and initialize $\mathcal{H}^{N_2+1}=\emptyset$, $\bm{a}^{N_2+1}= \mathbf{0} \in\mathbb{R}^{|\wh{\calJ}^*|}$. 

\State Obtain an estimate $\sigma'$ using \Cref{alg:estimateSigma} with input being $\eps=\veps/12$, $\wh{\calI}^*$ and $\wh{\calJ}^*$ and set
\[
N = N_2 + \frac{4120d^{15/2}}{(\sigma')^3}\cdot\frac{\log(m/\veps)}{\veps}
\]
with $d=|\wh{\calI}^*|=|\wh{\calJ}^*|$.
\For{$n=N_2+1,\dots, N$}
\State Construct estimates $\wh{A}_{\wh{\calI}^*, \wh{\calJ}^*}$ using the dataset $\mathcal{H}^n$.

\State Construct a solution $(\tilde{\bm{x}}^n, \tilde{\mu}^n)$ such that $\tilde{\bm{x}}^n_{(\wh{\calI}^{*})^{c}}=\bm{0}$ and $\tilde{\bm{x}}^n_{\wh{\calI}^*}$ is the solution to
\begin{equation}\label{eqn:OptQ2}
    \begin{bmatrix}
    \wh{A}^\top_{\wh{\calI}^*, \wh{\calJ}^*} & -\bm{e}^{|\wh{\calJ}^*|} \\
    (\bm{e}^{|\wh{\calI}^*|})^\top & 0
    \end{bmatrix}
    \begin{bmatrix}
    \bx_{\wh{\calI}^*}\\
    \mu
    \end{bmatrix}=
    \begin{bmatrix}
    \bm{a}^n/(N-n+1)\\
    1
    \end{bmatrix}.
\end{equation}

\State Project $(\tilde{\bx}^{n}, \tilde{\mu}^n)$
    ~to the set $\mathcal{L}=\{(\bx, \mu): \bx\succeq\bm{0}, \|(\bx, \mu)\|_2\leq L\}$ 
    to obtain $(\bx^n, \mu^n)$:
    \[
    (\bx^n, \mu^n) = \argmin_{(\bx, \mu)\in\mathcal{L}}\|((\bx, \mu))-(\tilde{\bx}^{n}, \tilde{\mu}^n)\|_2.
    \]
\State Observe $\wt{A}_{n,i_n,j_n}=A_{i_n,j_n}+\eta_n$, where $i_n$ and $j_n$ are uniformly sampled from $\wh{\calI}^*$ and $\wh{\calJ}^*$. 
\State Update $\mathcal{H}^{n+1}=\mathcal{H}^n\cup \{ (i_n,j_n,\wt{A}_{n,i_n, j_n}) \}$ and compute $\bm{a}^{n+1}$ as   
    \begin{equation}\label{eqn:UpdateAlpha2}
\bm{a}^{n+1}=\bm{a}^{n}- |\wh{\calJ}^*|\cdot|\wh{\calI}^*|\cdot \wt{A}_{n, i_n,j_n}\cdot x^n_{i_n}\cdot \bm{h}_{j_n} +\mu^n\cdot \bm{e}^{|\wh{\calJ}^*|},
\end{equation}
where $\bm{h}_{j_n}\in\mathbb{R}^{|\wh{\calJ}^*|}$ denotes a vector with $j_n$-th position being $1$ and other positions being $0$.
\EndFor
\State Compute
    $\bar{\bx}=\frac{1}{N-N_2}\cdot\sum_{n=N_2+1}^N\bx^n$.
\State \textbf{Output:} $\bar{\bx}$.
\end{algorithmic}
\end{algorithm}

\subsection{Analysis for the Resolving Procedure}\label{sec:resolving}
In this section, we provide the formal analysis for the resolving procedure in \Cref{alg:Twophase}. Note that the analysis in this section is regarding the resolving procedures carried out from line 7 to line 14 of \Cref{alg:Twophase}, and the expectation is taken over the randomness in the resolving procedures. A key step in our analysis is to characterize how $\bm{a}^n$ behaves during the execution of \Cref{alg:Twophase}. For notational convenience, we define
\begin{equation}\label{eqn:Average}
\tilde{\bm{a}}^n\triangleq\frac{\bm{a}^n}{N-n+1}.
\end{equation}
Our proof proceeds in the following three steps.

\textbf{Step 1: Bounding $\|\mathbb{E}[\bar{\bx}]-\bx^*(\wh{\calI}^*, \wh{\calJ}^*)\|_2$ by $\mathbb{E}[\bm{a}^{N+1}]$.} First, we show that the gap between the expected output of our algorithm $\E[\bar{\bx}]$ and the corresponding solution $\bx^*(\wh{\calI}^*, \wh{\calJ}^*)$ defined in \Cref{eqn:IJOptQ} can be bounded with respect to $\mathbb{E}[\bm{a}^{N+1}]$. Specifically, according to the update rule of $\tilde{\ba}^n$ shown in \Cref{eqn:UpdateAlpha2}, we know that
\begin{equation}\label{eqn:Aveupalpha}
\tilde{\bm{a}}^{n+1}=\tilde{\bm{a}}^{n}+\frac{\tilde{\bm{a}}^n - |\wh{\calJ}^*|\cdot|\wh{\calI}^*|\cdot \wt{A}_{n, i_n,j_n}\cdot x^n_{i_n}\cdot \bm{h}_{j_n} +\mu^n\cdot \bm{e}^{|\wh{\calJ}^*|}}{N-n}.
\end{equation}
Note that since $i_n$ (resp. $j_n$) is sampled from a uniform distribution over $\wh{\calI}^*$ (resp. $\wh{\calJ}^*$), we have that
\begin{equation}\label{eqn:012501}
\mathbb{E}_{\eta_n, i_n, j_n}\left[|\wh{\calJ}^*|\cdot|\wh{\calI}^*|\cdot \wt{A}_{n, i_n,j_n}\cdot x^n_{i_n}\cdot \bm{h}_{j_n}\right] = A^\top_{\wh{\calI}^*, \wh{\calJ}^*}\bx^n_{\wh{\calI}^*}.
\end{equation}
By taking expectation over $i_n$ and $j_n$, as well as $\wt{A}_{n, i_n,j_n}$, for each $n=N_2+1$ up to $n=N$ in \Cref{eqn:Aveupalpha} and combining with \Cref{eqn:012501}, we can show that $\E[\bar{\bx}]$ satisfies the following equation 
\[
\begin{bmatrix}
A^\top_{\wh{\mathcal{I}}^*, \wh{\mathcal{J}}^*} & -\bm{e}^{|\wh{\mathcal{J}}^*|} \\
(\bm{e}^{|\wh{\mathcal{I}}^*|})^\top & 0
\end{bmatrix}
\begin{bmatrix}
\mathbb{E}[\bar{\bx}_{\wh{\mathcal{I}}^*}]\\
\mathbb{E}[\bar{\mu}]
\end{bmatrix}=
\begin{bmatrix}
-\frac{\mathbb{E}[\bm{a}^{N+1}]}{N-N_2}\\
1
\end{bmatrix},
\]
where $\bar{\mu}=\frac{1}{N-N_2}\cdot\sum_{n=N_2+1}^N \mu^n$ in the above equation. On the other hand, from the definition of $\bx^*(\wh{\calI}^*, \wh{\calJ}^*)$ in \Cref{eqn:IJOptQ}, it holds that
\[
\begin{bmatrix}
A^\top_{\wh{\mathcal{I}}^*, \wh{\mathcal{J}}^*} & -\bm{e}^{|\wh{\mathcal{J}}^*|} \\
(\bm{e}^{|\wh{\mathcal{I}}^*|})^\top & 0
\end{bmatrix}
\begin{bmatrix}
\bx^*_{\wh{\mathcal{I}}^*}(\wh{\calI}^*, \wh{\calJ}^*)\\
\mu^*(\wh{\calI}^*, \wh{\calJ}^*)
\end{bmatrix}=
\begin{bmatrix}
\bm{0}\\
1
\end{bmatrix}.
\]
We can apply the perturbation theory of linear system (e.g. Theorem 1 of \citep{higham2002accuracy}) to bound the gap between $\E[\bar{\bx}]$ and $\bx^*(\wh{\calI}^*, \wh{\calJ}^*)$ since they are solutions to linear systems with perturbation on the right-hand side. We have that
\[
\begin{aligned}
    \|\bm{x}^*-\mathbb{E}[\bar{\bm{x}}]\|_2
    \leq \kappa\cdot \frac{\left\|\mathbb{E}[\bm{a}^{N+1}]\right\|_2}{N-N_2}.
\end{aligned}
\]
where $\kappa$ is the conditional number of the matrix $M(\wh{\calI}^*, \wh{\calJ}^*)$. Further upper bounding $\left\|\mathbb{E}[\bm{a}^{N+1}]\right\|_2\leq \sqrt{d}\left\|\mathbb{E}[\bm{a}^{N+1}]\right\|_\infty$ leads to
\begin{align*}
    \|\bm{x}^*-\mathbb{E}[\bar{\bm{x}}]\|_2
    \leq \kappa\cdot\sqrt{d} \frac{\left\|\mathbb{E}[\bm{a}^{N+1}]\right\|_\infty}{N-N_2}\leq \order\left(\frac{d^{3/2}}{\sigma'}\frac{\left\|\mathbb{E}[\bm{a}^{N+1}]\right\|_\infty}{N-N_2}\right),
\end{align*}
where the last inequality is due to $\sigma'\in [\frac{\sigma_{\wh{\calI}^*,\wh{\calJ}^*}}{2}, 2\sigma_{\wh{\calI}^*,\wh{\calJ}^*}]$ with high probability.

To further bound $\mathbb{E}[\bm{a}^{N+1}]$, we denote by $\tau$ the first time index at which $\|\tilde{\bm{a}}^n\|_{\infty}$ exceeds a problem-dependent constant $\eta$ (formally defined in \Cref{eqn:Stoptime} and \Cref{eqn:nu0}).

\textbf{Step 2: Bounding  $\mathbb{E}[\bm{a}^{N+1}]$ by $\E[N-\tau]$.} Following the update in \Cref{eqn:UpdateAlpha2}, we know that
\begin{align*}
\mathbb{E}[\bm{a}^{N+1}]&=\mathbb{E}\left[ \bm{a}^{\tau-1} \right] 
-\sum_{n=\tau}^N A^\top_{\wh{\mathcal{I}}^*, \wh{\mathcal{J}}^*}\mathbb{E}[\bx^n_{\wh{\mathcal{I}}^*}]+\sum_{n=\tau}^N\mathbb{E}[\mu^n]\cdot \bm{e}^{|\wh{\mathcal{J}}^*|}.
\end{align*}
From the definition of the stop time $\tau$, it holds that (since $\eta<1$), 
\[
\left\|\mathbb{E}[\bm{a}^{\tau-1}]\right\|_{\infty} \leq \mathbb{E}[N-\tau+2].
\]
Also, from the boundedness that $\|(\bx^n, \mu^n)\|_1\leq L$, we can derive that
\[\begin{aligned}
\left\|\sum_{n=\tau}^N \left(A^\top_{\wh{\mathcal{I}}^*, \wh{\mathcal{J}}^*} \mathbb{E}[\bx^n_{\wh{\mathcal{I}}^*}] - \mathbb{E}[\mu^n]\cdot\bm{e}^{|\wh{\mathcal{J}}^*|}\right) \right\|_{\infty}
\leq L\cdot(N-\tau).
\end{aligned}\]
In this way, we show that $\|\mathbb{E}[\bm{a}^{N+1}]\|_\infty\leq\order\left(\E[N-\tau]\right)$.

\textbf{Step 3: Final bound on $\E[N-\tau]$.}
To provide a bound on $\E[N-\tau]$, we first show in \Cref{lem:projection} (deferred to Appendix) that for periods $n$ as long as $\Omega\left(\frac{d^5}{(\sigma')^{2}}\cdot\log(d^2/\veps)\right)\leq n\leq\tau$, with a probability at least $1-\veps/N$, it holds that $\|(\tilde{\bx}^n, \tilde{\mu}^n)\|_2\leq L$, and as a result we know that $(\tilde{\bx}^n, \tilde{\mu}^n)=(\bx^n, \mu^n)$ in later rounds, which implies that $\wh{A}_{\wh{\calI}^*, \wh{\calJ}^*}^\top x^n_{\wh{\calI}^*} - \mu^n\cdot \bm{e}^{|\wh{\calJ}^*|} = \tilde{\bm{a}}^{n}$ since $(\bx^n, \mu^n)$ is then the solution of \Cref{eqn:OptQ2}.
Therefore, suppose that $\wh{A}_{\wh{\calI}^*, \wh{\calJ}^*}=A_{\wh{\calI}^*, \wh{\calJ}^*}$, then according to the update rule of $\tilde{\bm{a}}^{n+1}$ shown in \Cref{eqn:Aveupalpha}, $\tilde{\bm{a}}^{n+1}$ will have the same expectation as $\tilde{\bm{a}}^{n}$ such that they become a martingale. 
While $\wh{A}_{\wh{\calI}^*, \wh{\calJ}^*}$ is only an empirical estimate of 
$A_{\wh{\calI}^*, \wh{\calJ}^*}$ and may not coincide with it, we can show that 
the sequence $\tilde{\bm{a}}^{n}$ satisfies a sub-martingale property. 
The deviation from the true value is controlled by the estimation error of 
$\wh{A}_{\wh{\calI}^*, \wh{\calJ}^*}$, which diminishes as the number of 
samples increases.
Then, applying standard concentration inequalities for sub-martingales, we prove that 
$\tilde{\bm{a}}^n$ concentrates around $\mathbf{0}$, which implies that $\|\tilde{\bm{a}}^n\|_{\infty}$ is within the problem-dependent constant $\eta$ for sufficiently long time. 
In this way, we derive an upper bound $\mathbb{E}[N-\tau]$ (formally proved in \Cref{lem:Stoptime}) and we show that $\mathbb{E}[N-\tau]\leq\order\left(\frac{d^5}{(\sigma')^{2}}\cdot\log(d^2/\veps)\right)$. 

Combining these steps, we show that
\[
\|\bm{x}^*-\mathbb{E}[\bar{\bm{x}}]\|_2
    \leq \order\left(\frac{d^{3/2}}{\sigma'}\cdot \frac{\left\|\mathbb{E}[\bm{a}^{N+1}]\right\|_\infty}{N-N_2}\right)\leq \order\left(\frac{d^{3/2}}{\sigma'}\cdot \frac{\E[N-\tau]}{N-N_2}\right)\leq \order\left(\frac{d^{13/2}}{(\sigma')^{3}}\cdot\frac{\log(d^2/\veps)}{N-N_2}\right).
\]
With the specific choice of $N$ set in line 6 of \Cref{alg:Twophase}, which can be derived from setting the right-hand-side of the above inequality to be $\veps/(2d)$, we obtain the following result, showing that the resolving procedure enable us to well approximate the optimal solution to \Cref{eqn:OptQ}.

\begin{lemma}\label{lem:resolving}
Denote by $\wh{\calI}^*$ and $\wh{\calJ}^*$ the final output of \Cref{alg:Idenbasis} 
in the execution of \Cref{alg:Twophase}. Denote by $\bar{\bx}$ the output of 
\Cref{alg:Twophase}. Then, it holds that
\[
\left\|\mathbb{E}[\bar{\bx}]-\bx^*(\wh{\calI}^*, \wh{\calJ}^*)\right\|_2 \leq \veps/(2d).
\]
\end{lemma}
Note that the above \Cref{lem:resolving} holds for any pair of $\wh{\calI}^*$ and $\wh{\calJ}^*$ as long as they are the final output of \Cref{alg:Idenbasis} in the execution of \Cref{alg:Twophase}. On one hand, when $\wh{\calI}^*$ and $\wh{\calJ}^*$ is an optimal basis, we know that the resolving procedure in \Cref{alg:Twophase} well approximates an optimal solution. On the other hand, when $\wh{\calI}^*$ and $\wh{\calJ}^*$ is not optimal but enjoys a bounded suboptimality gap, we know that the resolving procedure in \Cref{alg:Twophase} approximates a sub-optimal solution with a bounded suboptimality gap.

\subsection{$\delta$-dependent Sample Complexity}\label{sec: instance-dependent}
We now show that our \Cref{alg:Twophase} achieves a $\delta$-dependent sample complexity for approximating the optimal solution to the primal LP problem $\VI$. 

Our result follows from the combination of \Cref{prop:DoublingTrick} and \Cref{lem:resolving}.
Note that when $N_1\geq\frac{2m\cdot\log(16m/\veps)}{\delta^2}$, from \Cref{prop:DoublingTrick}, with a probability at least $1-\veps/4$, we know that the final output $\wh{\calI}^*$ and $\wh{\calJ}^*$ of \Cref{alg:Idenbasis} in the execution of \Cref{alg:Twophase} is an optimal basis. Then, from \Cref{lem:resolving}, we know that our resolving procedure guarantees that $\|\mathbb{E}[\bar{\bx}]-\bx(\wh{\calI}^*, \wh{\calJ}^*)\|_2\leq\order(\frac{\veps}{d})=\order(\veps)$.\footnote{We only require $\|\mathbb{E}[\bar{\bx}]-\bx(\wh{\calI}^*, \wh{\calJ}^*)\|_2\leq \order(\veps)$ in proving \Cref{thm:sampleComplexity}. We require the tighter $\order(\frac
{\veps}{d})$ bound on $\|\mathbb{E}[\bar{\bx}]-\bx(\wh{\calI}^*, \wh{\calJ}^*)\|_2$ when proving the sub-optimality gap in \Cref{thm:sample-complexity-independent}.} Therefore, we establish in the following theorem our formal theoretical bound by applying \Cref{alg:Idenbasis} and \Cref{alg:Twophase}.

\begin{theorem}\label{thm:sampleComplexity}
For any $\veps>0$, suppose that the input $N_1$ of \Cref{alg:Twophase} satisfies that $N_1\geq\frac{2m\cdot\log(16m/\veps)}{\delta^2}$
, where $\delta=\min\{\delta_1, \delta_2\}$ with $\delta_1$ and $\delta_2$ given in Definition \ref{def:delta}. Then, the output of \Cref{alg:Twophase} guarantees that
\[
\argmin_{\bx^*\in\calX^*}\|\E[\bar{\bx}]-\bx^*\|_2\leq \veps.
\]
Moreover, with a probability at least $1-\order(\veps)$, the total number of samples required by \Cref{alg:Twophase} is bounded by   
$$
N= \order \left( \left(\frac{m}{\sigma_0^2}+ \frac{m}{\delta^2}\right)\cdot\log(m/\veps) + \frac{(d^*)^{15/2}}{\sigma^3}\cdot\frac{\log(m/\veps)}{\veps} \right),$$
where $\sigma_0$ is given in \Cref{def:sigma0}. 
Here, $\sigma>0$ is the minimum singular value of the full-rank matrix $A^*$ defined in \Cref{def:istar_jstar} and $\calI^*$ and $\calJ^*$ are defined in \Cref{def:istar_jstar} with $d^*=|\calI^*|=|\calJ^*|$.
\end{theorem}
{Note that the input $N_1$ presented in \Cref{thm:sampleComplexity} needs to satisfy the condition that $N_1\geq\frac{2m\cdot\log(16m/\veps)}{\delta^2}$, again depending on the gap $\delta$}. However, neither our \Cref{alg:Idenbasis} nor \Cref{alg:Twophase} requires the knowledge of $\delta$, {and} we can estimate $\delta$ within a factor of $2$ according to \Cref{thm:EstimateDelta}, by solving mixed integer programs with modern LP solvers as shown in \Cref{sec:oracle}. Therefore, we can use the estimate $\delta'$ obtained by \Cref{alg:estimateDelta} to determine the value of $N_1$ to serve as the input to \Cref{alg:Twophase}. Note that from \Cref{thm:EstimateDelta}, obtaining the estimate $\delta'$ requires a number of samples bounded by $\order\left(\frac{m}{\delta^2}\cdot\log(m/\eps)\right)$, which is captured by the sample complexity bound presented in \Cref{thm:sampleComplexity}.

{The bound in \Cref{thm:sampleComplexity} depends on the parameter $\sigma$, the condition number of the matrix $A^*$ in \Cref{eqn:012401}. This dependence is a consequence of our objective and estimation setup: we aim to compute an $\veps$-close NE by approximating the optimal solution $\bx^*$ using samples of the game matrix $A$, so estimation error in $A$ propagates to the solution. The parameter $\sigma$ captures this propagation, as is standard in the sensitivity analysis of linear systems (e.g., \citep{higham2002accuracy}). Therefore, we regard the presence of $\sigma$ as an inherent consequence of adopting LP-resolving algorithms. Similar condition-number dependencies appear in prior work on LP-resolving with instance-dependent guarantees (e.g., \citep{vera2021bayesian, bumpensanti2020re, li2021symmetry, jiang2022degeneracy}), where the constraint-matrix condition number enters regret bounds, which can be translated into sample-complexity guarantees.}



\subsection{$\delta$-Independent Sample Complexity}\label{sec: instance-independent}

In this section, we show that \Cref{alg:Twophase} also achieves a $\delta$-independent sub-optimality gap with $N$ samples, complementing the $\delta$-dependent sample complexity proven in \Cref{sec: instance-independent}. {Note that for a fixed sample size $N$, it is possible that the output of \Cref{alg:Twophase}, $\wh{\calI}^*$ and $\wh{\calJ}^*$, does not equal to $\calI^*$ and $\calJ^*$ defined in \Cref{def:istar_jstar}. To address this issue, we denote by $\mathcal{F}'$ the collection of all possible support $(\calI, \calJ)$ that forms a basis to $\VI$.
To be specific, $\mathcal{F}'$ contains all $(\calI, \calJ)\subset[m_1]\times[m_2]$ such that $|\calI|=|\calJ|$ and the square matrix $M(\calI, \calJ)=
\begin{bmatrix}
A^\top_{\calI, \calJ} & -\bm{e}^{|\calJ|} \\
(\bm{e}^{|\calI|})^\top & 0
\end{bmatrix}$ is non-singular. Formally, we define 
\begin{align}\label{eqn:Fprime}
    \calF' = \left\{(\calI,\calJ)\in[m_1]\times[m_2]: |\calI|=|\calJ|, \begin{bmatrix}
A^\top_{\calI, \calJ} & -\bm{e}^{|\calJ|} \\
(\bm{e}^{|\calI|})^\top & 0
\end{bmatrix} \text{~is non-singular}\right\}.
\end{align}

Then, for any $(\calI, \calJ)\in\mathcal{F}'$, we denote by $d_{\calI, \calJ}=|\calI|=|\calJ|$ and $\sigma_{\calI, \calJ}$ the smallest singular value of the square matrix $M(\calI, \calJ)$. From the definition of $\mathcal{F}'$, we know that $\sigma_{\calI, \calJ}>0$ for any $(\calI, \calJ)\in\mathcal{F}'$. Note that when the number of samples is independent of $\delta$, it can happen that \Cref{alg:Twophase} approximates the solution corresponding to any support $(\calI, \calJ)\in\mathcal{F}'$. We now define the following parameters corresponding to the \textit{worst-case} support $(\calI, \calJ)\in\mathcal{F}'$:}
\begin{equation}\label{eqn:worstcaseIJ}
d_0=\max_{(\calI, \calJ)\in\mathcal{F}'}\{d_{\calI, \calJ}\}.
\end{equation}
We are now ready to state our $\delta$-independent sample complexity in the following theorem.
\begin{theorem}\label{thm:sample-complexity-independent}
Given a fixed $\veps>0$ and any $\alpha>0$, \Cref{alg:Twophase} with input $\veps$ and $N_1=\frac{\alpha\cdot m}{\veps^2}$ guarantees that 
$\mathbb{E}[\bar{\bx}]$ forms a feasible solution to $\VI$ with an sub-optimality gap bounded by
\[
\veps\cdot \left(\frac{1}{2}+\left(1+\frac{16d_0^3}{\sigma_0}\right)\cdot \sqrt{\frac{1}{2\alpha}\cdot\left(\log(16m/\veps)+2\log\log(m/\sigma_0^2)\right)}\right),
\]
where $\sigma_0$ is defined in \Cref{def:sigma0}.
Moreover, with probability at least $1-\veps/4$, the total number of samples needed by \Cref{alg:Twophase} is upper bounded by
 \begin{align}\label{eqn:sample_complexity_delta_ind}
N = \order\left( \frac{\alpha\cdot m}{\veps^2} + \frac{m}{\sigma_0^2}\cdot\log(m/\veps)+ \frac{d_0^{15/2}}{\sigma_0^3}\cdot\frac{\log(m/\veps)}{\veps}\right).
 \end{align}
where $d_0$ is defined in \Cref{eqn:worstcaseIJ}.
\end{theorem}
In order to choose the $\alpha$ such that the sub-optimality gap is bounded by $\veps$, one can adopt a doubling trick to exponentially increase the value of $\alpha$ and check whether the following inequality holds,
\[
\left(1+\frac{32d^{3}}{\sigma'}\right)\cdot \sqrt{\frac{1}{2\alpha}\cdot\left(\log(16m/\veps) + 2\log\log(4m/\sigma^{'2})\right)}\leq\frac{1}{2},
\]
where $d$ and $\sigma'$ are defined in \Cref{alg:Twophase}.
{While in fact, if one purely seeks a worst-case $\order(1/\veps^2)$ sample complexity, one can also directly estimate the LP $\VI$ and solve it, we present our $\delta$-independent bound in \Cref{thm:sample-complexity-independent} to demonstrate the robustness of our \Cref{alg:Twophase}.


We provide a proof sketch of \Cref{thm:sample-complexity-independent} with the full proof deferred to \Cref{app:sample-complexity-independent}. Note that following \Cref{thm:FiniteGap}, for a fixed sample size $N'$, \Cref{alg:Idenbasis} outputs index sets $\mathcal{I}^{N'}, \mathcal{J}^{N'}$ such that the variable $\bx^*(\mathcal{I}^{N'}, \mathcal{J}^{N'})$, as defined in \Cref{eqn:IJOptQ}, forms a feasible solution to $\VI$ with an suboptimality gap $\widetilde{\order}(1/\sqrt{N'})$. Then, we use the index sets $\mathcal{I}^{N'}, \mathcal{J}^{N'}$ as the input to \Cref{alg:Twophase}. Note that given the index sets $\mathcal{I}^{N'}$ and $\mathcal{J}^{N'}$, the resolving steps in \Cref{alg:Twophase} essentially approximate the solution to the linear equations given in \Cref{eqn:IJOptQ} whose suboptimality gap is $\widetilde{\order}(1/\sqrt{N'})$. This is where the first term in \Cref{eqn:sample_complexity_delta_ind} comes from. Following the same procedure as the proof of \Cref{thm:sampleComplexity}, we can show that after $N'$ steps of the resolving process, the gap between $\mathbb{E}[\bar{\bx}]$ and $\bx^*(\mathcal{I}^{N'}, \mathcal{J}^{N'})$ is upper bounded by $\widetilde{\order}(1/N')$, which {corresponds to the second term in the sample complexity shown in \Cref{eqn:sample_complexity_delta_ind}}.

\section{Combining Two Steps Together}\label{sec:Combine}
Now we analyze the overall performance of our algorithm. Specifically, combining \Cref{thm:Infibasis2} and  \Cref{thm:FiniteGap} for the support identification step and \Cref{thm:sampleComplexity} and \Cref{thm:sample-complexity-independent} for the LP resolving step, we show the sample complexity of finding $\veps$-suboptimality-gap NE in the following theorem.  

\begin{theorem}\label{thm:BestBoth}
For a fixed accuracy level $\veps>0$, denote by $\bar{\bx}$ the output of \Cref{alg:Twophase} with $N$ samples. Then, \Cref{alg:Twophase} guarantees that $\mathbb{E}[\bar{\bx}]$ is a $\veps$-sub-optimality-gap with $N$ bounded by
\[
N = \order\left( \frac{d_0^6\cdot m}{\sigma_0^2\cdot\veps^2}+\frac{d_0^{15/2}}{\sigma_0^3}\cdot\frac{\log(m/\veps)}{\veps}\right)~\text{and}~N=\order\left(\left(\frac{m}{\sigma_0^2}+ \frac{m}{\delta^2}\right)\cdot\log(m/\veps) + \frac{(d^*)^{15/2}}{\sigma^3}\cdot\frac{\log(m/\veps)}{\veps} \right)
\]
with respect to different choices of $N_1$, 
where $d_0$ and $\sigma_0$ are defined in \Cref{eqn:worstcaseIJ}, $\delta=\min\{\delta_1, \delta_2\}$ with $\delta_1$ and $\delta_2$ defined in \Cref{def:delta}, $\sigma>0$ denotes the minimum singular value of $A^*$ defined in \Cref{def:istar_jstar}, and $d^*=|\calI^*|=|\calJ^*|$. 
\end{theorem}
The first term presented in \Cref{thm:BestBoth} corresponds to the $\delta$-independent bound presented in \Cref{thm:sample-complexity-independent} and the second term presented in \Cref{thm:BestBoth} corresponds to the $\delta$-dependent bound presented in \Cref{thm:sampleComplexity}. {We remark again that our algorithm does not require the knowledge of these problem-dependent parameters to achieve the bounds shown in \Cref{thm:BestBoth}. Moreover, in \Cref{app:parameter_estimation}, we present algorithms showing how to approximate the problem parameters using the same sample complexity.} However, if we further know the value of $\delta$ (or have an estimate of the value of $\delta$), then we can select the appropriate value of $N_1$ such that we achieve a sample complexity that is the minimum of the two terms presented in \Cref{thm:BestBoth}, which is formalized in the following corollary. The full proof of \Cref{coro:BestBoth} is deferred to \Cref{app:thm10}.
\begin{corollary}\label{coro:BestBoth}
For a fixed accuracy level $\veps>0$, denote by $\bar{\bx}$ the output of \Cref{alg:Twophase} with $N$ samples. Then, there exists a choice of $N_1$ to serve as input such that \Cref{alg:Twophase} guarantees that $\mathbb{E}[\bar{\bx}]$ enjoys a $\veps$-sub-optimality-gap with $N$ bounded by
\[
N = \order\left(\min\left\{\frac{d_0^{15/2}}{\sigma_0^3}\cdot\frac{\log(m/\veps)}{\veps} + \frac{d_0^6\cdot m}{\sigma_0^2\cdot\veps^2}, \left(\frac{m}{\sigma_0^2}+ \frac{m}{\delta^2}\right)\cdot\log(m/\veps) + \frac{(d^*)^{15/2}}{\sigma^3}\cdot\frac{\log(m/\veps)}{\veps}\right\}\right).
\]
\end{corollary}

\subsection{Results for the Dual Player}
In this section, we briefly discuss how we derive an approximate NE for the column player. 
Specifically, we apply the exact same procedure, \Cref{alg:Idenbasis} followed by \Cref{alg:Twophase}, but to the dual LP $\VD$ in \Cref{eqn:dual} instead. We then obtain the output  $\bar{\by}$ of \Cref{alg:Twophase} as the approximation to $\by^*$ for the dual player. Note that all our previous results for the primal LP in \Cref{eqn:primal} still applies to the dual LP. The only difference is that the problem-dependent constants are now defined for the dual LP instead. For example, we can similarly define the LP
\[
\begin{aligned}
\VD_{\calJ}=& \max&&\nu && \VI_{\calI, \calJ}= && \min&&\mu \\
&~~ \mbox{s.t.} &&\nu\cdot\bm{e}^{m_1}\preceq A\bm{y} &&  &&~~\mbox{s.t.} &&\mu\cdot\bm{e}^{|\calJ|}\succeq A_{:,\calJ}^\top\bm{x}\\
& &&(\bm{e}^{m_2})^\top\bm{y}=1 && && &&(\bm{e}^{m_1})^\top\bm{x}=1\\
& &&\bm{y}_{\mathcal{J}^c}=\mathbf{0} && && &&\bm{x}_{\mathcal{I}^c}=\mathbf{0}\\
& &&\bm{y}\succeq\mathbf{0}, \nu\in\mathbb{R}, && && &&\bm{x}\succeq\mathbf{0}, \mu\in\mathbb{R}.
\end{aligned}
\]
Then, 
the problem parameters presented in Definition \ref{def:delta} can be modified into the following:
\begin{definition}\label{def:delta2}
    Define $$\delta^{\mathrm{Dual}}_1\triangleq\min_{\mathcal{J}\subseteq[m_2]}\{\VD-\VD_{\mathcal{J}}: \VD-\VD_{\mathcal{J}}>0\}$$ to be the minimum non-zero primal gap where $\VD$ is defined in \Cref{eqn:dual} and we further restrict $\by_{\mathcal{J}^c}=\bm{0}$ to obtain $\VD_{\mathcal{J}}$ . Define $$\delta^{\mathrm{Dual}}_2\triangleq\min_{\mathcal{I}, \mathcal{J}}\{\VI_{\mathcal{I}, \mathcal{J}}-\VI_{[m_1],\calJ}: \VI_{\mathcal{I}, \mathcal{J}}-\VI_{[m_1],\calJ}>0, \VI_{[m_1], \calJ} = \VI\}$$ to be the minimum non-zero dual gap. 
    Define $\delta^{\mathrm{Dual}} \triangleq \min\{\delta_1^{\mathrm{Dual}},\delta_2^{\mathrm{Dual}}\}$.
\end{definition}
Similarly, we are able to define $\sigma^{\mathrm{Dual}}$, $\sigma_0^{\mathrm{Dual}}$, $d^{\mathrm{Dual}}$, and $d_0^{\mathrm{Dual}}$ correspondingly to the dual system.
Then, following the same procedure as \Cref{thm:sampleComplexity}, we know that by choosing a input sample size $N_1$ to \Cref{alg:Idenbasis} satisfying $N_1=\Theta\left(\frac{m\cdot\log(m/\veps)}{(\delta^{\mathrm{Dual}})^2} \right)$, 
\Cref{alg:Twophase} outputs a solution $\bar{\by}$ such that $\|\by^*-\mathbb{E}[\bar{\by}]\|_2 \leq \veps$, for some optimal dual solution $\by^*$, with a sample complexity bounded by
\[
N=\order\left(\left(\frac{m}{(\sigma_0^{\mathrm{Dual}})^2}+ \frac{m}{(\delta^{\mathrm{Dual}})^2}\right)\cdot\log(m/\veps) + \frac{(d^{\mathrm{Dual}})^{15/2}}{(\sigma^{\mathrm{Dual}})^3}\cdot\frac{\log(m/\veps)}{\veps} \right)
\]
Therefore, applying \Cref{alg:Idenbasis} and \Cref{alg:Twophase} separately to both $x$-player and $y$-player 
leads to the output $(\bar{\bx}, \bar{\by})$ such that $(\mathbb{E}[\bar{\bx}], \mathbb{E}[\bar{\by}])$ is a $\veps$-close NE. A similar $\delta^{\mathrm{Dual}}$-independent sample complexity guarantee
\[
N = \order\left( \frac{(d^{\mathrm{Dual}}_0)^6\cdot m}{(\sigma_0^{\mathrm{Dual}})^2\cdot\veps^2}+\frac{(d_0^{\mathrm{Dual}})^{15/2}}{(\sigma_0^{\mathrm{Dual}})^3}\cdot\frac{\log(m/\veps)}{\veps}\right)
\]
can be obtained by adapting \Cref{thm:FiniteGap} and \Cref{thm:sample-complexity-independent} to the dual LP.
\section{Parameters Estimation}\label{app:parameter_estimation}

In this section, we show how to estimate the unknown problem-dependent parameters mentioned in previous sections. First, we show that the gap $\delta$ related to LP defined in \Cref{def:delta} can be estimated {practically} efficiently via a mixed integer programming oracle using finite samples.
Specifically, for a general LP $V$ of the form:
\[
V = \min~ \bm{c}_0^\top\bx~~~\mbox{s.t.}~A_0\bx\preceq \bm{b}_0, ~\bm{0}\preceq \bx\preceq \bm{1}, 
\]
where the decision variable $\bx\in\mathbb{R}^{m_0}$. For any set $\calI\subset[m_0]$, we also denote by 
\[
V_{\calI} = \min~ \bm{c}_0^\top\bx~~~\mbox{s.t.}~A_0\bx\preceq \bm{b}_0, ~\bx_{\calI}=\bm{0}, ~\bm{0}\preceq \bx\preceq \bm{1}. 
\]
We assume that we have access to an oracle that can solve the following minimum non-zero suboptimality gap problem
\begin{equation}\label{eqn:012701}
    \min_{\mathcal{I}\subset[m_0]} \{ V_{\mathcal{I}} - V: V_{\mathcal{I}} - V > 0  \}
\end{equation}
for a given index set $[m_0]$ and a given LP $V$. We present a concrete algorithm for such oracle in detail in \Cref{sec:oracle}, which is based on solving a mixed integer programming and can be carried out by calling modern LP solvers. 
Based on this oracle, our algorithm is formally presented in \Cref{alg:estimateDelta}, which continuously samples $i_n\in[m_1]$ and $j_n\in[m_2]$ lexicographically and calls the oracle to find the corresponding $\wh{\delta}_1$ and $\wh{\delta}_2$ for the empirical average matrix. Once $\min\{\wh{\delta}_1,\wh{\delta}_2\}$ is sufficiently large compared to $\Rad(n/m, \eps/m)$ as defined in line 8, with a high probability, we can approximate $\delta$ within a constant factor.

\begin{algorithm}[t]
\caption{Algorithm to estimate $\delta$}
\label{alg:estimateDelta}
\begin{algorithmic}[1]
\State \textbf{Input:} a failure probability $\epsilon>0$, an oracle $\calA$ that can solve the problem in \Cref{eqn:012701}.

\For{$n=1, 2, \dots, $ until stopped}
\State Let $ a_n = (n ~\mathrm{mod}~m)+1$ and compute $a_n = j_n\cdot m_1+i_n$ with $i_n$ and $j_n$ are integers and $i_n\in\{1,\dots, m_1\}$ and $j_n\in\{1, \dots, m_2\}$.

\State Take the primal action $i_n$ and the dual action $j_n$, and observe $\wt{A}_{n, i_n,j_n} = A_{i_n, j_n} + \eta_n$.

\State Update $\mathcal{H}^{n+1}=\mathcal{H}^n\cup\{(i_n, j_n, \wt{A}_{n, i_n,j_n}\}$.

\State Construct the empirical mean matrix $\wh{A}$ using $\calH^{n+1}$ and LP $\wh{V}^{\mathrm{Prime}}$.

\State Call $\calA$ with input $\wh{V}^{\mathrm{Prime}}$ to compute 
\begin{equation}\label{eqn:012702}
    \wh{\delta}_1 = \min_{\mathcal{I}\subseteq[m_1]} \{ \wh{V}^{\mathrm{Prime}}_{\mathcal{I}} - \wh{V}^{\mathrm{Prime}}: \wh{V}^{\mathrm{Prime}}_{\mathcal{I}}-\wh{V}^{\mathrm{Prime}}>0  \}
\end{equation}
and compute
\begin{equation}\label{eqn:012703}
    \wh{\delta}_2=\min_{\mathcal{I}\subseteq [m_1], \mathcal{J}\subseteq [m_2]}\{\wh{V}^{\mathrm{Dual}}_{\mathcal{I},[m_2]} - \wh{V}^{\mathrm{Dual}}_{\mathcal{I}, \mathcal{J}}: \wh{V}^{\mathrm{Dual}}_{\mathcal{I},[m_2]} - \wh{V}^{\mathrm{Dual}}_{\mathcal{I}, \mathcal{J}} >0, \wh{V}^{\mathrm{Dual}}_{\mathcal{I},[m_2]}=\wh{V}^{\mathrm{Prime}}\}.
\end{equation}

\State Stop if $\wh{\delta}=\min\{\wh{\delta}_1, \wh{\delta}_2\}\geq 4\cdot\Rad(n/m, \eps/m)$ with $\Rad(n/m, \eps/m)=\sqrt{\frac{m\cdot\log(2m/\eps)}{2n}}$.
\EndFor

\State \textbf{Output:} the estimation $\wh{\delta}$ and the sample set $\mathcal{H}^{n+1}$.
\end{algorithmic}
\end{algorithm}

We provide a brief explanation on why \Cref{alg:estimateDelta} can be used to give an estimate of the gap $\delta$. Note that following standard Hoeffding inequality, for any $\eps>0$, we know that the gap between each element of the matrix $A$ and $\wh{A}$ is upper bounded by $\Rad(n/m, \eps/m)$ with probability at least $1-\eps$, where
\[
\Rad(n/m, \eps/m)=\sqrt{\frac{m\cdot\log(2m/\eps)}{2n}}.
\]
Under this high probability event, we know that the value between $\VI_{\mathcal{I}}$ and $\wh{V}_{\mathcal{I}}^{\mathrm{Prime}}$, as well as $\VD_{\mathcal{I}, \mathcal{J}}$ and $\wh{V}^{\mathrm{Dual}}_{\mathcal{I}, \mathcal{J}}$, for any $\mathcal{I}\subset[m_1]$ and $\mathcal{J}\subset[m_2]$, is upper bounded by $\Rad(n/m,\eps/m)$. This result further implies that
\begin{align*}
    |\delta_1-\wh{\delta}_1|\leq 2\cdot\Rad(n/m,\eps/m) \text{~~and~~}
|\delta_2-\wh{\delta}_2|\leq 2\cdot\Rad(n/m,\eps/m).
\end{align*}
From the above two inequalities, we can derive that
\[
|\delta-\wh{\delta}|\geq 2\cdot\Rad(n/m,\eps/m),
\]
where $\delta=\min\{\delta_1,\delta_2\}$ and $\wh{\delta}=\min\{\wh{\delta}_1,\wh{\delta}_2\}$.
Therefore, as long as $\wh{\delta}\geq4\cdot\Rad(n/m, \eps/m)$, we know that $\frac{\wh{\delta}}{2}\leq\delta\leq 2\wh{\delta}$ with probability at least $1-\eps$. Then we provide an upper bound on the number of samples \Cref{alg:estimateDelta} will use to output such $\wh{\delta}$. Specifically, when $n=\frac{18m}{\delta^2}\log(2m/\eps)$, we know that $\Rad(n/m,\eps/m) = \frac{\delta}{6}$ and with probability at least $1-\eps$,
\begin{align*}
    \wh{\delta} \geq \delta - 2\Rad(n/m,\eps/m) = 4\Rad(n/m,\eps/m).
\end{align*}
Therefore, we know that with probability $1-\eps$, \Cref{alg:estimateDelta} will use no more than $\frac{18m}{\delta^2}\log(2m/\eps)$ samples. Formally, we have the following theorem.
\begin{theorem}\label{thm:EstimateDelta}
Denote by $\wh{\delta}$ the output of \Cref{alg:estimateDelta} and denote by $N_3$ the number of samples used in \Cref{alg:estimateDelta}. Then, with a probability at least $1-\eps$, it holds that
\[
\frac{\wh{\delta}}{2}\leq \delta\leq 2\wh{\delta}\text{~~and~~} N_3\leq \frac{18m}{\delta^2}\cdot\log(2m/\eps).
\]
\end{theorem}
We remark that although the bound on $N_3$ in \Cref{thm:EstimateDelta} depends on the true value of $\delta$, our \Cref{alg:estimateDelta} does not require the knowledge of $\delta$. 

Next, we show that the parameter $\sigma$, the minimum singular value of $A^*$ defined in \Cref{eqn:012401}, can also be efficiently estimated using finite samples. Specifically, we develop an approach to estimate the smallest singular value of the matrix $M(\calI, \calJ)=\begin{bmatrix}
A^\top_{\mathcal{I}, \mathcal{J}}, & -\bm{e}^{|\mathcal{J}|} \\
 (\bm{e}^{|\mathcal{I}|})^\top & 0
\end{bmatrix}$
for any basis $\calI$ and $\calJ$ with $d'=|\calI|=|\calJ|$, which we denote by $\sigma_{\calI, \calJ}$. The algorithm is described in \Cref{alg:estimateSigma}.  
Similarly, we provide a brief explanation on why \Cref{alg:estimateSigma} can be used to give a good estimation of the parameter $\sigma_{\calI, \calJ}$. Note that following standard Hoeffding inequality, for any $\eps>0$, we know that the absolute value of the matrix $\Delta A= A_{\mathcal{I}, \mathcal{J}} - \wh{A}_{\mathcal{I}, \mathcal{J}}$ is upper bounded by $\Rad(n/(d')^2, \eps/(d')^2)$ with probability at least $1-\eps$. Then following Theorem 1 of \citet{stewart1998perturbation}, we know that the difference of the corresponding singular values of $M(\calI, \calJ)$ and $\wh{M}(\calI, \calJ)=\begin{bmatrix}
\wh{A}^\top_{\mathcal{I}, \mathcal{J}}, & -\bm{e}^{|\mathcal{J}|} \\
 (\bm{e}^{|\mathcal{I}|})^\top & 0
\end{bmatrix}$ (sorted from the largest to the smallest) is upper bounded by $\|\Delta A\|_2$, which is again upper bounded by $d'\cdot\Rad(n/(d')^2, \eps/(d')^2)$ with a probability at least $1-\eps$. Formally, it holds that
\begin{equation}\label{eqn:012601}
    \Pr\left[ \left| \sigma_{\calI, \calJ}-\wh{\sigma} \right| \geq d'\cdot\Rad(n/(d')^2, \eps/(d')^2) \right] \leq \eps.
\end{equation}
Therefore, as long as $\wh{\sigma}\geq 2d'\cdot\Rad(n/(d')^2, \eps/(d')^2)$, we know that $\frac{\wh{\sigma}}{2}\leq\sigma_{\calI, \calJ}$ with probability at least $1-\eps$.

\begin{algorithm}[t]
\caption{Algorithm to estimate smallest singular value}
\label{alg:estimateSigma}
\begin{algorithmic}[1]
\State \textbf{Input:} a failure probability $\epsilon>0$, the support set $\mathcal{I}$ and $\mathcal{J}$ that $|\calI|=|\calJ|=d'$.

\For{$n=1, 2, \dots, $ until stopped}
\State Let $ a_n = (n ~\mathrm{mod}~(d')^2)+1$ and compute $a_n = b_n\cdot d'+c_n$ with $b_n$ and $c_n$ are integers and $b_n\in\{1,\dots, d'\}$ and $c_n\in\{1, \dots, d'\}$.
\State We select $(i_n,j_n)\sim\mathcal{I}\times\mathcal{J}$ such that $i_n$ is the $b_n$-th element in $\mathcal{I}$ and $j_n$ is the $c_n$-th element in $\mathcal{J}$.

\State Take the primal action $i_n$ and the dual action $j_n$, and observe $\wt{A}_{n, i_n,j_n} = A_{i_n, j_n} + \eta_n$.

\State Update $\mathcal{H}^{n+1}=\mathcal{H}^n\cup\{(i_n, j_n, \wt{A}_{n, i_n,j_n)}\}$.

\State Construct $\wh{A}_{\mathcal{I}, \mathcal{J}}(\mathcal{H}^{n+1})$ and let $\wh{\sigma}$ denote the smallest singular value of $\begin{bmatrix}
\wh{A}^\top_{\mathcal{I}, \mathcal{J}}, & -\bm{e}^{|\mathcal{J}|} \\
 (\bm{e}^{|\mathcal{I}|})^\top & 0
\end{bmatrix}$.

\State Stop if $\wh{\sigma}\geq 2d'\cdot\Rad(n/(d')^2, \eps/(d')^2)$.
\EndFor

\State \textbf{Output:} the estimation $\wh{\sigma}$.
\end{algorithmic}
\end{algorithm}

\begin{theorem}\label{thm:BoundN1}
Denote by $\wh{\sigma}$ the output of \Cref{alg:estimateSigma} and denote by $N_4$ the number of samples used in \Cref{alg:estimateSigma}. Then, with a probability at least $1-\eps$,
we have
\[
\frac{\wh{\sigma}}{2}\leq\sigma_{\calI, \calJ}\leq2\wh{\sigma}\text{~~and~~}N_4\leq \frac{9(d')^4}{2\sigma_{\calI, \calJ}^2}\cdot\log(2(d')^2/\eps).
\]
\end{theorem}
Note that \Cref{thm:BoundN1} ensures that $\frac{\wh{\sigma}}{2}$ is a lower bound on $\sigma_{\calI, \calJ}$ and we also obtain a high probability bound on the number of samples needed.

\section{Conclusion}

In this paper, we develop new algorithms for computing Nash equilibrium (NE) estimates in matrix games with noisy bandit feedback. Our approach is inspired by the linear programming (LP) formulations for computing primal and dual optimal solutions and consists of two main components. First, we identify the support set and the associated constraints of an optimal solution, which is equivalent to recovering an optimal basis of the corresponding LP from finite samples. Then, we design a resolving algorithm to approximate the optimal solution given this basis. The resulting algorithm achieves the better of the instance-dependent and instance-independent guarantees. There are several natural extensions. In this work, we consider the fixed payoff matrix $A$ and the observations of each entry follow a stationary stationary distribution. Developing policies for time-varying nonstationary games \citep{cardoso2019competing,zhang2022no} remains open. Moreover, while our analysis assumes a stochastic setting, it would be interesting to investigate the robustness of our algorithms under adversarial corruptions. We leave these directions to future research.



\bibliographystyle{abbrvnat}
\bibliography{bibliography}

\clearpage

\OneAndAHalfSpacedXII

\begin{APPENDICES}
\crefalias{section}{appendix}


\section{Missing Proofs in \Cref{sec:support-identification}}

\subsection{Proof of \Cref{lem:Basis}}
\Cref{lem:Basis} follows from the theory of simplex method for solving LP. We rewrite the LP $\VI$ in \Cref{eqn:primal} in the standard form as 
\begin{equation}\label{eqn:Standprimal}
    \VI=\min_{\bx, \bm{s}\succeq\mathbf{0}, \mu\in\mathbb{R}}~ \mu ~~~\mbox{s.t.}~  A^\top\bx + \bm{s}- \mu\cdot\bm{e}^{m_2}=\bm{0}, ~(\bm{e}^{m_1})^\top\bx=1,
\end{equation}
where we introduce the slackness variables $\bm{s}\in\mathbb{R}^{m_2}$ and relax the restriction $\bx\succeq \mathbf{0}$. The index set for the variables $(\bx,\bm{s},\mu)$ is given by $[m_1+m_2+1]$ where the first $m_1$ coordinates correspond to $\bx$, the next $m_2$ coordinates correspond to $\bm{s}$, and the last coordinate corresponds to $\mu$. Denote by $\wh{\calI}$ an optimal basis to the LP \Cref{eqn:Standprimal}. 
Then, we know that the corresponding optimal solution, denoted by $(\bx^*, \bm{s}^*, \mu^*)$, can be derived as the \textit{unique} solution to the linear system
\begin{equation}\label{eqn:Standsystem}
    A_{\wh{\calI}_1, :}^\top\bx_{\wh{\calI}_1} + \bm{s}- \mu\cdot\bm{e}^{m_2}=\bm{0}, ~(\bm{e}^{|\wh{\calI}_1|})^\top\bx_{\wh{\calI}_1}=1, ~\bx_{\wh{\calI}_1^c}=\bm{0}, ~\bm{s}_{\wh{\calI}_2^c}=\bm{0},
\end{equation}
where we denote by $\wh{\calI}_1$ as $\wh{\calI}$ restricted to the index sets of $\bx$ and $\wh{\calI}_2$ as $\wh{\calI}$ restricted to the index sets of $\bm{s}$. Define $\wh{\calI}_1^c=[m_1]\backslash\wh{\calI}$ to be the complementary set of $\wh{\calI}_1$ with respect to $x$ and $\wh{\calI}_2^c=[m_1+m_2]\backslash[m_1]\backslash\wh{\calI}$ to be the complementary set of $\wh{\calI}_2$ with respect to $\bm{s}$. Note that there are $m_2+1$ constraints in $\VI$ in \Cref{eqn:Standprimal}. Then the basis set $\wh{\calI}$ should contain $m_2+1$ elements, with one element being the one denoting the index for $\mu$. 
This implies that $|\wh{\calI}_1|+|\wh{\calI}_2|=m_2$, which implies $|\wh{\calI}_1|=|\wh{\calI}_2^c|$. Then, we know that the linear system
\begin{equation}\label{eqn:Standsystem2}
    A_{\wh{\calI}_1, \wh{\calI}_2}^\top\bx_{\wh{\calI}_1} + \bm{s}_{\wh{\calI}_2}= \mu\cdot\bm{e}^{|\wh{\calI}_2|}, ~ A_{\wh{\calI}_1, \wh{\calI}_2^c}^\top\bx_{\wh{\calI}_1} = \mu\cdot\bm{e}^{|\wh{\calI}_2^c|},~(\bm{e}^{|\wh{\calI}_1|})^\top\bx_{\wh{\calI}_1}=1,
\end{equation}
also has a unique optimal solution, which is $(\bx^*_{\wh{\calI}_1}, \bm{s}_{\wh{\calI}_2}, \mu^*)$. Otherwise, if the linear system \Cref{eqn:Standsystem2} has a different solution, then we append the solution with $\bx_{\wh{\calI}_1^c}=\bm{0}$ and $\bm{s}_{\wh{\calI}^c_2}=\bm{0}$, and we obtain a different solution to \Cref{eqn:Standsystem}, which violates the uniqueness of the optimal solution $(\bx^*, \bm{s}^*, \mu^*)$. 

Moreover, the uniqueness of the optimal solution to \Cref{eqn:Standsystem2} implies that all the linear equations in \Cref{eqn:Standsystem2} are linearly independent of each other, since the number of equations and variables are equivalent to each other. Therefore, we know that the subset of the equations in \Cref{eqn:Standsystem2}, which is shown in the following linear system
\begin{equation}\label{eqn:Standsystem3}
     A_{\wh{\calI}_1, \wh{\calI}_2^c}^\top\bx_{\wh{\calI}_1} = \mu\cdot\bm{e}^{|\wh{\calI}_2^c|},~(\bm{e}^{|\wh{\calI}_1|})^\top\bx_{\wh{\calI}_1}=1
\end{equation}
is also linearly independent of each other. Further note that the number of equations and variables in \Cref{eqn:Standsystem3} is equivalent to each other, we know that \Cref{eqn:Standsystem3} enjoys a unique optimal solution, which is given by $(\bx^*_{\wh{\calI}_1}, \mu^*)$. Therefore, we know that the square matrix
\begin{equation}\label{eqn:squareMa}
\tilde{A}=\begin{bmatrix}
A^\top_{\wh{\calI}_1, \wh{\calI}^c_2} & -\bm{e}^{|\wh{\calI}^c_2|} \\
(\bm{e}^{|\wh{\calI}_1|})^\top & 0
\end{bmatrix}
\end{equation}
is of full rank and all of its columns are linear independent. Further, we denote by $\calI^{*'}$ the support set of $\bx^*$ such that $\bx^*_{\calI^{*'}}\succ\bm{0}$ and $\bx^*_{(\calI^{*'})^c}=\bm{0}$. Then, we know that $\calI^{*'}\subset\wh{\calI}_1$ and the matrix 
\begin{equation}\label{eqn:Ma}
\tilde{A}^*\triangleq\begin{bmatrix}
A^\top_{\calI^{*'}, \wh{\calI}^c_2} & -\bm{e}^{|\wh{\calI}^c_2|} \\
(\bm{e}^{|\calI^{*'}|})^\top & 0
\end{bmatrix}
\end{equation}
is of full column rank. We also know that $(\bx^*_{\calI^{*'}}, \mu^*)$ can be expressed as a solution to the linear system 
\begin{equation}\label{eqn:Standsystem4}
\tilde{A}^* \begin{bmatrix}
    \bx^*_{\calI^{*'}}\\
    \mu^*
\end{bmatrix}=\begin{bmatrix}
    \bm{0}\\
    1
\end{bmatrix}.
\end{equation}
We now specify an index set $\calJ^{*'}\subset\wh{\calI}^c_2$ such that $\calI^{*'}$ and $\calJ^{*'}$ satisfy the requirements. Note that the rank of a matrix is given by the largest rank of all its square submatrix. Therefore, we know that there exists a square submatrix of $\tilde{A}^*$, denoted as $\bar{A}^*$, such that the rank equals $|\calI^{*'}|+1$. On the other hand, the matrix $\tilde{A}^*$ has $|\calI^{*'}|+1$ number of columns. Therefore, we know that the submatrix $\bar{A}^*$ is obtained from selecting $|\calI^{*'}|+1$ number of linearly independent rows from the matrix $\tilde{A}^*$ (note that for square matrix $\bar{A}^*$, full row rank implies full rank of $|\calI^{*'}|+1$). In the following claim, we show that there exists such a square submatrix $\bar{A}^*$ which contains the last row $[(\bm{e}^{|\calI^{*'}|})^\top,  0]$.
\begin{claim}\label{claim:FullRank}
There exists a square submatrix $\bar{A}^*$ of $\tilde{A}^*$ {defined in \Cref{eqn:Ma}} such that $\bar{A}^*$ is of full rank (of rank $|\calI^{*'}|+1$) and $\bar{A}^*$ contains the last row $[(\bm{e}^{|\calI^{*'}|})^\top,  0]$.
\end{claim}
Therefore, we know that the square matrix
\begin{equation}\label{eqn:Ma1}
\bar{A}^*=\begin{bmatrix}
A^\top_{\calI^{*'}, \calJ^{*'}} & -\bm{e}^{|\calJ^{*'}|} \\
(\bm{e}^{|\calI^{*'}|})^\top & 0
\end{bmatrix}
\end{equation}
is of full rank, and it holds that 
\begin{equation}\label{eqn:Standsystem5}
\bar{A}^* \begin{bmatrix}
    \bx^*_{\calI^{*'}}\\
    \mu^*
\end{bmatrix}=\begin{bmatrix}
    \bm{0}\\
    1
\end{bmatrix}
\end{equation}
which also implies that $(\bx^*_{\calI^{*'}}, \mu^*)$ is the unique solution to the linear system \Cref{eqn:Standsystem5} with $\bx^*_{(\calI^{*'})^c}=\bm{0}$. We also have $|\calI^{*'}|=|\calJ^{*'}|$. Our proof is thus completed.

\subsubsection{Proof of \cref{claim:FullRank}}
We show that any full rank square submatrix (with rank $|\calI^{*'}|+1$) of $\tilde{A}^*$ can be converted into a full rank square submatrix $\bar{A}^*$ (with rank $|\calI^{*'}|+1$) containing the last row $[(\bm{e}^{|\calI^{*'}|})^\top,  0]$. Suppose that we have such a full rank square submatrix $A'$ consisting of $|\calI^{*'}|+1$ rows from $[A^\top_{\calI^{*'}, \wh{\calI}^c_2}, -\bm{e}^{|\wh{\calI}^c_2|}]$. We denote by the rows of $A'$ as $\bm{a}_1,\dots,\bm{a}_{d+1}$, where we set $d=|\calI^{*'}|$. Then, since the matrix $\tilde{A}^*$ is of rank $d+1$, we know that there exists a set of constants $\alpha_1, \dots, \alpha_{d+1}$ such that the last row can be expressed as
\begin{equation}\label{eqn:FullRank1}
[(\bm{e}^{|\calI^{*'}|})^\top,  0] = \sum_{i=1}^{d+1}\alpha_i\cdot\bm{a}_i.
\end{equation}
From comparing the last elements of the rows in \Cref{eqn:FullRank1}, we know that
\begin{equation}\label{eqn:FullRank2}
    0 = \sum_{i=1}^{d+1} \alpha_i.
\end{equation}
Now suppose that $\alpha_{d+1}\neq0$, we know that $\alpha_{d+1}=-\sum_{i=1}^d \alpha_i$. Therefore, we have
\begin{equation}\label{eqn:FullRank3}
\bm{a}_{d+1}=\frac{1}{\alpha_{d+1}}\cdot \left( [(\bm{e}^{|\calI^{*'}|})^\top,  0]+\sum_{i=1}^d \alpha_i\cdot\bm{a}_i \right).
\end{equation}
Since every row of $\tilde{A}^*$ can be expressed as the linear combination of $\bm{a}_1,\dots,\bm{a}_{d+1}$, it can also be expressed as the linear combination of $[(\bm{e}^{|\calI^{*'}|})^\top,  0]$ and $\bm{a}_1,\dots,\bm{a}_d$ by noting \Cref{eqn:FullRank3}. Therefore, we know that a square submatrix $\bar{A}^*$ consisting of $[(\bm{e}^{|\calI^{*'}|})^\top,  0]$ and $\bm{a}_1,\dots,\bm{a}_d$ is of the same rank as $\tilde{A}^*$, which is $d+1$, and thus $\bar{A}^*$ is of full row rank, thus full rank.

\subsection{Proof of \Cref{thm:Infibasis2}}\label{app:infibasis2}
We now condition on the event that 
\begin{equation}\label{def:event}
\mathcal{E}=\left\{\left| \wh{A}_{i,j}-A_{i,j}\right|\leq \Rad(N/m, \eps/m), ~~\forall i\in[m_1], j\in[m_2]  \right\},
\end{equation}
where $\wh{A}_{i,j}$ is constructed using $N/m$ number of i.i.d samples with mean $A_{i,j}$, $m=m_1m_2$, and $\Rad(N,\eps)\triangleq \sqrt{\frac{\log(1/\eps)}{2N}}$.
According to Hoeffding's inequality, we know that this event $\mathcal{E}$ happens with probability at least $1-\eps$. 

Given a fixed sample set $\calH$ {of size $N$} such that event $\mathcal{E}$ happens, we first bound the gap between $\VI_{\calI}$ and $\wh{V}^{\mathrm{Prime}}_{\calI}$, for any set $\calI$. The result is formalized in the following claim.
\begin{claim}\label{claim:Bound1}
Conditioned on the event $\mathcal{E}$ \Cref{def:event} happens, for all set $\calI\in[m_1]$, it holds that
\begin{equation}\label{eqn:121401}
\left|\VI_{\calI} - \wh{V}^{\mathrm{Prime}}_{\calI} \right| \leq 
\Rad(N/m, \eps/m).
\end{equation}
\end{claim}

We can also obtain the bound of the gap between the {the objective values of the} dual LPs, $\VD_{\calI, \calJ}$ and $\wh{V}^{\mathrm{Dual}}_{\calI, \calJ}$, for any sets $\calI\in[m_1]$ and $\calJ\in[m_2]$. The bound is formalized in the following claim. 
\begin{claim}\label{claim:Bound2}
Conditioned on the event $\mathcal{E}$ \Cref{def:event} happens, for any set $\calI\in[m_1]$ and $\calJ\in[m_2]$, it holds that 
\begin{equation}\label{eqn:121412}
    \left| \VD_{\calI, \calJ}- \wh{V}^{\mathrm{Dual}}_{\calI, \calJ} \right| \leq \Rad(N/m,\eps/m).
\end{equation}
\end{claim}

Equipped with \Cref{claim:Bound1} and \Cref{claim:Bound2}, we now proceed our proof.
We first show that the output of \Cref{alg:Idenbasis}, denoted as $\wh{\calI}^*$ and $\wh{\calJ}^*$, is indeed an optimal basis to the estimated LP $\wh{V}^{\mathrm{Prime}}$ and its dual $\wh{V}^{\mathrm{Dual}}$. We then show that as long as $N\geq N_0=\left(\frac{1}{\sigma_0^2}+\frac{1}{\delta^2} \right)\cdot 2m\log(2m/\epsilon)$, an optimal basis to $\wh{V}^{\mathrm{Prime}}$ is an optimal basis to $\VI$. 

Following the procedure in \Cref{alg:Idenbasis}, we identify an index set $\wh{\calI}^*\subset[m_1]$ as the support for $\bx$, and we set $x_i=0$ for each $i\in\wh{\calI}^{*c}$. Note that we cannot further delete one more element from the set $\wh{\calI}^*$ without changing the objective value according to our algorithm design. 
Thus, we know that our \Cref{alg:Idenbasis} correctly identifies an index set $\wh{\calI}^*$, which is an optimal basis to the estimated LP $\wh{V}^{\mathrm{Prime}}$.

Now given that $\wh{\calI}^*$ is an optimal basis to the LP $\wh{V}^{\mathrm{Prime}}$, we know that 
\begin{equation}\label{eqn:011505}
\wh{V}^{\mathrm{Prime}}=\wh{V}^{\mathrm{Prime}}_{\wh{\calI}^*}= \min_{\bx_{\wh{\calI}^*}\succeq \mathbf{0}, \mu\in\mathbb{R}}~ \mu ~~~\mbox{s.t.}~ \mu\cdot\bm{e}^{m_2}\succeq \wh{A}_{\wh{\calI}^*, :}^\top\bx_{\wh{\calI}^*}, ~(\bm{e}^{|\wh{\calI}^*|})^\top\bx_{\wh{\calI}^*}=1.
\end{equation}
Note that in the formulation of \Cref{eqn:011505}, we simply discard the columns of the constraint matrix in the index set $\wh{\calI}^*$. Thus, if we denote by $\wh{\bx}^*$ the optimal solution to $\wh{V}^{\mathrm{Prime}}$ corresponding to the basis $\wh{\calI}^*$ with $\wh{\bx}^*_{\calI^*}=\mathbf{0}$, then
one optimal solution to \Cref{eqn:011505} will be $\wh{\bx}^*_{\wh{\calI}^*}$.
The dual of \Cref{eqn:011505} is 
\begin{equation}\label{eqn:011506}
\wh{V}^{\mathrm{Dual}}_{\wh{\calI}^*}=\max_{\nu\in\mathbb{R}, \by\succeq0} ~\nu  ~~~\mbox{s.t.}~\nu\cdot\bm{e}^{|\wh{\calI}^*|} \leq \wh{A}_{\wh{\calI}^*, :}\by, ~(\bm{e}^{m_2})^\top\by=1.
\end{equation}
We now show that \Cref{alg:Idenbasis} correctly identifies an optimal basis of the estimated LP $\wh{V}^{\mathrm{Dual}}_{\wh{\calI}^*}$, denoted by $\wh{\calJ}^*$. Note that following the procedure of \Cref{alg:Idenbasis}, we identify an index set $\mathcal{J}\subset[m_2]$, where we set $y_j=0$ for each $j\in\mathcal{J}^c$. Similarly, we know that we cannot further delete one more $j$ from the set $\mathcal{J}$ without changing the objective value according to our algorithm design.
Thus, we know that our \Cref{alg:Idenbasis} correctly identifies the index set $\wh{\calJ}^*$, which is an optimal basis to the estimated LP $\wh{V}^{\mathrm{Dual}}_{\wh{\calI}^*}$, conditional on the event $\mathcal{E}$ in \Cref{def:event} happens.

It remains to show that the index sets $\wh{\calI}^*$ and $\wh{\calJ}^*$ identified by our \Cref{alg:Idenbasis} satisfy the conditions in \Cref{lem:Basis}, for the estimated LP $\wh{V}^{\mathrm{Prime}}$ and its dual $\wh{V}^{\mathrm{Dual}}$. {Denote by $(\wh{\bx}^*, \wh{\by}^*)$ an optimal primal-dual solution to $\wh{V}^{\mathrm{Prime}}$, corresponding to the optimal basis $\wh{\calI}^*$ and $\wh{\calJ}^*$.}
Since we know that $\wh{\by}^*_{\wh{\calJ}^*}>\bm{0}$, from the complementary slackness condition, the corresponding constraints in $\wh{V}^{\mathrm{Prime}}$ must be binding, i.e., $\wh{A}^\top_{:, \wh{\calJ}^*}\wh{\bx}^*=\wh{\mu}^*\cdot\bm{e}$, where we denote by $(\wh{\bx}^*, \wh{\mu}^*)$ the optimal solution to $\wh{V}^{\mathrm{Prime}}$ corresponding to the optimal basis $\wh{\calI}^*$ and $\wh{\calJ}^*$. Moreover, note that $\wh{\bx}^*_{(\wh{\calI}^*)^{c}}=\bm{0}$, we must have
\[
\wh{A}^\top_{:, \wh{\calJ}^*}\wh{\bx}^* = \wh{A}^\top_{\wh{\calI}^*, \wh{\calJ}^*}\wh{\bx}^*_{\wh{\calI}^*} = \wh{\mu}^*\cdot \bm{e}^{|\wh{\calJ}^*|},
\]
which proves the equations in condition \Cref{eqn:Lsystem1}. 

Next, we show that $|\wh{\calI}^*|=|\wh{\calJ}^*|$, and $\begin{bmatrix}
{A}^\top_{\wh{\calI}^*, \wh{\calJ}^*} & -\bm{e}^{|\wh{\calJ}^*|} \\
(\bm{e}^{|\wh{\calI}^*|})^\top & 0
\end{bmatrix}$ is a non-singular matrix. We first show that for any two index set $\calI$ and $\calJ$, the smallest singular value of the matrix $\begin{bmatrix}
{A}^\top_{\calI, \calJ} & -\bm{e}^{|\calJ|} \\
(\bm{e}^{|\calI|})^\top & 0
\end{bmatrix}$ is either $0$ or be greater than $|\calI||\calJ|\cdot\Rad(N/m, \eps/m)$, which implies that the operation of checking $\wh{\sigma}_{\calI, \calJ'}> |\calI||\calJ'|\cdot\Rad(N', \eps/m)$ in line 14 of \Cref{alg:Idenbasis} is equivalent to checking whether the matrix $\begin{bmatrix}
{A}^\top_{\calI, \calJ'} & -\bm{e}^{|\calJ'|} \\
(\bm{e}^{|\calI|})^\top & 0
\end{bmatrix}$ is full-rank or not. 
Since we have $N\geq N_0=\left(\frac{1}{\sigma_0^2}+\frac{1}{\delta^2} \right)\cdot 2m\log(2m/\epsilon)$, we have that
\begin{equation}\label{eqn:100401}
\sigma_0\geq \Rad(N/m, \eps/m).
\end{equation}
Denote by $\sigma_{\calI, \calJ}$ the smallest singular value of the matrix $\begin{bmatrix}
{A}^\top_{\calI, \calJ} & -\bm{e}^{|\calJ|} \\
(\bm{e}^{|\calI|})^\top & 0
\end{bmatrix}$ and denote by $\wh{\sigma}_{\calI, \calJ}$ the smallest singular value of the matrix $\begin{bmatrix}
\wh{A}^\top_{\calI, \calJ} & -\bm{e}^{|\calJ|} \\
(\bm{e}^{|\calI|})^\top & 0
\end{bmatrix}$. Following Theorem 1 of \cite{stewart1998perturbation}, we know that the difference of the corresponding singular values of the matrices $\begin{bmatrix}
{A}^\top_{\calI, \calJ} & -\bm{e}^{|\calJ|} \\
(\bm{e}^{|\calI|})^\top & 0
\end{bmatrix}$ and $\begin{bmatrix}
\wh{A}^\top_{\calI, \calJ} & -\bm{e}^{|\calJ|} \\
(\bm{e}^{|\calI|})^\top & 0
\end{bmatrix}$ is upper bounded by $\ell_2$-norm of their difference, i.e.,
\[
|\wh{\sigma}_{\calI, \calJ}-\sigma_{\calI, \calJ}|\leq\left\|\begin{bmatrix}
\wh{A}^\top_{\calI, \calJ} & -\bm{e}^{|\calJ|} \\
(\bm{e}^{|\calI|})^\top & 0
\end{bmatrix}-\begin{bmatrix}
{A}^\top_{\calI, \calJ} & -\bm{e}^{|\calJ|} \\
(\bm{e}^{|\calI|})^\top & 0
\end{bmatrix} \right\|_2 < |\calI||\calJ|\cdot\Rad(N/m, \eps/m).
\]
Therefore, $\sigma_{\calI, \calJ}=0$ implies that $\wh{\sigma}_{\calI, \calJ}\leq |{\calI}||\calJ|\cdot\Rad(N/m, \eps/m)$. On the other hand, if $\sigma_{\calI, \calJ}>0$, then following \Cref{eqn:100401} and the definition of $\sigma_0$ in \Cref{def:sigma0}, it holds that
\[
\sigma_{\calI, \calJ}\geq2|\calI||\calJ|\cdot\sigma_0\geq 2|\calI||\calJ|\cdot\Rad(N/m, \eps/m).
\]
As a result, we must have that 
\begin{equation}\label{eqn:100103}
\wh{\sigma}_{\calI, \calJ} \geq \sigma_{\calI, \calJ} - |\sigma_{\calI, \calJ}-\wh{\sigma}_{\calI, \calJ}| > (2|\calI||\calJ|-|\calI||\calJ|)\cdot\Rad(N/m, \eps/m)\geq |\calI||\calJ|\cdot\Rad(N/m, \eps/m)
\end{equation} 
In this way, we guarantee that $\begin{bmatrix}
{A}^\top_{\wh{\calI}^*, \wh{\calJ}^*} & -\bm{e}^{|\wh{\calJ}^*|} \\
(\bm{e}^{|\wh{\calI}^*|})^\top & 0
\end{bmatrix}$ is a full rank matrix.

We now prove that $|\wh{\calI}^*| = |\wh{\calJ}^*|$. We first show that $m_2=|\calJ_0| \geq |\wh{\calI}^*|$. Now suppose that $m_2=|\calJ_0| < |\wh{\calI}^*|$ and we show contradiction. We consider the LP given by
\begin{equation}\label{lp:100401}
    \VI_{\wh{\calI}^*}=\min_{\bx\succeq0, \mu\in\mathbb{R}}~ \mu ~~~\mbox{s.t.}~ \mu\cdot\bm{e}^{m_2}\succeq \wh{A}^\top\bx, ~\bx_{(\wh{\calI}^*)^c}=\bm{0}, ~(\bm{e}^{m_1})^\top\bx=1.
\end{equation}
Now, we can apply \Cref{lem:Basis} to the above LP in \Cref{lp:100401}, which is equivalent to considering the game matrix given by $\wh{A}_{\wh{\calI}^*, \calJ_0}$. Following \Cref{lem:Basis}, we know that there exists a subset $\calI'\subset\wh{\calI}^*$, a subset $\calJ'\subset\calJ_0$ and an optimal solution to the LP in \Cref{lp:100401}, denoted by $(\bx', \mu')$, such that 
\[
|\calI'|=|\calJ'|\leq \min\{ |\wh{\calI}^*|, |\calJ_0| \} \text{~and~}\bx'_{(\calI')^c} = \bm{0}.
\]
Note that we suppose $m_2=|\calJ_0| < |\wh{\calI}^*|$, there must exists an index $i'$ such that $i'\in\wh{\calI}^*$ and $i'\notin\calI'$. Also, from the optimality of $(\bx', \mu')$, we know that it must hold that
\begin{equation}\label{lp:100402}
\wh{V}^{\mathrm{Prime}}=\wh{V}^{\mathrm{Prime}}_{\wh{\calI}^*} = \wh{V}^{\mathrm{Prime}}_{\wh{\calI}'} = \min_{\bx\succeq0, \mu\in\mathbb{R}}~ \mu ~~~\mbox{s.t.}~ \mu\cdot\bm{e}^{m_2}\succeq \wh{A}^\top\bx, ~\bx_{(\wh{\calI}')^c}=\bm{0}, ~(\bm{e}^{m_1})^\top\bx=1.
\end{equation}
Therefore, we know that the element $i'$ would have been dropped in \Cref{alg:Idenbasis}, which contradicts with $i'\in\wh{\calI}^*$. In this way, we show that $m_2=|\calJ_0|\geq |\wh{\calI}^*|$. We further note that from the stopping condition in line 17 of \Cref{alg:Idenbasis}, we shrinkage the size of the set $\calJ$ from $\calJ=\calJ_0$, and we break the loop as long as $|\calJ|=|\wh{\calI}^*|$. In this way, we know that  $|\wh{\calJ}^*| \geq |\wh{\calI}^*|$. 

It only remains to show that $|\wh{\calJ}^*| \leq |\wh{\calI}^*|$. Now suppose that $|\wh{\calJ}^*| > |\wh{\calI}^*|$ and we show contradiction. We consider the LP given by
\begin{equation}\label{lp:100101}
\wh{V}^{\mathrm{Dual}}_{\wh{\calI}^*, \wh{\calJ}^*}=\max_{\nu\in\mathbb{R}, \by\succeq0} ~\nu  ~~~\mbox{s.t.}~\nu\cdot\bm{e}^{|\wh{\calI}^*|} \leq \wh{A}_{\wh{\calI}^*, :}\by, ~\by_{(\wh{\calJ}^{*})^c}=\bm{0}, ~(\bm{e}^{m_2})^\top\by=1.
\end{equation}
Now, we can apply \Cref{lem:Basis} to the above LP in \Cref{lp:100101}, which is equivalent to considering the game matrix given by $\wh{A}_{\wh{\calI}^*, \wh{\calJ}^*}$. Following \Cref{lem:Basis}, we know that there exists a subset $\calI''\subset\wh{\calI}^*$, a subset $\calJ''\subset\wh{\calJ}^*$ and an optimal solution to the LP in \Cref{lp:100101}, denoted by $(\by'', \nu'')$, such that 
\[
|\calI''|=|\calJ''|\leq \min\{ |\wh{\calI}^*|, |\wh{\calJ}^*| \} \text{~and~}\by''_{(\calJ'')^c} = \bm{0}.
\]
Moreover, we know that the matrix $\begin{bmatrix}
{A}^\top_{\calI'', \calJ''} & -\bm{e}^{|\calJ''|} \\
(\bm{e}^{|\calI''|})^\top & 0
\end{bmatrix}$ is a full rank matrix from \Cref{lem:Basis}.
Note that we suppose $|\wh{\calJ}^*| > |\wh{\calI}^*|$, there must exists an index $j''$ such that $j''\in\wh{\calJ}^*$ and $j''\notin\calJ''$. Also, from the optimality of $(\by'', \nu'')$, we know that it must hold that
\begin{equation}\label{lp:100102}
\wh{V}^{\mathrm{Prime}}=\wh{V}^{\mathrm{Dual}}_{\wh{\calI}^*, \wh{\calJ}^*} = \wh{V}^{\mathrm{Dual}}_{\wh{\calI}^*, \calJ''} = \max_{\nu\in\mathbb{R}, \by\succeq0} ~\nu  ~~~\mbox{s.t.}~\nu\cdot\bm{e}^{|\wh{\calI}^*|} \leq \wh{A}_{\wh{\calI}^*, :}\by, ~\by_{(\calJ'')^c}=\bm{0}, ~(\bm{e}^{m_2})^\top\by=1.
\end{equation}
Further, since the matrix $\begin{bmatrix}
{A}^\top_{\calI'', \calJ''} & -\bm{e}^{|\calJ''|} \\
(\bm{e}^{|\calI''|})^\top & 0
\end{bmatrix}$ is non-singular and of rank $|\calJ''|+1$, we know that the matrix $\begin{bmatrix}
{A}^\top_{\wh{\calI}^*, \calJ''} & -\bm{e}^{|\calJ''|} \\
(\bm{e}^{|\wh{\calI}^*|})^\top & 0
\end{bmatrix}$, which contains the matrix  $\begin{bmatrix}
{A}^\top_{\calI'', \calJ''} & -\bm{e}^{|\calJ''|} \\
(\bm{e}^{|\calI''|})^\top & 0
\end{bmatrix}$ as a sub-matrix, is of a rank at least $|\calJ''|+1$. This would imply that the matrix $\begin{bmatrix}
{A}^\top_{\wh{\calI}^*, \calJ''} & -\bm{e}^{|\calJ''|} \\
(\bm{e}^{|\wh{\calI}^*|})^\top & 0
\end{bmatrix}$ is of full rank since it has only $|\calJ''|+1$ rows. From \Cref{eqn:100103}, we know that $\wh{\sigma}_{\wh{\calI}^*, \calJ''}>|\wh{\calI}^*|\cdot\Rad(N/m, \eps/m)$. Together with \Cref{lp:100102}, we know that the element $j''$ would have been dropped in \Cref{alg:Idenbasis}, which contradicts with $j''\in\wh{\calJ}^*$. In this way, we show that $|\wh{\calJ}^*|\leq |\wh{\calI}^*|$. Thus, we conclude that $|\wh{\calJ}^*|= |\wh{\calI}^*|$. 

Finally, it remains to show that $\wh{\calI}^*$ and $\wh{\calJ}^*$ is also an optimal basis to the original LP $\VI$ and its dual $\VD$ as long as $N\geq\frac{2m}{\delta^2} \cdot \log(2m/\epsilon)$.
Suppose that $\wh{\calI}^*$ is not an optimal basis to the original LP $\VI$, then following the definition of $\delta=\min\{\delta_1, \delta_2\}$ with $\delta_1$ and $\delta_2$ given in Definition \ref{def:delta}, we must have
\[
\VI_{\wh{\calI}^*} > \VI + \delta.
\]
On the other hand, from \Cref{eqn:121401} in \Cref{claim:Bound1}, we know that
\[
|\VI_{\wh{\calI}^*}-\wh{V}_{\wh{\calI}^*}^{\mathrm{Prime}}| \leq \Rad(N/m, \eps/m),
\]
and
\[
|\VI-\wh{V}^{\mathrm{Prime}}| \leq\Rad(N/m, \eps/m).
\]
However, we know that
\[
\wh{V}^{\mathrm{Prime}}_{\wh{\calI}^*} = \wh{V}^{\mathrm{Prime}}
\]
since $\wh{\calI}^*$ is an optimal basis to $\wh{V}^{\mathrm{Prime}}$. Therefore, we conclude that we have
\begin{equation}\label{eqn:012301}
\VI+2\Rad(N/m,\eps/m)\geq \VI_{\wh{\calI}^*}>\VI+\delta.
\end{equation}
However, the inequality \Cref{eqn:012301} contradicts with the condition that $\Rad(N/m, \eps/m)\leq\delta/2$ according to the definition of $\Rad(N/m,\eps/m)$. Therefore, we know that $\wh{\calI}^*$ is also an optimal basis to $\VI$. 

We then show that $\wh{\calJ}^*$ is an optimal basis to the dual LP $\VD_{\wh{\calI}^*}$. Suppose that $\wh{\calJ}^*$ is not an optimal basis to the original dual LP $\VD_{\wh{\calI}^*}$, then following the definition of $\delta=\min\{\delta_1, \delta_2\}$ with $\delta_1$ and $\delta_2$ given in Definition \ref{def:delta}, we must have
\[
\VD_{\wh{\calI}^*, \wh{\calJ}^*} < \VD_{\wh{\calI}^*} - \delta.
\]
On the other hand, from \Cref{eqn:121412} in \Cref{claim:Bound2}, we know that 
\[
|\VD_{\wh{\calI}^*, \wh{\calJ}^*} - \wh{V}^{\mathrm{Dual}}_{\wh{\calI}^*, \wh{\calJ}^*}| \leq \Rad(N/m, \eps/m),
\]
and
\[
|\VD_{\wh{\calI}^*} - \wh{V}^{\mathrm{Dual}}_{\wh{\calI}^*}| \leq \Rad(N/m, \eps/m).
\]
However, we know that 
\[
\wh{V}^{\mathrm{Dual}}_{\wh{\calI}^*, \wh{\calJ}^*}=\wh{V}^{\mathrm{Dual}}_{\wh{\calI}^*}
\]
since $\wh{\calJ}^*$ is an optimal basis to $\wh{V}^{\mathrm{Dual}}_{\wh{\calI}^*}$. Therefore, we conclude that we have
\begin{equation}\label{eqn:012302}
\VD_{\wh{\calI}^*} - 2\Rad(N/m, \eps/m) \leq \VD_{\wh{\calI}^*, \wh{\calJ}^*} \leq \VD_{\wh{\calI}^*} - \delta.
\end{equation}
However, the inequality \Cref{eqn:012302} contradicts with the condition that $\Rad(N/m, \eps/m)\leq\delta/2$. Therefore, we know that $\wh{\calJ}^*$ is also an optimal basis to $\VD_{\wh{\calI}^*}$. We conclude that $\wh{\calI}^*$ and $\wh{\calJ}^*$ is an optimal basis to the original LP $\VI$ and its dual $\VD$, i.e., $\wh{\calI}^*$ and $\wh{\calJ}^*$ satisfy the conditions described in \Cref{lem:Basis} for the original LP $\VI$ and $\VD$.

\subsubsection{Proof of \Cref{claim:Bound1}}
Denote by $\bx^*$ an optimal solution to $\VI_{\calI}$. We now construct a feasible solution to $\wh{V}^{\mathrm{Prime}}_{\calI}$ based on $\bx^*$. Note that conditional on the event $\mathcal{E}$ happens, we have that 
\begin{equation}\label{eqn:121402}
    \wh{A}^\top\bx^*\preceq A^\top\bx^*+ \|\bx^*\|_1\cdot\Rad(N/m, \eps/m)\cdot \bm{e}^{m_2} \preceq \mu^*\cdot\bm{e}^{m_2} +\cdot\Rad(N/m, \eps/m)\cdot \bm{e}^{m_2}.
\end{equation}
Therefore, we know that $(\bx^*,\mu^*+\Rad(N/m, \eps/m))$ is a feasible solution to the LP $\hatVII$ defined in \Cref{eqn:Izero}, meaning that
\begin{equation}\label{eqn:011401}
\wh{V}^{\mathrm{Prime}}_{\calI}\leq \mu^* + \Rad(N/m, \eps/m) = \VI_{\calI} + \Rad(N/m, \eps/m).
\end{equation}
On the other hand, denote by $\wh{\bx}^*$ and $\wh{\mu}^*$ certain optimal solution to $\wh{V}^{\mathrm{Prime}}_{\calI}$. We know that
\begin{equation}\label{eqn:011402}
    A^\top\wh{\bx}^* \preceq \wh{A}^\top\wh{\bx}^* + \|\wh{\bx}^*\|_1\cdot\Rad(N/m, \eps/m)\cdot \bm{e}^{m_2} \preceq \wh{\mu}^*\cdot\bm{e}^{m_2} + \Rad(N/m,\eps/m)\cdot \bm{e}^{m_2}.
\end{equation}
Therefore, we know that $(\wh{\bx}^*,\wh{\mu}^*+\Rad(N/m, \eps/m))$ is a feasible solution to the LP $\VI_{\calI}$ , meaning that
\begin{equation}\label{eqn:011403}
    \VI_{\calI}\leq \wh{\mu}^* + \Rad(N/m, \eps/m)=\wh{V}^{\mathrm{Prime}}_{\calI} + \Rad(N/m, \eps/m). 
\end{equation}
Our proof is completed from \Cref{eqn:011401} and \Cref{eqn:011403}. 

\subsubsection{Proof of \Cref{claim:Bound2}}
Denote by $(\by^*, \nu^*)$ an optimal solution to $\VD_{\calI, \calJ}$. Similar to the proof for \Cref{claim:Bound1}, we now construct a feasible solution to $\wh{V}^{\mathrm{Dual}}_{\calI, \calJ}$ based on $\by^*$. Note that conditional on the event $\mathcal{E}$ defined in \Cref{def:event} happens, we have that 
\begin{equation}\label{eqn:121408}
    \wh{A}_{\calI, :}\by^* \succeq A_{\calI, :}\by^* - \|\by^*\|_1\cdot\Rad(N/m, \eps/m) \cdot\bm{e}^{m_1}\succeq \nu^* - \Rad(N/m, \eps/m) \cdot\bm{e}^{m_1}.
\end{equation}
Therefore, $(\by^*,\nu^*-\Rad(N/m, \eps/m))$ forms a feasible solution to $\wh{V}^{\mathrm{Dual}}_{\calI, \calJ}$.
Thus, we know that
\begin{equation}\label{eqn:011404}
    \wh{V}^{\mathrm{Dual}}_{\calI, \calJ} \geq \nu^*-\Rad(N/m, \eps/m)=\VD_{\calI, \calJ} - \Rad(N/m, \eps/m). 
\end{equation}
On the other hand, denote by $\wh{\by}^*$ and $\wh{\nu}^*$ an optimal solution to $\wh{V}^{\mathrm{Dual}}_{\calI, \calJ}$. We know that
\begin{equation}\label{eqn:011405}
    A_{\calI, :}\wh{\by}^*\succeq \wh{A}_{\calI, :}\wh{\by}^*-\|\wh{\by}^*\|_1\cdot\Rad(N/m, \eps/m)\cdot\bm{e}^{m_1} \succeq \wh{\nu}^*- \Rad(N/m, \eps/m)\cdot\bm{e}^{m_1}.
\end{equation}
Therefore, $(\wh{\by}^*, \wh{\nu}^*-\Rad(N/m, \eps/m))$ forms a feasible solution to $\VD_{\calI, \calJ}$ and
\begin{equation}\label{eqn:011406}
    \VD_{\calI, \calJ} \geq \wh{\nu}^* - \Rad(N/m, \eps/m) = \wh{V}_{\calI, \calJ}^{\mathrm{Dual}} - \Rad(N/m, \eps/m). 
\end{equation}
Our proof is thus completed from \Cref{eqn:011404} and \Cref{eqn:011406}.

\subsection{Proof of \Cref{thm:FiniteGap}}\label{app:instance-independent-id}
Consider a fixed $N$ and we condition on the following high-probability event 
\begin{equation}\label{eqn:012201}
\mathcal{E}=\left\{\left| \wh{A}_{i,j}-A_{i,j}\right|\leq \Rad(N/m, \eps/m), ~~\forall i\in[m_1], j\in[m_2]  \right\},
\end{equation}
which happens with probability at least $1-\eps$. We denote by $\calI^N$ and $\calJ^N$ the output of \Cref{alg:Idenbasis}. From the implementation of \Cref{alg:Idenbasis}, we know that $\calI^N$ and $\calJ^N$ is an optimal basis to the LP $\wh{V}^{\mathrm{Prime}}$ and $\wh{V}^{\mathrm{Dual}}_{\calI^N}$. Therefore, we know that 
\[
\wh{V}^{\mathrm{Prime}} = \wh{V}^{\mathrm{Prime}}_{\calI^N} = \wh{V}_{\calI^N, \calJ^N}^{\mathrm{Dual}}.
\]
On the other hand, according to \Cref{claim:Bound1} and \Cref{claim:Bound2}, we know that
\[
|\VI-\wh{V}^{\mathrm{Prime}}|\leq\Rad(N/m, \eps/m)\text{~~and~~}|\VD_{\calI^N, \calJ^N} - \wh{V}^{\mathrm{Dual}}_{\calI^N, \calJ^N}| \leq\Rad(N/m, \eps/m).
\]
Therefore, we have
\[
|\VI-\VI_{\calI^N}|\leq2\cdot\Rad(N/m, \eps/m)\text{~~and~~} |\VI-\VD_{\calI^N, \calJ^N}|\leq 2\cdot\Rad(N/m, \eps/m),
\]
which completes our proof of \Cref{eqn:012310}. 

We now denote by $\wh{\bx}$ the optimal primal-dual solution to the LP $\wh{V}^{\mathrm{Prime}}$, corresponding to the optimal basis $\calI^N$ and $\calJ^N$. Therefore, it holds that
\begin{equation}\label{eqn:012101}
    \wh{\bx}_{(\calI^N)^c} = \bm{0},
\end{equation}
and $\wh{\bx}_{\calI^N}$ is the solution to
\begin{equation}\label{eqn:012102}
\begin{bmatrix}
\wh{A}^\top_{\calI^N, \calJ^N} & -\bm{e}^{|\calJ^N|} \\
(\bm{e}^{|\calI^N|})^\top & 0
\end{bmatrix}
\begin{bmatrix}
\wh{\bx}_{\calI^N}\\
\wh{\mu}
\end{bmatrix}=
\begin{bmatrix}
\bm{0}\\
1
\end{bmatrix}.
\end{equation}
We now apply the perturbation analysis of linear equation to bound the distance between $\wh{\bx}_{\calI^N}$ and $\bx^*_{\calI^N}(\calI^N, \calJ^N)$ defined in \Cref{eqn:IJOptQ}. We denote by $\wh{\sigma}_{\min}$ the smallest absolute value of the eigenvalues and $\wh{\sigma}_{\max}$ the largest absolute value of the eigenvalues of the matrix
\[
\begin{bmatrix}
\wh{A}^\top_{\calI^N, \calJ^N} & -\bm{e}^{|\calJ^N|} \\
(\bm{e}^{|\calI^N|})^\top & 0
\end{bmatrix}.
\]
Since $\calI^N$ and $\calJ^N$ is the output of \Cref{alg:Idenbasis}, we know that $\wh{\sigma}_{\min}>0$. 
The perturbation of the matrix is denoted as
\[
\Delta A=A_{\calI^N, \calJ^N} - \wh{A}_{\calI^N, \calJ^N}.
\]
Then, it holds that
\begin{equation}\label{eqn:012104}
\|\Delta A\|_2 \leq d\cdot \Rad(N/m, \eps/m).
\end{equation}
We now classify into two situations.

If $\Rad(N/m, \eps/m) > \frac{\wh{\sigma}_{\min}}{2d},$ we know that $d\cdot 
\wh{\kappa}\cdot\Rad(N/m,\eps/m)\geq \frac{1}{2}$,
meaning that the suboptimality gap is $\order(d\cdot \wh{\kappa}\Rad(N/m,\eps/m))$ since it is at most $\order(1)$.

Otherwise, if $\Rad(N/m, \eps/m) \leq \frac{\wh{\sigma}_{\min}}{2d},$ we know that
\begin{equation}\label{eqn:012105}
\|\Delta A\|_2 \leq d\cdot\Rad(N/m, \eps/m) \leq \frac{\wh{\sigma}_{\min}}{2}.
\end{equation}
Following standard perturbation analysis of linear equations (Theorem 1 of \citep{higham2002accuracy}), we can obtain that
\begin{equation}\label{eqn:012106}
\begin{aligned}
    \frac{\|(\wh{\bm{x}}_{\calI^N}, \wh{\mu})-(\bm{x}^*_{\calI^N}(\calI^N, \calJ^N), \mu^*(\calI^N, \calJ^N))\|_2}{\|(\wh{\bm{x}}_{\calI^N}, \wh{\mu})\|_2}
    &\leq \frac{\wh{\sigma}_{\max}/\wh{\sigma}_{\min}}{1-\frac{\|\Delta A\|_2}{\wh{\sigma}_{\min}}}\cdot \left( \frac{\|\Delta A\|_2}{\wh{\sigma}_{\max}} \right)\leq 2\cdot\|\Delta A\|_2/\wh{\sigma}_{\min}.
\end{aligned}
\end{equation}
Since $(\wh{\bx}_{\calI^N}, \wh{\mu})$ is the solution to the linear equations in \Cref{eqn:012102}, we know that 
\[
\|\wh{A}^\top_{\calI^N, \calJ^N}\wh{\bx}_{\calI^N}\|_{\infty} \leq d
\]
and thus $|\wh{\mu}|\leq d$. Therefore, we know that
\begin{equation}\label{eqn:012201}
    \|(\wh{\bx}_{\calI^N}, \wh{\mu})\|_2 \leq 2d. 
\end{equation}
Combining \Cref{eqn:012201} and \Cref{eqn:012106}, we can obtain that
\begin{equation}\label{eqn:012202}
\|(\wh{\bm{x}}_{\calI^N}, \wh{\mu})-(\bm{x}^*_{\calI^N}(\calI^N, \calJ^N), \mu^*(\calI^N, \calJ^N))\|_2\leq
\frac{4d^2}{\wh{\sigma}_{\min}}\cdot\Rad(N/m, \eps/m). 
\end{equation}
From the feasibility of $(\wh{\bx}_{\calI^N}, \wh{\mu})$ to the LP $\wh{V}^{\mathrm{Prime}}$, we know that
\begin{equation}\label{eqn:012204}
\wh{\mu}\cdot\bm{e}^{m_2}\succeq\wh{A}^\top \wh{\bx}, 
\end{equation}
which implies that 
\begin{equation}\label{eqn:012205}
\begin{aligned}
A^\top \bx^*(\calI^N, \calJ^N) &= \wh{A}^\top \wh{\bx}+\wh{A}^\top (\bx^*(\calI^N, \calJ^N)-\wh{\bx}) + \Delta A^\top \bx^*(\calI^N, \calJ^N)\\
&\preceq \left(\wh{\mu}+\|\wh{A}^\top (\bx^*(\calI^N, \calJ^N)-\wh{\bx})\|_{\infty} + \|\Delta A^\top\bx^*(\calI^N, \calJ^N)\|_{\infty}\right) \bm{e}^{m_2} \\
&\preceq \left(\wh{\mu}+\|\wh{A}^\top (\bx^*(\calI^N, \calJ^N)-\wh{\bx})\|_{2} + \|\Delta A^\top\bx^*(\calI^N, \calJ^N)\|_{\infty}\right) \bm{e}^{m_2} \\
&\preceq \left(\wh{\mu}+\frac{4d^2\wh{\sigma}_{\max}}{\wh{\sigma}_{\min}}\cdot\Rad(N/m,\eps/m) + \Rad(N/m,\eps/m)\right) \bm{e}^{m_2}.
\end{aligned}
\end{equation}
We conclude that we can set 
\[
\mu \triangleq \wh{\mu}+\Rad(N/m, \eps/m) + \frac{4d^2\wh{\sigma}_{\max}}{\wh{\sigma}_{\min}}\cdot\Rad(N/m, \eps/m).
\]
Therefore, we know that the sub-optimality gap for the solution $\bx^*(\calI^N, \calJ^N)$ is upper bounded by $\order\left(d^2\cdot \wh{\kappa}\cdot \Rad(N/m, \eps/m)\right)$.
\section{Missing Proofs in \Cref{sec: sample_complexity}}
\subsection{Proof of \Cref{prop:DoublingTrick}} Our proof is based on the proof of \Cref{thm:Infibasis2} and \Cref{thm:FiniteGap}, with the only key difference being to show that the doubling trick is able to check the non-singularity of the matrix. We prove the three arguments in \Cref{prop:DoublingTrick} one by one.

\noindent\textbf{Proof of Argument 1}: We show that when $N_1$ is large enough such that $N_1\geq\frac{2m\cdot\log(16m/\veps)}{\delta^2}$, with probability at least $1-\veps/4$, we have $\wh{\calI}^*$ and $\wh{\calJ}^*$ become an optimal basis to $\VI$ satisfying the conditions described in \Cref{lem:Basis}. For each round $k$ of calling \Cref{alg:Idenbasis} in the execution of \Cref{alg:Twophase} with the corresponding value of $N'$, we condition on the event that 
\begin{equation}\label{def:DTevent}
\mathcal{E}_k=\left\{\left| \wh{A}_{i,j}-A_{i,j}\right|\leq \Rad(N'/m, \eps/m), ~~\forall i\in[m_1], j\in[m_2]  \right\},
\end{equation}
where $\wh{A}_{i,j}$ is constructed using $N/m$ number of i.i.d samples with mean $A_{i,j}$, $m=m_1m_2$, and $\Rad(N'/m,\eps/m)\triangleq \sqrt{\frac{m\cdot\log(m/\eps)}{2N'}}$ with $\eps=\veps/(8k^2)$.
According to Hoeffding's inequality, we know that this event $\mathcal{E}_k$ happens with probability at least $1-\eps=1-\veps/(8k^2)$. We now condition on the event
\[
\mathcal{E}=\cup_{k}\mathcal{E}_k
\]
happens. From the union bound, we know that the event $\mathcal{E}$ happens with a probability at least $1-\veps/4$.

We first show that the final $\wh{\calI}^*$ and $\wh{\calJ}^*$ in \Cref{alg:Twophase}, is indeed an optimal basis to the estimated LP $\wh{V}^{\mathrm{Prime}}$ and its dual $\wh{V}^{\mathrm{Dual}}$. We then show that as long as $N_1\geq\frac{2m}{\delta^2} \cdot \log(8m/\veps)$, an optimal basis to $\wh{V}^{\mathrm{Prime}}$ is an optimal basis to $\VI$. 

We now focus on the final iteration of calling \Cref{alg:Idenbasis} in \Cref{alg:Twophase}, with the sample size to call \Cref{alg:Idenbasis} being $N'$. Following the procedure in \Cref{alg:Idenbasis}, we identify an index set $\wh{\calI}^*\subset[m_1]$, where we set $x_i=0$ for each $i\in\wh{\calI}^{*c}$. Note that we cannot further delete one more $i$ from the set $\wh{\calI}^*$ without changing the objective value according to our algorithm design. 
Thus, we know that our \Cref{alg:Idenbasis} correctly identifies an index set $\wh{\calI}^*$, which is an optimal basis to the estimated LP $\wh{V}^{\mathrm{Prime}}$.
Now given that $\wh{\calI}^*$ is an optimal basis to the LP $\wh{V}^{\mathrm{Prime}}$, we know that 
\begin{equation}\label{eqn:DT011505}
\wh{V}^{\mathrm{Prime}}=\wh{V}^{\mathrm{Prime}}_{\wh{\calI}^*}= \min_{\bx_{\wh{\calI}^*}\succeq 0, \mu\in\mathbb{R}}~ \mu ~~~\mbox{s.t.}~ \mu\cdot\bm{e}^{m_2}\succeq \wh{A}_{\wh{\calI}^*, :}^\top\bx_{\wh{\calI}^*}, ~(\bm{e}^{|\wh{\calI}^*|})^\top\bx_{\wh{\calI}^*}=1.
\end{equation}
Note that in the formulation of \Cref{eqn:DT011505}, we simply discard the columns of the constraint matrix not in the index set $\wh{\calI}^*$ (equivalently the rows of the matrix $A$ not in the index set $\calI^*$). Thus, if we denote by $\hat{\bx}^*$ the optimal solution to $\wh{V}^{\mathrm{Prime}}$ corresponding to the basis $\wh{\calI}^*$, then
one optimal solution to \Cref{eqn:DT011505} will just be $\hat{\bx}^*_{\wh{\calI}^*}$.
The dual of \Cref{eqn:DT011505} is
\begin{equation}\label{eqn:DT011506}
\wh{V}^{\mathrm{Dual}}_{\wh{\calI}^*}=\max_{\nu\in\R, \by\succeq0} ~\nu  ~~~\mbox{s.t.}~\nu\cdot\bm{e}^{|\wh{\calI}^*|} \leq \wh{A}_{\wh{\calI}^*, :}\by, ~(\bm{e}^{m_2})^\top\by=1.
\end{equation}
We now show that \Cref{alg:Idenbasis} correctly identifies the support for an optimal solution of the estimated LP $\wh{V}^{\mathrm{Dual}}_{\wh{\calI}^*}$, denoted by $\wh{\calJ}^*$. Note that following the procedure of \Cref{alg:Idenbasis}, we identify an index set $\wh{\mathcal{J}}\subset[m_2]$, where we set $y_j=0$ for each $j\in\wh{\mathcal{J}}^{*c}$. Note that we denote by $\hat{\bx}^*$ the optimal basic solution to $\wh{V}^{\mathrm{Prime}}$ that is supported on $\wh{\calI}^*$, which is also an optimal solution to the LP $\wh{V}^{\mathrm{Prime}}_{\wh{\calI}^*}$. We further denote by $\hat{\by}^*$ an optimal solution to $\wh{V}^{\mathrm{Dual}}_{\wh{\calI}^*}$ that is supported on the set $\wh{\calJ}^*$. From the complementary slackness condition, the corresponding constraints in $\wh{V}^{\mathrm{Prime}}_{\wh{\calI}^*}$ must be binding, i.e., $\wh{A}^\top_{:, \wh{\calJ}'}\hat{\bx}^*=\hat{\mu}^*\cdot\bm{e}$. Moreover, note that $\hat{\bx}^*_{(\wh{\calI}^*)^{c}}=\bm{0}$, we must have
\begin{equation}\label{eqn:082201}
\wh{A}^\top_{:, \wh{\calJ}'}\hat{\bx}^* = \wh{A}^\top_{\wh{\calI}^*, \wh{\calJ}'}\hat{\bx}^*_{\wh{\calI}^*} = \hat{\mu}^*\cdot \bm{e}^{|\wh{\calJ}'|}.
\end{equation}
Now, note that we have $|\wh{\calI}^*| = |\wh{\calJ}^*|$ and it is ensured that the matrix $\begin{bmatrix}
\wh{A}^\top_{\wh{\calI}^*, \wh{\calJ}^*} & -\bm{e}^{|\wh{\calJ}^*|} \\
(\bm{e}^{|\wh{\calI}^*|})^\top & 0
\end{bmatrix}$ is a non-singular matrix. We know that the solution $\hat{\bx}^*_{\wh{\calI}^*}$ is a unique solution to the lineaer system in \Cref{eqn:082201}. In this way, we show that $\wh{\calI}^*$ and $\wh{\calJ}^*$ forms an optimal basis to the estimated LP $\wh{V}^{\mathrm{Prime}}$, i.e., the conditions in \Cref{eqn:Lsystem1} and \Cref{lem:Basis} are satisfied by $\wh{\calI}^*$ and $\wh{\calJ}^*$ for the estimated LP $\wh{V}^{\mathrm{Prime}}$,

From the terminating condition in line 4 of \Cref{alg:Twophase}, we know that $|\wh{\calI}^*|=|\wh{\calJ}^*|$. We then show that the square matrix $\begin{bmatrix}
{A}^\top_{\wh{\calI}^*, \wh{\calJ}^*} & -\bm{e}^{|\wh{\calJ}^*|} \\
(\bm{e}^{|\wh{\calI}^*|})^\top & 0
\end{bmatrix}$ is a non-singular matrix. Suppose that the smallest singular value of the matrix $\begin{bmatrix}
{A}^\top_{\wh{\calI}^*, \wh{\calJ}^*} & -\bm{e}^{|\wh{\calJ}^*|} \\
(\bm{e}^{|\wh{\calI}^*|})^\top & 0
\end{bmatrix}$ is $0$, which implies singularity.
Denote by $\wh{\sigma}_{\wh{\calI}^*, \wh{\calJ}^*}$ the smallest singular value of the matrix $\begin{bmatrix}
\wh{A}^\top_{\wh{\calI}^*, \wh{\calJ}^*} & -\bm{e}^{|\wh{\calJ}^*|} \\
(\bm{e}^{|\wh{\calI}^*|})^\top & 0
\end{bmatrix}$. Following Theorem 1 of \cite{stewart1998perturbation}, we know that the difference of the corresponding singular values of the matrices $\begin{bmatrix}
{A}^\top_{\wh{\calI}^*, \wh{\calJ}^*} & -\bm{e}^{|\wh{\calJ}^*|} \\
(\bm{e}^{|\wh{\calI}^*|})^\top & 0
\end{bmatrix}$ and $\begin{bmatrix}
\wh{A}^\top_{\wh{\calI}^*, \wh{\calJ}^*} & -\bm{e}^{|\wh{\calJ}^*|} \\
(\bm{e}^{|\wh{\calI}^*|})^\top & 0
\end{bmatrix}$ is upper bounded by $\ell_2$-norm of their difference, i.e.,
\[
|\wh{\sigma}_{\wh{\calI}^*, \wh{\calJ}^*}-0|\leq\left\|\begin{bmatrix}
{A}^\top_{\wh{\calI}^*, \wh{\calJ}^*} & -\bm{e}^{|\wh{\calJ}^*|} \\
(\bm{e}^{|\wh{\calI}^*|})^\top & 0
\end{bmatrix}-\begin{bmatrix}
\wh{A}^\top_{\wh{\calI}^*, \wh{\calJ}^*} & -\bm{e}^{|\wh{\calJ}^*|} \\
(\bm{e}^{|\wh{\calI}^*|})^\top & 0
\end{bmatrix} \right\|_2 \leq |\wh{\calI}^*|\cdot |\wh{\calJ}^*|\cdot\Rad(N'/m, \eps/m).
\]
Therefore, we know that $\wh{\sigma}_{\wh{\calI}^*, \wh{\calJ}^*}\leq |\wh{\calI}^*|\cdot |\wh{\calJ}^*|\cdot\Rad(N', \eps/m)$, which contradicts the condition in line 14 of \Cref{alg:Idenbasis} that requires $\wh{\sigma}_{\wh{\calI}^*, \wh{\calJ}^*}> |\wh{\calI}^*|\cdot |\wh{\calJ}^*|\cdot\Rad(N', \eps/m)$. In this way, we guarantee that $\begin{bmatrix}
{A}^\top_{\wh{\calI}^*, \wh{\calJ}^*} & -\bm{e}^{|\wh{\calJ}^*|} \\
(\bm{e}^{|\wh{\calI}^*|})^\top & 0
\end{bmatrix}$ is a non-singular matrix.


Finally, it remains to show that $\wh{\calI}^*$ and $\wh{\calJ}^*$ is also an optimal basis to the original LP $\VI$ and its dual $\VD$ as long as $N_1\geq\frac{2m}{\delta^2} \cdot \log(8m/\veps)$.
Suppose that $\wh{\calI}^*$ is not an optimal basis to the original LP $\VI$, then following the definition of $\delta=\min\{\delta_1, \delta_2\}$ with $\delta_1$ and $\delta_2$ given in Definition \ref{def:delta}, we must have
\[
\VI_{\wh{\calI}^*} > \VI + \delta.
\]
On the other hand, since we condition on the event $\mathcal{E}$ happens, we must have
\[
\left| \wh{A}_{i,j}-A_{i,j}\right|\leq \Rad(N_1/m, \veps/(8m)), ~~\forall i\in[m_1], j\in[m_2].
\]
From \Cref{eqn:121401} in \Cref{claim:Bound1}, we know that
\[
|\VI_{\wh{\calI}^*}-\wh{V}_{\wh{\calI}^*}^{\mathrm{Prime}}| \leq \Rad(N'/m, \veps/(8m)),
\]
and
\[
|\VI-\wh{V}^{\mathrm{Prime}}| \leq\Rad(N'/m, \veps/(8m)).
\]
However, we know that
\[
\wh{V}^{\mathrm{Prime}}_{\wh{\calI}^*} = \wh{V}^{\mathrm{Prime}}
\]
since $\wh{\calI}^*$ is an optimal basis to $\wh{V}^{\mathrm{Prime}}$. Therefore, we conclude that we have
\begin{equation}\label{eqn:DT012301}
\VI+2\Rad(N'/m,\veps/(8m))\geq \VI_{\wh{\calI}^*}>\VI+\delta.
\end{equation}
However, the inequality \Cref{eqn:DT012301} contradicts with the condition that $\Rad(N'/m, \veps/(8m))\leq\delta/2$ according to the definition of $\Rad(N'/m,\eps/m)$ since $N'\geq N_1$. Therefore, we know that $\wh{\calI}^*$ is also an optimal basis to $\VI$. 

We then show that $\wh{\calJ}^*$ is an optimal basis to the dual LP $\VD_{\wh{\calI}^*}$. Suppose that $\wh{\calJ}^*$ is not an optimal basis to the original dual LP $\VD_{\wh{\calI}^*}$, then following the definition of $\delta=\min\{\delta_1, \delta_2\}$ with $\delta_1$ and $\delta_2$ given in Definition \ref{def:delta}, we must have
\[
\VD_{\wh{\calI}^*, \wh{\calJ}^*} < \VD_{\wh{\calI}^*} - \delta.
\]
On the other hand, since we condition on the event $\mathcal{E}$ happens, we must have
\[
\left| \wh{A}_{i,j}-A_{i,j}\right|\leq \Rad(N_1/m, \veps/(8m)), ~~\forall i\in[m_1], j\in[m_2].
\]
From \Cref{eqn:121412} in \Cref{claim:Bound2}, we know that 
\[
|\VD_{\wh{\calI}^*, \wh{\calJ}^*} - \wh{V}^{\mathrm{Dual}}_{\wh{\calI}^*, \wh{\calJ}^*}| \leq \Rad(N'/m, \veps/(8m)),
\]
and
\[
|\VD_{\wh{\calI}^*} - \wh{V}^{\mathrm{Dual}}_{\wh{\calI}^*}| \leq \Rad(N'/m, \veps/(8m)).
\]
However, we know that 
\[
\wh{V}^{\mathrm{Dual}}_{\wh{\calI}^*, \wh{\calJ}^*}=\wh{V}^{\mathrm{Dual}}_{\wh{\calI}^*}
\]
since $\wh{\calJ}^*$ is an optimal basis to $\wh{V}^{\mathrm{Dual}}_{\wh{\calI}^*}$. Therefore, we conclude that we have
\begin{equation}\label{eqn:012302}
\VD_{\wh{\calI}^*} - 2\Rad(N'/m, \veps/(8m)) \leq \VD_{\wh{\calI}^*, \wh{\calJ}^*} \leq \VD_{\wh{\calI}^*} - \delta.
\end{equation}
However, the inequality \Cref{eqn:012302} contradicts with the condition that $\Rad(N'/m, \veps/(8m))\leq\delta/2$ since $N'\geq N_1$. Therefore, we know that $\wh{\calJ}^*$ is also an optimal basis to $\VD_{\wh{\calI}^*}$. We conclude that $\wh{\calI}^*$ and $\wh{\calJ}^*$ is an optimal basis to the original LP $\VI$ and its dual $\VD$, i.e., $\wh{\calI}^*$ and $\wh{\calJ}^*$ satisfy the conditions described in \Cref{lem:Basis} for the original LP $\VI$ and $\VD$. Since we call \Cref{alg:Idenbasis} by a number of $K$ times in the execution of \Cref{alg:Twophase}, we take the union bound, and we know that with a probability at least $1-K\cdot\eps$, $\wh{\calI}^*$ and $\wh{\calJ}^*$ is an optimal basis to the original LP $\VI$ and its dual $\VD$, i.e., $\wh{\calI}^*$ and $\wh{\calJ}^*$ satisfy the conditions described in \Cref{lem:Basis} for the original LP $\VI$ and $\VD$, which completes our proof.

\noindent\textbf{Proof of Argument 2}: We now prove our second argument that for an arbitrary $N_1$, with probability at least $1-\veps/4$, the solution $\bx^*(\wh{\calI}^*, \wh{\calJ}^*)$ forms a feasible solution to $\VI$ 
with a sub-optimality gap bounded by $\order\left(\wh{\kappa}\cdot\Rad(N'/m, \eps/m)\right)$. Since the doubling trick in line 4 of \Cref{alg:Twophase} guarantees that we have $|\wh{\calI}^*| = |\wh{\calJ}^*|$, we can directly apply the result in \Cref{thm:FiniteGap}. Since we call \Cref{alg:Idenbasis} by a number of $K$ times in the execution of \Cref{alg:Twophase}, we take the union bound, and we know that with a probability at least $1-\veps/4$, the solution $\bx^*(\wh{\calI}^*, \wh{\calJ}^*)$ forms a feasible solution to $\VI$ 
with a sub-optimality gap bounded by $\order\left(\wh{\kappa}\cdot\Rad(N'/m, \eps/m)\right)$ with $\eps=\veps/(8K^2)$, which completes our proof of the second argument.

\noindent\textbf{Proof of Argument 3}: We now bound the value of $N_2$. Again, we condition on the event $\mathcal{E}$ happens and we set the value of $\eps =\veps/(8K^2)$ (the value of $K$ will be bounded later). 
We now show that when $N'\geq \frac{2m}{\sigma_0^2}\cdot\log(2m/\epsilon)$ for any two index set $\calI$ and $\calJ$, the smallest singular value of the matrix $\begin{bmatrix}
\wh{A}^\top_{\calI, \calJ} & -\bm{e}^{|\calJ|} \\
(\bm{e}^{|\calI|})^\top & 0
\end{bmatrix}$ is either $0$ or be greater than $|\calI||\calJ|\cdot\Rad(N'/m, \eps/m)$, which implies that the operation of checking $\wh{\sigma}_{\calI, \calJ'}> |\calI||\calJ'|\cdot\Rad(N', \eps/m)$ in line 14 of \Cref{alg:Idenbasis} is equivalent to checking whether the matrix $\begin{bmatrix}
\wh{A}^\top_{\calI, \calJ'} & -\bm{e}^{|\calJ'|} \\
(\bm{e}^{|\calI|})^\top & 0
\end{bmatrix}$ is singular or not. 
Rearranging the terms in $N'\geq \frac{2m}{\sigma_0^2}\cdot\log(2m/\epsilon)$, we know that
\[
\sigma_0\geq \Rad(N'/m, \eps/m)
\]
using the definition of  $\Rad(N'/m,\eps/m)$.
Denote by $\sigma_{\calI, \calJ}$ the smallest singular value of the matrix $\begin{bmatrix}
{A}^\top_{\calI, \calJ} & -\bm{e}^{|\calJ|} \\
(\bm{e}^{|\calI|})^\top & 0
\end{bmatrix}$ and denote by $\wh{\sigma}_{\calI, \calJ}$ the smallest singular value of the matrix $\begin{bmatrix}
\wh{A}^\top_{\calI, \calJ} & -\bm{e}^{|\calJ|} \\
(\bm{e}^{|\calI|})^\top & 0
\end{bmatrix}$. Following Theorem 1 of \cite{stewart1998perturbation}, we know that the difference of the corresponding singular values of the matrices $\begin{bmatrix}
{A}^\top_{\calI, \calJ} & -\bm{e}^{|\calJ|} \\
(\bm{e}^{|\calI|})^\top & 0
\end{bmatrix}$ and $\begin{bmatrix}
\wh{A}^\top_{\calI, \calJ} & -\bm{e}^{|\calJ|} \\
(\bm{e}^{|\calI|})^\top & 0
\end{bmatrix}$ is upper bounded by $\ell_2$-norm of their difference, i.e.,
\[
|\wh{\sigma}_{\calI, \calJ}-\sigma'_{\calI, \calJ}|\leq\left\|\begin{bmatrix}
{A}^\top_{\calI, \calJ} & -\bm{e}^{|\calJ|} \\
(\bm{e}^{|\calI|})^\top & 0
\end{bmatrix}-\begin{bmatrix}
\wh{A}^\top_{\calI, \calJ} & -\bm{e}^{|\calJ|} \\
(\bm{e}^{|\calI|})^\top & 0
\end{bmatrix} \right\|_2 \leq |\calI|\cdot |\calJ|\cdot\Rad(N'/m, \eps/m).
\]
Therefore, $\sigma'_{\calI, \calJ}=0$ implies that $\wh{\sigma}_{\calI, \calJ}\leq |\calI||\calJ|\cdot\Rad(N'/m, \eps/m)$. On the other hand, if $\sigma'_{\calI, \calJ}>0$, then it holds that
\[
\sigma'_{\calI, \calJ}\geq2|\calI||\calJ|\cdot\sigma_{0}\geq 2|\calI||\calJ|\cdot\Rad(N'/m, \eps/m),
\]
{where the first inequality is by definition of $\sigma_0$ in \Cref{def:sigma0} and the second inequality is due to the condition of $N'$.}
As a result, we have
\begin{align*}
    \wh{\sigma}_{\calI, \calJ} \geq \sigma'_{\calI, \calJ} - |\calI|\cdot |\calJ|\cdot\Rad(N'/m, \eps'/m) \geq |\calI|\cdot |\calJ|\cdot\Rad(N'/m,\eps/m).
\end{align*}
Combining both cases, we show that as long as $N'\geq \frac{2m}{\sigma_0^2}\cdot\log(2m/\epsilon)$,
the operation of checking $\wh{\sigma}_{\calI, \calJ'}> |\calI||\calJ'|\cdot\Rad(N'/m, \eps/m)$ in line 14 of \Cref{alg:Idenbasis} is equivalent to checking whether the matrix $\begin{bmatrix}
{A}^\top_{\calI, \calJ'} & -\bm{e}^{|\calJ'|} \\
(\bm{e}^{|\calI|})^\top & 0
\end{bmatrix}$ and $\begin{bmatrix}
\wh{A}^\top_{\calI, \calJ'} & -\bm{e}^{|\calJ'|} \\
(\bm{e}^{|\calI|})^\top & 0
\end{bmatrix}$ is singular or not.

Note that the index set $\wh{\calJ}^*$ is obtained from a basic optimal solution to the dual LP \Cref{eqn:DT011506}, where there are $|\wh{\calI}^*|+1$ constraints. Thus, the matrix $\begin{bmatrix}
\wh{A}^\top_{\wh{\calI}^*, \wh{\calJ}^*} & -\bm{e}^{|\wh{\calJ}^*|} \\
(\bm{e}^{|\wh{\calI}^*|})^\top & 0
\end{bmatrix}$ can have at most $|\wh{\calI}^*|+1$ number of linearly independent columns, which implies that $|\wh{\calJ}^*| \leq |\wh{\calI}^*|$, and the columns in the matrix $\begin{bmatrix}
\wh{A}^\top_{\wh{\calI}^*, \wh{\calJ}^*} & -\bm{e}^{|\wh{\calJ}^*|} \\
(\bm{e}^{|\wh{\calI}^*|})^\top & 0
\end{bmatrix}$ are linearly independent from each other. Moreover, from the stopping condition in line 17 of \Cref{alg:Idenbasis}, we know that $|\wh{\calJ}^*| \geq |\wh{\calI}^*|$, which results in $|\wh{\calI}^*| = |\wh{\calJ}^*|$. In this way, we show that as long as $N'\geq \frac{2m}{\sigma_0^2}\cdot\log(2m/\epsilon)$, we will meet the termination condition $|\wh{\calI}^*|=|\wh{\calJ}^*|$ described in line 4 of \Cref{alg:Twophase}. Note that we condition on the event $\mathcal{E}=\cup_k \mathcal{E}_k$ happens, which happens with a probability at least $1-\veps/4$.
Therefore, we know that with a probability at least $1-\veps/4$, we have that
\[
N' \leq \frac{4m}{\sigma_0^2}\cdot\log(2m/\epsilon)
\]
which implies that
\[
N_2\leq N_1 + \frac{8m}{\sigma_0^2}\cdot\log(2m/\epsilon).
\]
It only remains to bound the value of $K$ since we have $\eps=\veps/(8K^2)$. Note that every time calling \Cref{alg:Idenbasis}, we exponentially increase the sample size, which implies that $K=\log(N_2)$. Plugging in the above bound, we have that
\[
N_2\leq N_1 + \frac{8m}{\sigma_0^2}\cdot\log(16mK^2/\veps)\leq N_1 + \frac{8m}{\sigma_0^2}\cdot\log(16m\log^2(N_2)/\veps),
\]
which implies that 
\[
N_2\leq N_1 + \order\left(\frac{m}{\sigma_0^2}\cdot\log(m/\veps)\right).
\]
Our proof is thus completed.

\subsection{Proof of \Cref{lem:resolving}}
We now fix the failure probability $\eps=\veps/(N-N_2+1)^2$.
For a fixed dataset $\mathcal{H}^n=\{(i_k,j_k,\tilde{A}_{k,i_k,j_k})\}_{k=1}^n$, we denote by $\wh{A}(\mathcal{H}^n)$ the sample-average estimation of $A$ constructed from the dataset $\mathcal{H}^n$. Let $n$ denote the iteration over $n$ in \Cref{alg:Twophase}.
In the following, we denote $(\bx^*(\wh{\calI}^*, \wh{\calJ}^*), \mu^*(\wh{\calI}^*, \wh{\calJ}^*))$ as $(\bx^*, \mu^*)$. For a fixed $(\wh{\calI}^*, \wh{\calJ}^*)$, we now consider the gap between $\mathbb{E}[\bar{\bx}]$ and $\bx^*(\wh{\calI}^*, \wh{\calJ}^*)$.

We consider the stochastic process $\tilde{\bm{a}}^{n}$ defined in \Cref{eqn:Average}. For a fixed $\eta>0$ which will be specified later, we define sets
\begin{equation}\label{eqn:defX}
\mathcal{A}=\{ \tilde{\bm{a}}'\in\mathbb{R}^{|\wh{\mathcal{J}}^*|}: \tilde{a}'_{j}\in[-\eta, +\eta], \forall j\in\wh{\mathcal{J}}^* \}.
\end{equation}
Since $\tilde{\bm{a}}^{N_2+1}=\mathbf{0}$, we have  $\tilde{\bm{a}}^{N_2+1}\in\mathcal{A}$. Then, we show that $\tilde{\bm{a}}^n$ ``behave well'' as long as they stay in the region $\mathcal{A}$ for a
sufficiently long time, where $\tilde{\ba}^n=\frac{\ba^n}{N-n+1}$. To this end, we define a stopping time
\begin{equation}\label{eqn:Stoptime}
\tau=\min_{n\in[N_2+1,N]}\{ \tilde{\bm{a}}^n \notin\mathcal{A} \}.
\end{equation}
In this proof, we also define the matrix
\[
    M(\wh{\calI}^*, \wh{\calJ}^*) = \begin{bmatrix}
A_{\wh{\calI}^*, \wh{\calJ}^*} & -\bm{e}^{|\wh{\calI}^*|} \\
(\bm{e}^{|\wh{\calJ}^*|})^\top & 0
\end{bmatrix}
\]
corresponding to the definition in \Cref{eqn:true},
and denote by $\kappa$ the conditional number of the matrix $M(\wh{\calI}^*, \wh{\calJ}^*)$ defined above.
The first lemma shows that when $n$ is large enough but smaller than the stopping time $\tau$, we have $ \|(\bm{x}^n, \mu^n)\|_2\leq L$ with $L=4$. 
The proof is deferred to \Cref{app:lemProjection}.
\begin{lemma}\label{lem:projection}
There exist a constant $N'_0$  
such that {$N_2+N'_0\leq\tau$ and} the event
\[
\mathcal{E}_0'=\left\{ N_2+N'_0\leq\tau\text{~and~for~all~}n\text{~such~that~}N_2+N_0'\leq n\leq\tau, \text{it~holds~}\|(\tilde{\bm{x}}^{n}, \tilde{\mu}^{n})\|_2\leq L \right\}
\]
happens with probability at least $1-2N\cdot\eps$, where $(\tilde{\bm{x}}^{n}, \tilde{\mu}^{n})$ denotes the solution to \Cref{eqn:OptQ2}, $\tau$ is defined in \Cref{eqn:Stoptime} with 
\begin{align}\label{eqn:nu0}
    \eta=\frac{1}{8\sqrt{d}\cdot\kappa},
\end{align}
where $\kappa$ denotes the conditional number of the matrix $M(\wh{\calI}^*, \wh{\calJ}^*)$.
Specifically, $N_0'$ is given as follows
\begin{equation}\label{eqn:N0prime}
N_0'= 32d^4\cdot\frac{\kappa^2}{\|M(\wh{\calI}^*, \wh{\calJ}^*)\|_2^2}\cdot\log(2d^2/\eps).
\end{equation}
\end{lemma}
The next lemma bounds $\mathbb{E}[N-\tau]$ with the proof deferred to \Cref{app:lemStoptime}.
\begin{lemma}\label{lem:Stoptime}
Let the stopping time $\tau$ be defined in \Cref{eqn:Stoptime} with $\eta=\frac{1}{8d^{3/2}\cdot\kappa}$. It holds that
\[
\mathbb{E}[N-\tau]\leq N_0'+
2\cdot \exp(-\eta^2/(8d^4))+(N-N_2)^2\cdot\eps,
\]
where $N_0'$ is given in \Cref{eqn:N0prime}, as long as
\begin{equation}\label{eqn:022104}
N\geq  N_2+1024 d\cdot\kappa^2\cdot\log(4d^2/\eps).
\end{equation}
In addition, for any $N'$ such that $N_2+N_0'\leq N'\leq N$, it holds that
\begin{equation}\label{eqn:HighProbtau}
\Pr[\tau\leq N']\leq \frac{\eta^2}{4d^4}\cdot \exp\left( -\frac{\eta^2(N-N')}{8d^4} \right)+(N-N_2)\cdot\eps.
\end{equation}
\end{lemma}
Now we are ready to prove \Cref{thm:sampleComplexity}.
From the definition of the stopping time $\tau$ in \Cref{eqn:Stoptime}, we know that for each $j\in\wh{\mathcal{J}}^*$, it holds
\[
a_{j}^{\tau-1}\in\left[-(N-\tau+1)\cdot\eta, (N-\tau+1)\cdot\eta\right].
\]
Since $(i_n, j_n)$ is uniformly drawn from $\wh{\calI}^*\times \wh{\calJ}^*$ and is independent of the randomness of $\bx^n$, we know that
\begin{align*}
    &\mathbb{E}_{\tilde{A}_n, i_n, j_n}\left[\left|\wh{\mathcal{J}}^*\right|\cdot|\wh{\mathcal{I}}^*|\cdot \tilde{A}_{n,i_n,j_n}\cdot x^n_{i_n}\cdot \bm{h}_{j_n}\right]=\sum_{i=1}^{|\wh{\calI}^*|}\sum_{j=1}^{|\wh{\calJ}^*|}A_{ij}x_i^n\cdot \bh_j = A_{\wh{\calI}^*,\wh{\calJ}^*}^\top \bx^n.
\end{align*}
Then, following the update in \Cref{eqn:UpdateAlpha2}, we know that
\begin{align*}
\mathbb{E}[\bm{a}^{N+1}]&=\mathbb{E}\left[ \bm{a}^{\tau-1} \right] 
-\sum_{n=\tau}^N A^\top_{\wh{\mathcal{I}}^*, \wh{\mathcal{J}}^*}\mathbb{E}[\bx^n_{\wh{\mathcal{I}}^*}]+\sum_{n=\tau}^N\mathbb{E}[\mu^n]\cdot \bm{e}^{|\wh{\mathcal{J}}^*|},
\end{align*}
Thus,
we have that
\[
\left\|\mathbb{E}[\bm{a}^{N+1}]\right\|_{\infty}\leq \left\|\mathbb{E}[\bm{a}^{\tau-1}]\right\|_{\infty} + \left\|\sum_{n=\tau}^N \left(A^\top_{\wh{\mathcal{I}}^*, \wh{\mathcal{J}}^*} \mathbb{E}[\bx^n_{\wh{\mathcal{I}}^*}] - \mathbb{E}[\mu^n]\cdot\bm{e}^{|\wh{\mathcal{J}}^*|}\right) \right\|_{\infty}.
\]
Following the definition of the stop time $\tau$ in \Cref{eqn:Stoptime}, since it is easy to show that $\eta\leq1$, we know that
\[
\left\|\mathbb{E}[\bm{a}^{\tau-1}]\right\|_{\infty} \leq \mathbb{E}[N-\tau+2].
\]
Also, from the boundedness that 
\[
\|(\bx^n, \mu^n)\|_1\leq \sqrt{d+1}\cdot\|(\bx^n, \mu^n)\|_2\leq\sqrt{d+1}\cdot L\leq 2\sqrt{d}\cdot L,
\]
we know that
\[\begin{aligned}
\left\|\sum_{n=\tau}^N \left(A^\top_{\wh{\mathcal{I}}^*, \wh{\mathcal{J}}^*} \mathbb{E}[\bx^n_{\wh{\mathcal{I}}^*}] - \mathbb{E}[\mu^n]\cdot\bm{e}^{|\wh{\mathcal{J}}^*|}\right) \right\|_{\infty}
\leq 2L\sqrt{d}\cdot(N-\tau).
\end{aligned}\]
As a result, for each $j\in\wh{\mathcal{J}}^*$, it holds that
\begin{equation}\label{eqn:011907}
\begin{aligned}
&\left|\mathbb{E}[a_{j}^{N+1}]\right| \leq  (L\sqrt{d}+1)\cdot\mathbb{E}[N-\tau]+2\\
\leq &  (L\sqrt{d}+1)\cdot N_0'+
2(L\sqrt{d}+1)\cdot \exp(-\eta^2/(8d^4))+(L\sqrt{d}+1)(N-N_2)^2\cdot\eps+2,
\end{aligned}
\end{equation}
where the second inequality holds because of the upper bound on $\mathbb{E}[N-\tau]$ established in \Cref{lem:Stoptime}.
We further note that following the update in \Cref{eqn:UpdateAlpha2}, we have that 
\begin{align*}
\mathbb{E}[\bm{a}^{N+1}]
&=-\sum_{n=N_2}^N A^\top_{\wh{\mathcal{I}}^*, \wh{\mathcal{J}}^*}\mathbb{E}[\bx^n_{\wh{\mathcal{I}}^*}]+\sum_{n=N_2}^N\mathbb{E}[\mu^n]\cdot \bm{e}^{|\wh{\mathcal{J}}^*|}\\
&= -(N-N_2+1)\cdot A^\top_{\wh{\mathcal{I}}^*, \wh{\mathcal{J}}^*}\mathbb{E}[\bar{\bx}_{\wh{\mathcal{I}}^*}] + (N-N_2+1)\cdot\mathbb{E}[\bar{\mu}]\cdot\bm{e}^{|\wh{\mathcal{J}}^*|},
\end{align*}
where $\bar{\mu}=\frac{1}{N-N_2+1}\cdot\sum_{n=N_2}^N \mu^n$ and $\bar{\bx}=\frac{1}{N-N_2+1}\cdot\sum_{n=N_2}^N \bx^n$. 
Thus, we conclude that it holds 
\begin{equation}\label{eqn:OptQ10}
\begin{bmatrix}
A^\top_{\wh{\mathcal{I}}^*, \wh{\mathcal{J}}^*} & -\bm{e}^{|\wh{\mathcal{J}}^*|} \\
(\bm{e}^{|\wh{\mathcal{I}}^*|})^\top & 0
\end{bmatrix}
\begin{bmatrix}
\mathbb{E}[\bar{\bx}_{\wh{\mathcal{I}}^*}]\\
\mathbb{E}[\bar{\mu}]
\end{bmatrix}=
\begin{bmatrix}
-\frac{\mathbb{E}[\bm{a}^{N+1}]}{N-N_2}\\
1
\end{bmatrix},
\end{equation}
We now compare the linear equations \Cref{eqn:OptQ10} to \Cref{eqn:OptQ} to bound the gap between $\bar{\bx}$ and $\bx^*$, where $\bx_{\wh{\calI}^{*c}}^*=\bar{\bx}_{\wh{\calI}^{*c}}=\mathbf{0}$. Following standard perturbation analysis of linear equations (Theorem 1 of \citep{higham2002accuracy}), we know that
\begin{equation}\label{eqn:011510}
\begin{aligned}
    \|\bm{x}^*-\mathbb{E}[\bar{\bm{x}}]\|_2
    &\leq \kappa\cdot \frac{\left\|\mathbb{E}[\bm{a}^{N+1}]\right\|_2}{N-N_2},
\end{aligned}
\end{equation}
where $\kappa$ is the condition number of the matrix $M(\wh{\calI}^*, \wh{\calJ}^*)$.
We now plug in the formulation of $N'_0$ that is defined in \Cref{eqn:N0prime}, which also satisfies the condition in \Cref{eqn:022104}. {Through the bound in \Cref{eqn:011907} and \Cref{eqn:011510}}, we have
\begin{equation}\label{eqn:011512}
\begin{aligned}
\|\bm{x}^*-\mathbb{E}[\bar{\bm{x}}]\|_2&\leq \kappa\cdot \frac{\left\|\mathbb{E}[\bm{a}^{N+1}]\right\|_2}{N-N_2}\leq \sqrt{d}\kappa\cdot \frac{\left\|\mathbb{E}[\bm{a}^{N+1}]\right\|_\infty}{N-N_2}\\
&\leq \frac{\sqrt{d}\kappa}{N-N_2}\cdot\left((L\sqrt{d}+1)\cdot N_0'+
2(L\sqrt{d}+1)\cdot \exp(-\eta^2/(8d^4))+(L\sqrt{d}+1)N^2\cdot\eps+2\right)\\
&\leq 256d\kappa\cdot\left( d^4\cdot\frac{\kappa^2}{\|A^*\|_2^2}+
1\right)\cdot\log(2d^2/\eps)\cdot\frac{1}{N-N_2}+\frac{10d\kappa}{N-N_2}+10d\kappa (N-N_2)\cdot\eps,
\end{aligned}
\end{equation}
where the last inequality follows from upper bounding 
\[
\exp(-\eta^2/(8d^4))\leq 1 \text{~and~} (L\sqrt{d}+1)(N-N_2)^2\cdot\eps+2\leq 2(L\sqrt{d}+1)(N-N_2)^2\cdot\eps\leq 10\sqrt{d}(N-N_2)^2\cdot\eps.
\]
In order to translate the bound in \Cref{eqn:011512} into the sample complexity bound, we let the right hand sides of \Cref{eqn:011512} to be upper bounded $\veps/(2d)$ and we specify $\eps=\veps/(N-N_2+1)^2$. Then, we would require the following condition to be satisfied 
\begin{equation}\label{eqn:090801}
N-N_2\geq 552 d^{2}\cdot\kappa\cdot\left(1+d^4\cdot\frac{\kappa^2}{\|A^*\|_2^2}\right)\cdot\frac{\log(1/\veps)}{\veps} =552d^{2}\cdot\kappa\cdot\left(1+d^4\cdot\frac{1}{\sigma_{\wh{\calI}^*, \wh{\calJ}^*}^2}\right)\cdot\frac{\log(1/\veps)}{\veps}.
\end{equation}
Further upper bounding $\kappa$ by $d/\sigma_{\wh{\calI}^*, \wh{\calJ}^*}$ and note that $1+d^4\cdot\frac{1}{\sigma_{\wh{\calI}^*, \wh{\calJ}^*}}\leq 2d^4\cdot\frac{1}{\sigma_{\wh{\calI}^*, \wh{\calJ}^*}^2}$, we have that setting
\[
N-N_2 =\frac{ 4120d^{15/2}}{(\sigma_{\wh{\calI}^*, \wh{\calJ}^*})^3}\cdot\frac{\log(m/\veps)}{\veps}
\]
would satisfy the condition in \Cref{eqn:090801}.
Our proof is thus completed.

\subsubsection{Proof of \Cref{lem:projection}}\label{app:lemProjection}
We first prove that $N_2+N'_0\leq\tau$. 
From the update rule in \Cref{eqn:UpdateAlpha2}, we know that
\[
\bm{a}^{N_2+N'_0}=\bm{a}^{N_2+1}- \sum_{n=N_2+1}^{N_2+N'_0-1}|\wh{\calJ}^*|\cdot|\wh{\calI}^*|\cdot \wt{A}_{n, i_n,j_n}\cdot x^n_{i_n}\cdot \bm{h}_{j_n} +\sum_{n=N_2+1}^{N_2+N'_0-1}\mu^n\cdot \bm{e}^{|\wh{\calJ}^*|}
\]
Noting that from definition $\bm{a}^{N_2+1}=\bm{0}$ and $(\bx^n, \mu^n)\in\mathcal{L}$ such that $\|(\bx^n, \mu^n)\|_\infty\leq\|(\bx^n, \mu^n)\|_2\leq L$, for each element $j\in\wh{\calJ}^*$, we have that
\[
\bm{a}_j^{N_2+N'_0}\leq d^2\cdot N_0'\cdot L + L\cdot N_0'\leq (d^2+1)L\cdot N_0',
\]
which implies that
\[
\|\bm{a}^{N_2+N'_0}\|_{\infty}\leq (d^2+1)L\cdot N_0'.
\]
Then, we know that
\[
\|\tilde{\bm{a}}^{N_2+N'_0}\|_{\infty}\leq \frac{(d^2+1)L\cdot N_0'}{N-N_2-N_0'}.
\]
As a result, as long as $N$ is large enough such that
\begin{equation}\label{eqn:092301}
N\geq N_2+\left( \frac{(d^2+1)L}{\eta}+1 \right)\cdot N_0',
\end{equation}
we will have $\|\tilde{\bm{a}}^{N_2+N'_0}\|_{\infty}\leq\eta$ which implies that $N_2+N'_0\leq\tau$. Plugging in the formulation of $N$, $N'_0$ and $\eta$, we have that
\[
N-N_2=\frac{ 4120d^{15/2}}{(\sigma_{\wh{\calI}^*, \wh{\calJ}^*})^3}\cdot\frac{\log(m/\veps)}{\veps}
\]
and 
\[\begin{aligned}
\left( \frac{(d^2+1)L}{\eta}+1 \right)\cdot N_0'&\leq 128d^{5/2}\kappa\cdot N_0'\leq 4120 d^{13/2}\cdot\frac{\kappa^3}{\|M(\wh{\calI}^*, \wh{\calJ}^*)\|_2^2}\cdot\log(m/\eps)\\
&=4120 d^{13/2}\cdot\frac{\kappa}{\sigma_{\wh{\calI}^*, \wh{\calJ}^*}^2}\cdot\log(m/\eps)\leq \frac{ 4120d^{15/2}}{(\sigma_{\wh{\calI}^*, \wh{\calJ}^*})^3}\cdot\frac{\log(m/\veps)}{\veps},
\end{aligned}\]
where we upper bound $\kappa$ by $d/\sigma_{\wh{\calI}^*, \wh{\calJ}^*}$. Therefore,
we know that the condition in \Cref{eqn:092301} is satisfied and in this way, we finish our proof that $\|\tilde{\bm{a}}^{N_2+N'_0}\|_{\infty}\leq\eta$ which implies that $N_2+N'_0\leq\tau$.

We now finish our remaining proof. Denote by $(\bm{x}^*, \mu^*)$ the basic solution corresponding to the basis $\wh{\calI}^*$ and $\wh{\calJ}^*$, as defined as the unique solution to \Cref{eqn:OptQ}. Then, it holds that
\begin{equation}\label{eqn:OptQ3}
A^\top_{\wh{\calI}^*, \wh{\calJ}^*}\bm{x}^*_{\wh{\calI}^*}=\mu^*\cdot\bm{e}^{|\wh{\calJ}^*|}.
\end{equation}
We compare $(\tilde{\bm{x}}^n, \tilde{\mu}^n)$ with $(\bm{x}^*, \mu^*)$ and when $n$ is large enough. In the remaining proof, we focus on comparing $(\tilde{\bm{x}}^n, \tilde{\mu}^n)$ with $(\bm{x}^*, \mu^*)$. 

Note that for $n\geq N_2+1$, $\tilde{\bm{x}}^n$ is the solution to the following linear equations
\begin{equation}\label{eqn:OptQ4}
\begin{bmatrix}
&\wh{A}^\top_{\wh{\calI}^*, \wh{\calJ}^*}(\mathcal{H}^n) & -\bm{e}^{|\wh{\calJ}^*|} \\
& (\bm{e}^{|\wh{\calI}^*|})^\top & 0
\end{bmatrix}
\begin{bmatrix}
\tilde{\bx}^n_{\wh{\calI}^*}\\
\tilde{\mu}^n
\end{bmatrix}=
\begin{bmatrix}
\bm{a}^n/(N-n+1)\\
 1
\end{bmatrix}.
\end{equation}
When $n\leq\tau$, we know that
\begin{equation}\label{eqn:020301}
\left\| \frac{\bm{a}^n}{N-n+1} \right\|_\infty\leq\eta.
\end{equation}
Now, we denote by $\mathbb{I}\{i_q =i, j_q = j\}$ an indicator function of whether $i_q = i$ and $j_q = j$, for $q=1,\dots,N$, and we denote by
\[
n_{i,j} \triangleq \sum_{q=1}^n \mathbb{I}\{i_q=i,j_q=j\},
\]
the number of samples of the matrix entry $A_{i,j}$. From the uniform sampling rule and standard Hoeffding bound, noting $\mathbb{E}[n_{i,j}]=(n-N_2)/d^2$, we know that for all $i\in\wh{\calI}^*$ and $j\in\wh{\calJ}^*$, 
\[
\Pr[n_{i,j}\geq (n-N_2)/(2d^2)]= 1-\Pr(n_{i,j}<(n-N_2)/(2d^2))\geq 1-\exp(-(n-N_2)/(2d^4)).
\]
We now condition on the event $\mathcal{E}_n'$ happens, which is defined as
\[
\mathcal{E}_n' = \left\{ n_{i,j}\geq (n-N_2)/(2d^2), \forall i\in\wh{\calI}^*, \forall j\in\wh{\calJ}^* \right\}.
\]
From the union bound, we know that 
\[
\Pr[\mathcal{E}_n']\geq 1-d^2\cdot\exp(-(n-N_2)/(2d^4))\geq 1-\eps,
\]
as long as $n$ satisfies the condition
\begin{equation}\label{eqn:061201}
d^2\cdot\exp(-(n-N_2)/(2d^4)) \leq \eps.
\end{equation}

Conditional on the event $\mathcal{E}_n'$ happens, we know that the absolute value of each element of $\wh{A}_{\wh{\calI}^*, \wh{\calJ}^*}(\mathcal{H}^n)-A_{\wh{\calI}^*, \wh{\calJ}^*}$ is upper bounded by $\Rad((n-N_2)/(2d^2), \eps/d^2)$, given that the following event
\begin{equation}\label{def:event2}
\mathcal{E}_n=\left\{\left|\frac{1}{n_{i,j}}\cdot\sum_{q=1}^{n_{i,j}} A_{q}(i,j)\cdot \mathbb{I}\{i_q=i,j_q=j\}-A_{i,j}\right|\leq \Rad((n-N_2)/(2d^2), \eps/d^2), ~~\forall i\in\wh{\calI}^*, j\in\wh{\calJ}^*  \right\}
\end{equation}
is assumed to happen (it holds with probability at least $1-\eps$ following standard Hoeffding bound). We now bound the distance between the solutions to \Cref{eqn:OptQ3} and \Cref{eqn:OptQ4}. The perturbation of the matrix is denoted as
\[
\Delta A=A_{\wh{\calI}^*, \wh{\calJ}^*} - \wh{A}_{\wh{\calI}^*, \wh{\calJ}^*}(\mathcal{H}^n).
\]
It holds that
\begin{equation}\label{eqn:020303}
\|\Delta A\|_2 \leq \sqrt{d}\cdot\Rad((n-N_2)/(2d^2), \eps/d^2).
\end{equation}
Therefore, as long as 
\begin{equation}\label{eqn:020304}
\|\Delta A\|_2 \leq \sqrt{d}\cdot\Rad((n-N_2)/(2d^2), \eps/d^2) \leq \frac{\|M(\wh{\calI}^*, \wh{\calJ}^*)\|_2}{2\kappa},
\end{equation}
{noting that $\left\| \frac{\bm{a}^n}{N-n+1} \right\|_\infty\leq\eta$ from \Cref{eqn:020301} and} following standard perturbation analysis of linear equations (Theorem 1 of \citep{higham2002accuracy}), we have that

\begin{equation}\label{eqn:020305}
\begin{aligned}
    \frac{\|(\tilde{\bm{x}}^n_{\wh{\calI}^*}, \tilde{\mu}^n)-(\bm{x}^*_{\wh{\calI}^*}, \mu^*)\|_2}{\|(\bm{x}^*_{\wh{\calI}^*}, \mu^*)\|_2}
    &\leq \frac{\kappa}{1-\kappa\cdot\frac{\|\Delta A\|_2}{\|M(\wh{\calI}^*, \wh{\calJ}^*)\|_2}}\cdot \left( \frac{\|\Delta A\|_2}{\|M(\wh{\calI}^*, \wh{\calJ}^*)\|_2}+\sqrt{d}\cdot\eta \right)\\
    &\leq 2\cdot\kappa\cdot\left( \frac{\|\Delta A\|_2}{\|M(\wh{\calI}^*, \wh{\calJ}^*)\|_2}+\sqrt{d}\cdot\eta \right),
\end{aligned}
\end{equation}
where 
$\kappa$ denotes the conditional number of $M(\wh{\calI}^*, \wh{\calJ}^*)$ defined in \Cref{eqn:true}. 
We now want to propose conditions on $\|\Delta A\|_2$ and $\eta$ such that the right hand side of \Cref{eqn:020305} is small enough such that we can show the $l_2$ norm of $(\tilde{\bm{x}}^n_{\wh{\calI}^*}, \tilde{\mu}^n)-(\bm{x}^*_{\wh{\calI}^*}, \mu^*)$ is at most $L/2$. Together with the upper bound that 
\[
\|(\bm{x}^*_{\wh{\calI}^*}, \mu^*)\|_2\leq\|(\bm{x}^*_{\wh{\calI}^*}, \mu^*)\|_1\leq L/2,
\]
our proof will be completed.

We note that $\|(\bm{x}^*_{\wh{\calI}^*}, \mu^*)\|_2\leq \|(\bm{x}^*_{\wh{\calI}^*}, \mu^*)\|_1\leq\frac{L}{2}$. Then, \Cref{eqn:020305} can be rewritten as
\[
\|(\tilde{\bm{x}}^n_{\wh{\calI}^*}, \tilde{\mu}^n)-(\bm{x}^*_{\wh{\calI}^*}, \mu^*)\|_2
\leq L\cdot\kappa\cdot\left( \frac{\|\Delta A\|_2}{\|M(\wh{\calI}^*, \wh{\calJ}^*)\|_2}+\sqrt{d}\cdot\eta \right).
\]
We now set the condition on $n$ such that $n$ satisfies the condition in \Cref{eqn:020304} and the following condition
\begin{equation}\label{eqn:020306}
 L\cdot\kappa\cdot \frac{\|\Delta A\|_2}{\|M(\wh{\calI}^*, \wh{\calJ}^*)\|_2}\leq L\cdot \sqrt{d}\cdot\frac{\kappa}{\|M(\wh{\calI}^*, \wh{\calJ}^*)\|_2}\cdot\Rad((n-N_2)/(2d^2), \eps/d^2)  \leq \frac{L}{4},
\end{equation}
we also set the condition on $\eta$ such that
\begin{equation}\label{eqn:020307}
2L\cdot\kappa\cdot \sqrt{d}\cdot\eta \leq\frac{L}{4}.
\end{equation}
Clearly, as long as conditions \Cref{eqn:020304}, \Cref{eqn:020306}, and \Cref{eqn:020307} are satisfied, we have that 
\[
\|(\tilde{\bm{x}}^n_{\wh{\calI}^*}, \tilde{\mu}^n)-(\bm{x}^*_{\wh{\calI}^*}, \mu^*)\|_2\leq  L\cdot\kappa\cdot\left( \frac{\|\Delta A\|_2}{\|M(\wh{\calI}^*, \wh{\calJ}^*)\|_2}+\sqrt{d}\cdot\eta \right)\leq L/2,
\]
which will complete our proof.

We now set the condition $n\geq N_2+ N_0'$ with $N_0'$ given by
\begin{equation}\label{eqn:022101}
 N'_0:= 32d^4\cdot\frac{\kappa^2}{\|M(\wh{\calI}^*, \wh{\calJ}^*)\|_2^2}\cdot\log(2d^2/\eps).
\end{equation}
Then, the conditions \Cref{eqn:061201}, \Cref{eqn:020304} and \Cref{eqn:020306} are satisfied. Moreover, we set 
\begin{equation}\label{eqn:022102}
\eta:= \frac{1}{8\sqrt{d}\cdot\kappa}.
\end{equation}
Then, the condition in \Cref{eqn:020307} is satisfied. Therefore, with the conditions in \Cref{eqn:022101} and \Cref{eqn:022102}, we know that
\[
\|(\tilde{\bm{x}}^n_{\wh{\calI}^*}, \tilde{\mu}^n)-(\bm{x}^*_{\wh{\calI}^*}, \mu^*)\|_2\leq L/2.
\]
Moreover, note that $\|(\bm{x}^*_{\wh{\calI}^*}, \mu^*)\|_2\leq\|(\bm{x}^*_{\wh{\calI}^*}, \mu^*)\|_1\leq L/2$, we have that
\[
\|(\tilde{\bm{x}}^n_{\wh{\calI}^*}, \tilde{\mu}^n)\|_2 \leq \|(\tilde{\bm{x}}^n_{\wh{\calI}^*}, \tilde{\mu}^n)-(\bm{x}^*_{\wh{\calI}^*}, \mu^*)\|_2 + \|(\bm{x}^*_{\wh{\calI}^*}, \mu^*)\|_2 \leq L.
\]
Our proof is completed by taking a union bound over the probability of the event $\mathcal{E}'_n$ and $\mathcal{E}_n$, which is lower bounded by $1-\eps$, for all $n$. To be specific, note that
\[
\mathcal{E}_0'\subset \cup_{n=N_0}^N\left( \mathcal{E}_n \cup\mathcal{E}_n' \right).
\]
From the union bound, we know that
\[
\Pr[\mathcal{E}_0']\geq 1-\sum_{n=N_2}^N\left( 1-\Pr[\mathcal{E}_n]+1-\Pr[\mathcal{E}_n'] \right)\geq 1-2N\eps.
\]
Our proof is thus completed.

\subsubsection{Proof of \Cref{lem:Stoptime}}\label{app:lemStoptime}
Now we fix an arbitrary $j\in\wh{\calJ}^*$. For any $N_2+N'_0\leq N'\leq N$, it holds that
\[
\tilde{a}_{j}^{N'}-\tilde{a}_{j}^{N_2+{N}'_0}=\sum_{n=N_2+{N}'_0}^{N'-1}(\tilde{a}_{j}^{n+1}-\tilde{a}_{j}^{n}).
\]
We define $\xi_{j}^{n}=\tilde{a}_{j}^{n+1}-\tilde{a}_{j}^{n}$. Then, we have
\[
\tilde{a}_{j}^{N'}-\tilde{a}_{j}^{N_2+{N}'_0}=\sum_{n=N_2+{N}'_0}^{N'-1}(\xi_{j}^{n}-\mathbb{E}[\xi_{j}^{n}|\mathcal{H}^n])+\sum_{n=N_2+{N}'_0}^{N'-1}\mathbb{E}[\xi_{j}^{n}|\mathcal{H}^n],
\]
where $\mathcal{H}^n$ denotes the filtration of information up to step $n$.
Note that due to the update in \Cref{eqn:Aveupalpha}, we have
\[
\xi_{j}^{n}=\frac{\tilde{a}_{j}^{n}-|\wh{\calJ}^*|\cdot|\wh{\calI}^*|\cdot \tilde{A}_{n,i_n,j_n}\cdot x^n_{i_n}\cdot 1_{\{j=j_n\}}+\mu^n}{N-n},
\]
where $1_{\{j=j_n\}}$ is an indicator function of whether $j=j_n$.
Then, it holds that
\begin{equation}\label{eqn:011901}
|\xi_{j}^{n}-\mathbb{E}[\xi_{j}^{n}|\mathcal{H}^n]|\leq \frac{d^2}{N-n}
\end{equation}
where the inequality follows from the fact that the value of $\tilde{a}_j^{n}$ is deterministic given the filtration $\mathcal{H}^n$ and we have $\|(\bm{x}^n, \mu^n)\|_\infty\leq\|(\bm{x}^n, \mu^n)\|_2\leq L$ for any $n$, as well as $|\wh{\calJ}^*|=|\wh{\calI}^*|=d$.
Note that
\[
\{\xi_{j}^{n}-\mathbb{E}[\xi_{j}^{n}|\mathcal{H}^n]\}_{\forall n=N_2+N'_0,\dots,N'}
\]
forms a martingale difference sequence. Following Hoeffding's inequality, for any $N''\leq N'$ and any $b>0$, it holds that
\[\begin{aligned}
\Pr\left[ \left| \sum_{n=N_2+{N}'_0}^{N''}(\xi_{j}^{n}-\mathbb{E}[\xi_{j}^{n}|\mathcal{H}^n]) \right|\geq b \right]&\leq 2\exp\left( -\frac{b^2}{2\cdot\sum_{n=N_2+{N}'_0}^{N''}d^4/(N-n)^2 } \right)\\
&\leq 2\exp\left( -\frac{b^2\cdot(N-N'')}{d^4} \right).
\end{aligned}\]
Therefore, we have that
\begin{equation}\label{eqn:011902}
\begin{aligned}
&\Pr\left[ \left| \sum_{n=N_2+{N}'_0}^{N''}(\xi_{j}^{n}-\mathbb{E}[\xi_{j}^{n}|\mathcal{H}^n]) \right|\geq b \text{~for~some~}N_2+{N}'_0\leq N''\leq N' \right]\\
\leq &\sum_{N''=N_2+{N}'_0}^{N'} 2\exp\left( -\frac{b^2\cdot(N-N'')}{2d^4} \right)\leq \frac{b^2}{d^4}\cdot \exp\left( -\frac{b^2\cdot(N-N')}{2d^4 } \right)
\end{aligned}
\end{equation}
holds for any $b>0$.

We now bound the probability that $\tau>N'$ for one particular $N'$ such that $N_2+{N}'_0\leq N'\leq N$. Suppose that $N'\leq \tau$, then,  from \Cref{lem:projection}, for each $n\leq N'$, we know that $\|(\tilde{\bm{x}}^n, \tilde{\mu}^n)\|_2\leq L$ 
and therefore $\bm{x}^n=\tilde{\bm{x}}^n$, $\mu^n=\tilde{\mu}^n$ as the solution to \Cref{eqn:OptQ2}. We have
\[
\tilde{a}_{j}^{n}=\wh{A}^\top_{\wh{\calI}^*, j}(\mathcal{H}^n) \bm{x}^n_{\wh{\calI}^*}-\mu^n.
\]
Now, we denote by $\mathbb{I}\{i_q =i, j_q = j\}$ an indicator function of whether $i_q = i$ and $j_q = j$, and we denote by
\[
n_{i,j} \triangleq \sum_{q=1}^n \mathbb{I}\{i_q =i, j_q = j\},
\]
the number of noisy samples for $A_{i,j}$ till iteration $n$. From the uniform sampling rule and standard Hoeffding bound, noting $\mathbb{E}[n_{i,j}]\geq (n-N_2)/d^2$, we know that
\[
\Pr[n_{i,j}\geq (n-N_2)/(2d^2)]= 1-\Pr[n_{i,j}<(n-N_2)/(2d^2)]\geq 1-\exp(-(n-N_2)/(2d^4)).
\]
We now condition on the event $\mathcal{E}_n'$ happens, which is defined as
\[
\mathcal{E}_n' = \left\{ n_{i,j}\geq (n-N_2)/(2d^2), \forall i\in\wh{\calI}^*, \forall j\in\wh{\calJ}^* \right\}.
\]
From the Hoeffding bound, we know that 
\[
\Pr[\mathcal{E}_n']\geq 1-d^2\cdot\exp(-(n-N_2)/(2d^4)).
\]
We now set a $N_1'$ such that the above probability bound is small enough when $n\geq N_2+N_1'$. To be specific, $N_1'$ is set as
\begin{equation}\label{eqn:n1}
    N_1' = 2d^4\cdot\log(4d^6/\eps).
\end{equation}
Then, denote by $\mathcal{E'}$ the event such that
\[
\mathcal{E}' = \{ \cap_{n\geq N_2+N_1'}\mathcal{E}'_n \}.
\]
From the union bound, we have
\[
\Pr[\mathcal{E}']\geq 1- \sum_{n=N_2+N_1'}^N d^2\cdot\exp(-(n-N_2)/(2d^4))\geq 1- 2d^6\cdot \exp(-N_1'/(2d^4))=1-\eps/2.
\]
Conditional on the event $\mathcal{E}'$ happens, we know that the absolute value of each element of $\wh{A}_{\wh{\calI}^*, \wh{\calJ}^*}(\mathcal{H}^n)-A_{\wh{\calI}^*, \wh{\calJ}^*}$ is upper bounded by $\Rad((n-N_2)/(2d^2), \eps/d^2)$, given that the following event
\begin{equation}\label{def:event21}
\mathcal{E}_n=\left\{\left|\frac{1}{n_{i,j}}\cdot\sum_{q=1}^{N} \wt{A}_{q, i_q,j_q}\cdot \mathbb{I}\{i_q=i,j_q=j\}-A_{i,j}\right|\leq \Rad((n-N_2)/(2d^2), \eps/(2d^2)), ~~\forall i\in\wh{\calI}^*, j\in\wh{\calJ}^*  \right\}
\end{equation}
is assumed to happen (it holds with probability at least $1-\eps/2$ following standard Hoeffding bound). Therefore, from the union bound, we know that both event $\mathcal{E}'$ and event $\mathcal{E}_n$ happen with probability at least $1-\eps$.

Conditional on the events $\mathcal{E}'$ and $\mathcal{E}_n$ happen, for $n\geq N_2+\max\{N_0',  N_1'\} = N_2+N_0'$,
it holds that
\begin{equation}\label{eqn:011903}
\left|\mathbb{E}_{i_n, j_n}[\xi_{j}^n|\mathcal{H}^n]\right|\leq\frac{1}{N-n}\cdot \left\|\left(\wh{A}_{\wh{\calI}^*, j}(\mathcal{H}^n) - A_{\wh{\calI}^*, j}\right)\cdot\bx^n_{\wh{\calI}^*}\right\|_{\infty} \leq\frac{ \Rad((n-N_2)/(2d^2),\eps/(2d^2))}{N-n}.
\end{equation}
Then, we know that
\begin{equation}\label{eqn:011904}
\begin{aligned}
\frac{\sum_{n=N_2+{N}'_0}^{N'-1}\left|\mathbb{E}[\xi_{j}^{n}|\mathcal{H}^n]\right|}{\sqrt{2}\cdot d}&\leq\sqrt{\frac{\log(4d^2/\eps)}{2}}\cdot\sum_{n=N_2+{N}'_0}^{N'-1}\frac{1}{\sqrt{n}\cdot(N-n)}\\
&\leq \sqrt{\frac{\log(4d^2/\eps)}{2}}\cdot \sqrt{N'-1}\cdot \sum_{n=N_2+{N}'_0}^{N'-1}\frac{1}{n\cdot(N-n)}\\
&=\sqrt{\frac{\log(4d^2/\eps)}{2}}\cdot \frac{\sqrt{N'-1}}{N}\cdot \sum_{n=N_2+{N}'_0}^{N'-1}\left( \frac{1}{n}+\frac{1}{N-n} \right)\\
&\leq \sqrt{2\log(4d^2/\eps)}\cdot \frac{\sqrt{N'-1}}{N}\cdot \log(N)
\leq \frac{\sqrt{2\log(4d^2/\eps)}}{\sqrt{N}}\cdot\log(N)\\
&\leq\frac{\eta}{2}
\end{aligned}
\end{equation}
for a $N$ large enough such that
\begin{equation}\label{eqn:022103}
N\geq \frac{16\log(4d^2/\eps)}{\eta^2}=1024 d\cdot\kappa^2\cdot\log(4d^2/\eps).
\end{equation}
We are now ready to bound the probability for $\tau\leq N'$. For any $N''$ such that $N_2+N_0'\leq N''\leq N'-1$, we have
\[
\tilde{a}_{j}^{N''}-\tilde{a}_{j}^{N_2+{N}'_0}=\sum_{n=N_2+{N}'_0}^{N''-1}(\xi_{j}^{n}-\mathbb{E}[\xi_{j}^{n}|\mathcal{H}^n])+\sum_{n=N_2+{N}'_0}^{N''-1}\mathbb{E}[\xi_{j}^{n}|\mathcal{H}^n].
\]
Then, $\tau\leq N'$ implies that $\left| \sum_{n=N_2+{N}'_0}^{N''-1}(\xi_{j}^{n}-\mathbb{E}[\xi_{j}^{n}|\mathcal{H}^n]) \right|\geq\frac{\eta}{2}$ or $\left|\sum_{n=N_2+{N}'_0}^{N''-1}\mathbb{E}[\xi_{j}^{n}|\mathcal{H}^n] \right|\geq\frac{\eta}{2}$ for some $N''$ such that $N_2+N_0'\leq N''\leq N'-1$. From \Cref{eqn:011904} and \Cref{eqn:011902} with $b=\eta/2$, and applying a union bound over all $j\in\wh{\calJ}^*$, as well as event $\mathcal{E}$ and $\mathcal{E}'$, we know that
\begin{equation}\label{eqn:011905}
\Pr[\tau\leq N']\leq \frac{\eta^2}{4d^4}\cdot \exp\left( -\frac{\eta^2\cdot(N-N')}{8d^4 } \right)+(N-N_2)\cdot\eps,
\end{equation}
where the $N\cdot\eps$ terms comes from a union bound over all the events $\mathcal{E}'$ and $\mathcal{E}_n$ for all $n$.
Therefore, noting $\tau$ is defined to be greater than $N_2+1$, we know that
\[
\mathbb{E}[N-\tau]=\sum_{N'=N_2+1}^N \Pr[\tau \leq N']\leq {N}'_0+\sum_{N'=N_2+N'_0}^N \Pr[\tau \leq N']\leq {N}'_0+
2\cdot \exp(-\eta^2/(8d^4))+(N-N_2)^2\cdot\eps,
\]
which completes our proof.

\subsection{Proof of \Cref{thm:sampleComplexity}}

Note that when $N_1\geq\frac{2m\cdot\log(16m/\veps)}{\delta^2}$, from \Cref{prop:DoublingTrick}, with a probability at least $1-\veps/4$, we know that the final output $\wh{\calI}^*$ and $\wh{\calJ}^*$ of \Cref{alg:Idenbasis} in the execution of \Cref{alg:Twophase} is an optimal basis, i.e., $\wh{\calI}^*=\calI^*$ and $\wh{\calJ}^*=\calJ^*$, which we denote by event $\mathcal{E}_1$. Also, from \Cref{thm:BoundN1}, with a probability at least $1-\veps/12$, we know that our estimation $\sigma'$ approximates $\sigma$ well in that $\sigma'/2\leq\sigma\leq2\sigma'$, which we denote by event $\mathcal{E}_2$. We now denote by $\mathcal{E}$ the good event that 
\[
\mathcal{E} = \mathcal{E}_1\cap \mathcal{E}_2.
\]
Taking the union bound, we have that 
\[
P(\mathcal{E}) \geq 1-\veps/3.
\]
Further combined with \Cref{lem:resolving}, we know that conditional on the event $\mathcal{E}$ happens, we have that 
\[
\argmin_{\bx^*\in\calX^*}\|\E[\bar{\bx}]-\bx^*\|_2\leq \veps/(2d^*),
\]
since $\wh{\calI}^*$ and $\wh{\calJ}^*$ become an optimal basis and $\bx(\wh{\calI}^*, \wh{\calJ}^*)$ is an optimal solution to $\VI$. 
As a result, we conclude that with a probability at least $1-\veps/3$, it holds that
\[
\argmin_{\bx^*\in\calX^*}\|\E[\bar{\bx}]-\bx^*\|_2\leq \veps/(2d^*).
\]
We now consider the bad case that the event $\mathcal{E}$ does not happen, which happens with a probability at most $\veps/3$. In this case, since we have $\mathbb{E}[\bar{\bx}]\in\Delta_{m_1}$ and $\bx^*\in\Delta_{m_1}$ for an arbitrary $\bx^*\in\mathcal{X}^*$, we know that 
\[
\|\mathbb{E}[\bar{\bx}]-\bx^*\|_2 \leq \sqrt{\|\mathbb{E}[\bar{\bx}]\|_2^2+\|\bx^*\|^2_2} \leq \sqrt{2}.
\]
Therefore, we have that
\[
\argmin_{\bx^*\in\calX^*}\|\E[\bar{\bx}]-\bx^*\|_2\leq \veps/2 + \sqrt{2}\cdot\veps/3 \leq\veps.
\]
In order to bound the total number of samples, we simply use the bound on $N_2$ in \Cref{prop:DoublingTrick} and use the definition of $N$ in line 6 of \Cref{alg:Twophase}, we get that 
\[
N= \order \left( \left(\frac{m}{\sigma_0^2}+ \frac{m}{\delta^2}\right)\cdot\log(m/\veps) + \frac{(d^*)^{15/2}}{\sigma^3}\cdot\frac{\log(m/\veps)}{\veps} \right).
\]
Our proof is thus completed.
\subsection{Proof of \Cref{thm:sample-complexity-independent}}\label{app:sample-complexity-independent}
We now present the worst-case bound on the suboptimality of \Cref{alg:Idenbasis} and \Cref{alg:Twophase} with a finite sample size. We let $N'$ denote its final value when calling \Cref{alg:Idenbasis} in the iteration of \Cref{alg:Twophase}. 
Denote by $\calI^{N'}$ and $\calJ^{N'}$ the output of \Cref{alg:Idenbasis}, we know from \Cref{thm:FiniteGap} and \Cref{prop:DoublingTrick} that with a probability at least $1-\veps/4$, the solution $\bx^*(\calI^{N'}, \calJ^{N'})$ forms a feasible solution to $\VI$ (by properly defining variable $\mu$) with an sub-optimality gap bounded by 
\[
\Rad(N'/m, \eps/m) + 4(d')^2\wh{\kappa}\cdot\Rad(N'/m, \eps/m),
\]
where $\eps=\veps/(8K^2)$ with $K$ denotes the number of times that \Cref{alg:Idenbasis} is called in the iteration in line 3 and line 4 of \Cref{alg:Twophase}, $\hat{\sigma}_{\min}$ and $\hat{\sigma}_{\max}$ are the smallest and the largest absolute value of the singular values of the matrix 
\[
\begin{bmatrix}
\wh{A}^\top_{\calI^{N'}, \calJ^{N'}} & -\bm{e}^{|\calJ^{N'}|} \\
(\bm{e}^{|\calI^{N'}|})^\top & 0
\end{bmatrix},
\]
and we denote $d'=|\calI^{N'}|=|\calJ^{N'}|$.
We now further denote by $\bar{\bx}$ the output of \Cref{alg:Twophase} and we compare $\bar{\bx}$ with $\bx^*(\calI^{N'}, \calJ^{N'})$. Note that the procedure of \Cref{alg:Twophase} is used to approximate the solution of the linear equations 
\begin{equation}\label{eqn:012901}
\begin{bmatrix}
A^\top_{\calI^{N'}, \calJ^{N'}} & -\bm{e}^{|\calJ^{N'}|} \\
 (\bm{e}^{|\calI^{N'}|})^\top & 0
\end{bmatrix}
\begin{bmatrix}
\bx_{\calI^{N'}}^*(\calI^{N'}, \calJ^{N'})\\
\mu^*(\calI^{N'}, \calJ^{N'})
\end{bmatrix}=
\begin{bmatrix}
\bm{0}\\
1
\end{bmatrix},
\end{equation}
which is exactly $\bx^*(\calI^{N'}, \calJ^{N'})$. Further note that from \Cref{lem:resolving}, we know that for any pair of $\calI^{N'}$ and $\calJ^{N'}$, it is guaranteed that
\begin{equation}\label{eqn:012903}
\left\|\bm{x}^*(\calI^{N'}, \calJ^{N'})-\mathbb{E}[\bar{\bm{x}}]\right\|_2\leq \veps/(2d'),
\end{equation}
We further note that
\[\begin{aligned}
&\left\|A^\top\bx^*(\calI^{N'}, \calJ^{N'})-A^\top\mathbb{E}[\bar{\bx}]\right\|_{\infty}=\max_{j\in[m_2]}\left\{ \left|A_{:, j}^\top(\bx^*(\calI^{N'}, \calJ^{N'})-\mathbb{E}[\bar{\bx}])\right| \right\}\\
&= \max_{j\in[m_2]}\left\{ \left|A_{\mathcal{I}^{N'}, j}^\top(\bx_{\mathcal{I}^{N'}}^*(\calI^{N'}, \calJ^{N'})-\mathbb{E}[\bar{\bx}_{\mathcal{I}^{N'}}])\right| \right\} \leq d'\cdot\|\bx^*(\calI^{N'}, \calJ^{N'})-\mathbb{E}[\bar{\bx}]\|_2.
\end{aligned}\]
Therefore, we know that $\mathbb{E}[\bar{\bx}]$ is a feasible solution to $\VI$ with an optimality gap bounded by 
\begin{equation}\label{eqn:012905new}
\Rad(N'/m, \eps/m) + 4(d')^2\wh{\kappa}\cdot\Rad(N'/m, \eps/m) + \veps/2.
\end{equation}
It only remains to prove that the final value of $N'$ guarantees that the first part of \Cref{eqn:012905new} satisfies that
\[
\Rad(N'/m, \eps/m) + 4(d')^2\wh{\kappa}\cdot\Rad(N'/m, \eps/m)=\order(\veps),
\]
with $\eps=\veps/(8K^2)$ with $K$ denotes the number of times that \Cref{alg:Idenbasis} is called in the iteration in line 3 and line 4 of \Cref{alg:Twophase}.

We now bound the value of the parameter $\wh{\kappa}$. Consider the last iteration of calling \Cref{alg:Idenbasis} in the execution of \Cref{alg:Twophase}. We define the event $\mathcal{E}$ by
\[
\mathcal{E}=\left\{\left| \wh{A}_{i,j}-A_{i,j}\right|\leq \Rad(N'/m, \eps/m), ~~\forall i\in[m_1], j\in[m_2]  \right\}.
\]
From standard Hoeffding inequality, we know that
\[
\Pr[\mathcal{E}]\geq1-\eps.
\]
We also denote by 
\[
\Delta A = \wh{A}_{\mathcal{I}^{N'}, \mathcal{J}^{N'}} - A_{\mathcal{I}^{N'}, \mathcal{J}^{N'}}.
\]
Then following Theorem 1 of \cite{stewart1998perturbation}, we know that the difference of the corresponding singular values of $
\begin{bmatrix}
A^\top_{\calI^{N'}, \calJ^{N'}} & -\bm{e}^{|\calJ^{N'}|} \\
(\bm{e}^{|\calI^{N'}|})^\top & 0
\end{bmatrix}$ and $
\begin{bmatrix}
\wh{A}^\top_{\calI^{N'}, \calJ^{N'}} & -\bm{e}^{|\calJ^{N'}|} \\
(\bm{e}^{|\calI^{N'}|})^\top & 0
\end{bmatrix}$ (sorted from the largest to the smallest) is upper bounded by $\|\Delta A\|_2$, which is again upper bounded by $d'\cdot\Rad(N'/m, \eps/m)$ with a probability at least $1-\eps$. Formally, it holds that
\[
    \Pr\left[ \left| \sigma_1-\sigma_1' \right| \geq d'\cdot\Rad(N'/m, \eps/m) \right] \leq \eps,
\]
for any two corresponding singular values of $
\begin{bmatrix}
A^\top_{\calI^{N'}, \calJ^{N'}} & -\bm{e}^{|\calJ^{N'}|} \\
(\bm{e}^{|\calI^{N'}|})^\top & 0
\end{bmatrix}$ and $
\begin{bmatrix}
\wh{A}^\top_{\calI^{N'}, \calJ^{N'}} & -\bm{e}^{|\calJ^{N'}|} \\
(\bm{e}^{|\calI^{N'}|})^\top & 0
\end{bmatrix}$, denoted by $\sigma_1$ and $\sigma_1'$. 

As a result, when 
\[
d'\cdot\Rad(N'/m, \eps/m) \leq \wh{\sigma}_{\min}/2,
\]
where we denote by $\wh{\sigma}_{\min}>0$ the smallest singular value of $
\begin{bmatrix}
\wh{A}^\top_{\calI^{N'}, \calJ^{N'}} & -\bm{e}^{|\calJ^{N'}|} \\
(\bm{e}^{|\calI^{N'}|})^\top & 0
\end{bmatrix}$, we know that conditional on the event $\mathcal{E}$ happens (which happens with probability at least $1-\eps$), we have 
\[
\frac{\sigma_1'}{2}\leq \sigma_1\leq \frac{3\sigma_1'}{2}
\]
for any two corresponding singular values, $\sigma_1$ and $\sigma_1'$, of $
\begin{bmatrix}
A^\top_{\calI^{N'}, \calJ^{N'}} & -\bm{e}^{|\calJ^{N'}|} \\
(\bm{e}^{|\calI^{N'}|})^\top & 0
\end{bmatrix}$ and $
\begin{bmatrix}
\wh{A}^\top_{\calI^{N'}, \calJ^{N'}} & -\bm{e}^{|\calJ^{N'}|} \\
(\bm{e}^{|\calI^{N'}|})^\top & 0
\end{bmatrix}$.
Therefore, denoting by $\kappa'$ the conditional number of $
\begin{bmatrix}
A^\top_{\calI^{N'}, \calJ^{N'}} & -\bm{e}^{|\calJ^{N'}|} \\
(\bm{e}^{|\calI^{N'}|})^\top & 0
\end{bmatrix}$, we know that when
\begin{equation}\label{eqn:061801}
d'\cdot\Rad(N'/m, \eps/m) \leq  \wh{\sigma}_{\min}/2,
\end{equation}
it holds that $\sigma'>0$ with $\sigma'$ denotes the smallest singular value of $
\begin{bmatrix}
A^\top_{\calI^{N'}, \calJ^{N'}} & -\bm{e}^{|\calJ^{N'}|} \\
(\bm{e}^{|\calI^{N'}|})^\top & 0
\end{bmatrix}$, and
\[
    \Pr\left[  \frac{\wh{\kappa}}{4}\leq\kappa'\leq \frac{9\wh{\kappa}}{4}  \right] \geq 1 - \eps.
\]
Note that the condition in \Cref{eqn:061801} is automatically guaranteed by line 14 of \Cref{alg:Idenbasis}. Therefore, we conclude that condition on the event $\mathcal{E}$ happens, we have that $\sigma'>0$ and 
\[
\frac{\wh{\kappa}}{4}\leq\kappa'\leq \frac{9\wh{\kappa}}{4}.
\]
Plugging in the bound of $\wh{\kappa}$ using $\kappa'$, we know that it remains to prove that 
\[
\Rad(N'/m, \eps/m) + 16(d')^2\kappa'\cdot\Rad(N'/m, \eps/m)= \order(\veps),
\]
which is further implied by
\begin{equation}\label{eqn:090301}
\Rad(N'/m, \eps/m) + \frac{16(d')^3}{\sigma'}\cdot\Rad(N'/m, \eps/m)= \order(\veps),
\end{equation}
since we upper bound $\kappa'$ by $d'/\sigma'$.

We then bound $\eps=\veps/(8K^2)$ by upper bounding $K$ which denotes the number of times that \Cref{alg:Idenbasis} is called in the iteration in line 3 and line 4 of \Cref{alg:Twophase}. Note that every time calling \Cref{alg:Idenbasis}, we exponentially increase the sample size, which implies that $K=\lceil\log_2(N_2)\rceil$. The bound on $N_2$ is presented in \Cref{prop:DoublingTrick}. Therefore, we know that 
\begin{equation}\label{eqn:090302}
\begin{aligned}
&\Rad(N'/m, \eps/m) + \frac{16(d')^3}{\sigma'}\cdot\Rad(N'/m, \eps/m)= \left(1+\frac{16(d')^3}{\sigma'}\right)\cdot\sqrt{\frac{m\cdot\log(16mK^2/\veps)}{2N'}}\\
\leq& \left(1+\frac{16(d')^3}{\sigma'}\right)\cdot \sqrt{\frac{m}{2N'}\cdot\left(\log(16m/\veps)+2\log\log(m/\sigma_0^2)\right)}.
\end{aligned}
\end{equation}
Now since we have that $N'\geq N_1 = \frac{\alpha\cdot m}{\veps^2}$, where $\alpha$ is a pre-given parameter, it holds that
\begin{equation}\label{eqn:090303}
\Rad(N'/m, \eps/m) + \frac{16(d')^3}{\sigma'}\cdot\Rad(N'/m, \eps/m)\leq \veps\cdot \left(1+\frac{16(d')^3}{\sigma'}\right)\cdot \sqrt{\frac{1}{2\alpha}\cdot\left(\log(16m/\veps)+2\log\log(m/\sigma_0^2)\right)}.
\end{equation}
Combining \Cref{eqn:090303} with \Cref{eqn:012905new}, we know that $\mathbb{E}[\bar{\bx}]$ is a feasible solution to $\VI$ with an optimality gap bounded by 
\[
\veps\cdot \left(\frac{1}{2}+\left(1+\frac{16(d')^3}{\sigma'}\right)\cdot \sqrt{\frac{1}{2\alpha}\cdot\left(\log(16m/\veps)+2\log\log(m/\sigma_0^2)\right)}\right).
\]
For the total number of samples, note that we apply the bound in \Cref{prop:DoublingTrick} to bound $N_2$, i.e., 
\[
N_2 = \frac{\alpha\cdot m}{\veps^2} + \order\left(\frac{m}{\sigma_0^2}\cdot\log(m/\veps)\right)
\]
Further combined with the expression of $N$ in line 6 of \Cref{alg:Twophase}, we conclude that 
\[
N = \order\left( \frac{\alpha\cdot m}{\veps^2} + \frac{m}{\sigma_0^2}\cdot\log(m/\veps)+ \frac{d^{15/2}}{(\sigma')^3}\cdot\frac{\log(m/\veps)}{\veps}\right).
\]
We further upper bound the algorithm-dependent quantity $d$ by the algorithm-independent quantity $d_0$, and lower bound the algorithm-dependent quantity $\sigma'$ by the algorithm-independent quantity $\sigma_0$. We have that the total sample complexity can be upper bounded as
\[
N = \order\left( \frac{\alpha\cdot m}{\veps^2} + \frac{m}{\sigma_0^2}\cdot\log(m/\veps)+ \frac{d_0^{15/2}}{\sigma_0^3}\cdot\frac{\log(m/\veps)}{\veps}\right).
\]
Our proof is thus completed.

\section{Missing Proofs in \Cref{sec:Combine}}

\subsection{Proof of \Cref{thm:BestBoth}}
We first prove the $\delta$-independent bound. Note that from \Cref{thm:sample-complexity-independent}, we know that $\mathbb{E}[\bar{\bx}]$ forms a feasible solution to $\VI$ with an sub-optimality gap bounded by
\[
\veps\cdot \left(\frac{1}{2}+\left(1+\frac{16(d')^3}{\sigma'}\right)\cdot \sqrt{\frac{1}{2\alpha}\cdot\left(\log(16m/\veps)+2\log\log(m/\sigma_0^2)\right)}\right),
\]
with a sample complexity upper bounded by
 \begin{align}\label{eqn:sample_complexity_delta_ind1}
N = \order\left( \frac{\alpha\cdot m}{\veps^2} + \frac{m}{\sigma_0^2}\cdot\log(m/\veps)+ \frac{(d')^{15/2}}{(\sigma')^3}\cdot\frac{\log(m/\veps)}{\veps}\right).
\end{align}
Suppose the parameter $\alpha$ is appropriately selected such that 
\[
\left(1+\frac{16(d')^3}{\sigma'}\right)\cdot \sqrt{\frac{1}{2\alpha}\cdot\left(\log(16m/\veps)+2\log\log(m/\sigma_0^2)\right)}\leq\frac{1}{2},
\]
we know that the suboptimality gap becomes $\veps$ and the sample complexity bound becomes 
\[
N = \order\left(\frac{d_0^{15/2}}{\sigma_0^3}\cdot\frac{\log(m/\veps)}{\veps} + \frac{d_0^6\cdot m}{\sigma_0^2\cdot\veps^2}\right).
\]
We then prove the $\delta$-dependent bound. Conditional on $\wh{\calI}^*=\calI^*$ and $\wh{\calJ}^*=\calJ^*$, we have
\[
\|A^\top\bx^*-A^\top\mathbb{E}[\bar{\bx}]\|_{\infty}=\max_{j\in[m_2]}\left\{ \left|A_{:, j}^\top(\bx^*-\mathbb{E}[\bar{\bx}])\right| \right\}= \max_{j\in[m_2]}\left\{ \left|A_{\mathcal{I}^*, j}^\top(\bx_{\mathcal{I}^*}^*-\mathbb{E}[\bar{\bx}_{\mathcal{I}^*}])\right| \right\} \leq d^*\cdot\|\bx^*-\mathbb{E}[\bar{\bx}]\|_2.
\]
Then, we know that $\mathbb{E}[\bar{\bx}]$ is a feasible solution to $\VI$ with an optimality gap bounded by (follows from \Cref{eqn:011512}) 
\begin{equation}\label{eqn:0129071}
\order\left((d^*)^{5/2}\kappa\cdot\left( \frac{(d^*)^5}{\sigma^2}+
1\right)\cdot\frac{\log(N-N_2)}{N-N_2}\right).
\end{equation}
Therefore, we set the quantity in \Cref{eqn:0129071} to be $\veps$ and toghther with the bound on $N_2$ presented in \Cref{prop:DoublingTrick}, we obtain the sample complexity bound
\[
N=\order\left(\left(\frac{m}{\sigma_0^2}+ \frac{m}{\delta^2}\right)\cdot\log(m/\veps) + \frac{(d^*)^{15/2}}{\sigma^3}\cdot\frac{\log(m/\veps)}{\veps} \right).
\]
Our proof is thus completed.

\subsection{Proof of \Cref{coro:BestBoth}}\label{app:thm10}
We classify into the following two scenarios:\\
i) The $\delta$-dependent bound is smaller. We consider the case where $\delta$ is large enough such that the $\delta$-dependent bound is smaller than the $\delta$-independent bound, i.e., 
\[
\order\left( \frac{d_0^{15/2}}{\sigma_0^3}\cdot\frac{\log(m/\veps)}{\veps} + \frac{d_0^6\cdot m}{\sigma_0^2\cdot\veps^2} \right)\geq \order\left( \left(\frac{m}{\sigma_0^2}+ \frac{m}{\delta^2}\right)\cdot\log(m/\veps) + \frac{(d^*)^{15/2}}{\sigma^3}\cdot\frac{\log(m/\veps)}{\veps} \right).
\]
In this case, as long as the input sample size $N_1$ satisfies the condition that $N_1=\Theta\left( \frac{m\cdot\log(m/\veps)}{\delta^2} \right)$, the $\delta$-dependent guarantee is achievable (implied by \Cref{alg:Idenbasis} finds the optimal basis with a high probability following \Cref{prop:DoublingTrick}). Then, in this case, following \Cref{thm:sampleComplexity} and \Cref{thm:BestBoth}, we know that a $\veps$-suboptimality gap NE can be obtained by our \Cref{alg:Idenbasis} and \Cref{alg:Twophase} with a sample complexity bound of 
\[\begin{aligned}
&\order\left( \left(\frac{m}{\sigma_0^2}+ \frac{m}{\delta^2}\right)\cdot\log(m/\veps) + \frac{(d^*)^{15/2}}{\sigma^3}\cdot\frac{\log(m/\veps)}{\veps} \right)\\
= &\order\left(\min\left\{ \frac{d_0^{15/2}}{\sigma_0^3}\cdot\frac{\log(m/\veps)}{\veps} + \frac{d_0^6\cdot m}{\sigma_0^2\cdot\veps^2},  \left(\frac{m}{\sigma_0^2}+ \frac{m}{\delta^2}\right)\cdot\log(m/\veps) + \frac{(d^*)^{15/2}}{\sigma^3}\cdot\frac{\log(m/\veps)}{\veps} \right\}\right)
\end{aligned}\]
ii) The $\delta$-independent bound is smaller. We consider the case where $\delta$ is small enough such that the $\delta$-independent bound is smaller than the $\delta$-dependent bound, i.e.,
\[
\order\left( \frac{d_0^{15/2}}{\sigma_0^3}\cdot\frac{\log(m/\veps)}{\veps} + \frac{d_0^6\cdot m}{\sigma_0^2\cdot\veps^2}\right)\leq \order\left( \left(\frac{m}{\sigma_0^2}+ \frac{m}{\delta^2}\right)\cdot\log(m/\veps) + \frac{(d^*)^{15/2}}{\sigma^3}\cdot\frac{\log(m/\veps)}{\veps} \right)
\]
In this case, we can set the input sample size $N_1=\frac{\alpha\cdot m}{\eps^2}$ to \Cref{alg:Idenbasis}. Following \Cref{thm:sample-complexity-independent}, we know that as long as we select the parameter $\alpha$ such that 
\[
\left(1+\frac{16d_0^3}{\sigma_0}\right)\cdot \sqrt{\frac{1}{2\alpha}\cdot\left(\log(16m/\veps)+2\log\log(m/\sigma_0^2)\right)}\leq\frac{1}{2},
\]
we have that $\mathbb{E}[\bar{\bx}]$ forms a feasible solution to $\VI$ with an sub-optimality gap bounded by $\veps$. Therefore, from \Cref{thm:sample-complexity-independent} and \Cref{thm:BestBoth}, we know that a $\veps$-suboptimality gap NE can be obtained by our \Cref{alg:Idenbasis} and \Cref{alg:Twophase} with a sample complexity bound of 
\begin{equation}\label{eqn:061802}
\begin{aligned}
&\order\left( \frac{d_0^{15/2}}{\sigma_0^3}\cdot\frac{\log(m/\veps)}{\veps} + \frac{d_0^6\cdot m}{\sigma_0^2\cdot\veps^2} \right)\\
=&\order\left(\min\left\{ \frac{d_0^{15/2}}{\sigma_0^3}\cdot\frac{\log(m/\veps)}{\veps} + \frac{d_0^6\cdot m}{\sigma_0^2\cdot\veps^2},  \left(\frac{m}{\sigma_0^2}+ \frac{m}{\delta^2}\right)\cdot\log(m/\veps) + \frac{(d^*)^{15/2}}{\sigma^3}\cdot\frac{\log(m/\veps)}{\veps} \right\}\right)
\end{aligned}
\end{equation}
Therefore, we conclude that there exists a choice of $N_1$ to serve as input such that \Cref{alg:Twophase} guarantees that $\mathbb{E}[\bar{\bx}]$ is a $\veps$-sub-optimality-gap with $N$ bounded by
\[
N = \order\left(\min\left\{\frac{d_0^{15/2}}{\sigma_0^3}\cdot\frac{\log(m/\veps)}{\veps} + \frac{d_0^6\cdot m}{\sigma_0^2\cdot\veps^2}, \left(\frac{m}{\sigma_0^2}+ \frac{m}{\delta^2}\right)\cdot\log(m/\veps) + \frac{(d^*)^{15/2}}{\sigma^3}\cdot\frac{\log(m/\veps)}{\veps}\right\}\right).
\]
Our proof is thus completed.

\subsection{Proof of \Cref{thm:EstimateDelta}}
Note that following standard Hoeffding inequality, for any $\eps>0$, we know that the gap between each element of the matrix $A$ and $\wh{A}$ is upper bounded by $\Rad(n/m, \eps/(2n^2m))$ with a probability at least $1-\eps/(2n^2)$, where we denote this event as $\mathcal{E}_n$. We now denote by 
\[
\mathcal{E}=\cap_{n=1}^{\infty}\mathcal{E}_n
\]
and from the union bound, we have that
\[
\Pr[\mathcal{E}] \geq 1-\sum_{n=1}^{\infty}\Pr[\mathcal{E}_n^c]\geq 1-\sum_{n=1}^{\infty}\eps/(2n^2)\geq 1-\eps.
\]
We now condition on the event $\mathcal{E}$ happens.
Then, from \Cref{claim:Bound1} and \Cref{claim:Bound2}, we know that the value between $\VI_{\mathcal{I}}$ and $\wh{V}^{\mathrm{Prime}}$, as well as $\VD_{\mathcal{I}, \mathcal{J}}$ and $\wh{V}_{\mathcal{I}, \mathcal{J}}$, for any $\mathcal{I}\subset[m_1]$ and $\mathcal{J}\subset[m_2]$, is upper bounded by $\Rad(n/m,\eps/(2n^2m))$. This result further implies that the following inequalities
\[
|\delta_1-\wh{\delta}_1|\leq 2\cdot\Rad(n/m,\eps/(2n^2m)) \text{~~and~~} |\delta_2-\wh{\delta}_2|\leq 2\cdot\Rad(n/m,\eps/(2n^2m)).
\]
Recall the definition that 
\[
\delta=\min\{\delta_1, \delta_2\} \text{~~and~~}\wh{\delta}=\min\{\wh{\delta}_1, \wh{\delta}_2\}. 
\]
Without loss of generality, we let $\delta=\delta_1$, then, conditional on the event $\mathcal{E}$ happens, we know that
\[
\wh{\delta}\leq \wh{\delta}_1 \leq \delta_1+2\cdot\Rad(n/m,\eps/(2n^2m))=\delta+ 2\cdot\Rad(n/m,\eps/(2n^2m)). 
\]
Moreover, note that 
\[
\wh{\delta}_1\geq\delta_1-2\cdot\Rad(n/m,\eps/(2n^2m))\geq\delta-2\cdot\Rad(n/m,\eps/(2n^2m))
\]
and
\[
\wh{\delta}_2\geq\delta_2-2\cdot\Rad(n/m,\eps/(2n^2m))\geq\delta-2\cdot\Rad(n/m,\eps/(2n^2m)).
\]
We have that
\[
\wh{\delta}\geq\delta-2\cdot\Rad(n/m,\eps/(2n^2m))
\]
conditional on the event $\mathcal{E}$ happens.
Thus, it holds that
\[
\Pr\left[|\delta-\wh{\delta}|\geq 2\cdot\Rad(n/m,\eps/(2n^2m))\right] \leq \eps.
\]
Therefore, as long as 
\[
\wh{\delta}\geq4\cdot\Rad(n/m, \eps/(2n^2m)),
\]
we know that 
\[\frac{\wh{\delta}}{2}\leq\delta\leq2\wh{\delta}
\]
with a probability at least $1-\eps$. On the other hand, in order to bound $N_3$, we know that with a probability at least $1-\eps$, it holds that
\[
\delta\leq \wh{\delta}+2\cdot\Rad(N_3/m, \eps/(2N_3^2m)).
\]
The definition of $N_3$ implies that $\wh{\delta}= 4\cdot\Rad(N_3/m, \eps/(2n_3^2m))$, 
which implies that
\[
\delta\leq 6\cdot\Rad(N_3/m, \eps/(2N_3^2m))
\]
and thus
\[
N_3\leq \order\left(\frac{m}{\delta^2}\cdot\log(2m/\eps)\right).
\]
Our proof is thus completed.

\subsection{Proof of \Cref{thm:BoundN1}}
Note that following standard Hoeffding inequality, for any $\eps>0$, we know that the absolute value of the matrix $\Delta A= A_{\mathcal{I}, \mathcal{J}} - \wh{A}_{\mathcal{I}, \mathcal{J}}$ is upper bounded by $\Rad(n/(d')^2, \eps/(2n^2(d')^2))$ with a probability at least $1-\eps/(2n^2)$. Then following Theorem 1 of \cite{stewart1998perturbation}, we know that the difference of the corresponding singular values of $\begin{bmatrix}
A^\top_{\mathcal{I}, \mathcal{J}}, & -\bm{e}^{|\mathcal{J}|} \\
 (\bm{e}^{|\mathcal{I}|})^\top & 0
\end{bmatrix}$ and $\begin{bmatrix}
\wh{A}^\top_{\mathcal{I}, \mathcal{J}}, & -\bm{e}^{|\mathcal{J}|} \\
 (\bm{e}^{|\mathcal{I}|})^\top & 0
\end{bmatrix}$ (sorted from the largest to the smallest) is upper bounded by $\|\Delta A\|_2$, which is again upper bounded by $d\cdot\Rad(n/(d')^2, \eps/(2n^2(d')^2))$, with a probability at least $1-\eps/(2n^2)$. We denote this event by $\mathcal{E}_n$. Formally, it holds that
\[
    \Pr[\mathcal{E}_n]=\Pr\left[ \left| \sigma_{\calI, \calJ}-\wh{\sigma} \right| \geq d'\cdot\Rad(n/(d')^2, \eps/(2n^2(d')^2)) \right] \leq \eps/(2n^2).
\]
We further denote by
\[
\mathcal{E} = \cap_{n=1}^{\infty}\mathcal{E}_n
\]
and from the union bound, it holds that
\[
\Pr[\mathcal{E}] \geq  1-\sum_{n=1}^{\infty}\Pr[\mathcal{E}_n^c]\geq 1-\sum_{n=1}^{\infty}\eps/(2n^2)\geq 1-\eps.
\]
We now condition on the event $\mathcal{E}$ happens.
Then, as long as 
\[
\wh{\sigma}\geq 2d'\cdot\Rad(n/(d')^2, \eps/(2n^2(d')^2)),
\]
we know that 
\[
\frac{\wh{\sigma}}{2}\leq\sigma_{\calI, \calJ}\leq2\wh{\sigma}
\]
with a probability at least $1-\eps$. On the other hand, in order to bound $N_4$, we know that with a probability at least $1-\eps$, we have
\[
\sigma_{\calI, \calJ}\leq \wh{\sigma}+d'\cdot\Rad(N_4/(d')^2,\eps/(2N_4^2(d')^2)),
\]
which implies that 
\[
N_4\leq \order\left(\frac{(d')^4}{\sigma_{\calI, \calJ}^2}\cdot\log(2(d')^2/\eps)\right).
\]
Our proof is thus completed.

\subsection{Oracle Algorithm for Solving the Problem in \Cref{eqn:012701}}\label{sec:oracle}
We consider the general LP of the form:
\[
V = \min~ \bm{c}_0^\top\bx~~~\mbox{s.t.}~A_0\bx\preceq \bm{b}_0, ~\bm{0}\preceq \bx\preceq \bm{1}, 
\]
where the decision variable $\bx\in\mathbb{R}^{m_0}$. For any set $\calI\subset[m_0]$, we also denote by 
\[
V_{\calI} = \min~ \bm{c}_0^\top\bx~~~\mbox{s.t.}~A_0\bx\preceq \bm{b}_0, ~\bx_{\calI}=\bm{0}, ~\bm{0}\preceq \bx\preceq \bm{1}. 
\]
We now provide one potential oracle algorithm that can solve the following problem
\begin{equation}\label{eqn:071001}
  V_{\delta} =  \min_{\mathcal{I}\subset[m_0]} \{ V_{\mathcal{I}} - V: V_{\mathcal{I}} - V > 0  \}
\end{equation}
as long as we are given an $\eps>0$ such that $\eps< V_{\delta}$. In practice, the value $\eps$ can be set as the desired accuracy level for computing $V_{\delta}$ or just be a sufficiently small value. To be specific, we can solve the following Mixed Integer Programming (MIP) problem to determine the value of $V_{\delta}$. 

\begin{equation}\label{eqn:MIP}
\begin{aligned}
V_{\delta}= \min~~&\bm{c}_0^\top\bx -V \\
\mbox{s.t.}~~ &\bm{c}_0^\top\bx-V\geq \eps\\
&A_0\bx\preceq \bm{b}_0\\
& \bm{z}\succeq\bm{x} \\
& \sum_{i=1}^{m_0} z_i \;\le\; m_0-1,\\
&\bm{z}\in\{0,1\}^{m_0}, \bm{x}\succeq\mathbf{0}.
\end{aligned}
\end{equation}
The above MIP problem can be solved by calling the modern LP solver.

\end{APPENDICES}

\end{document}